%% file: main.tex
\documentclass{article}

\input{packages}
\input{commands}

\title{Analysis of Primal-Dual Langevin Algorithms}
\date{To appear in \emph{Foundations of Computational Mathematics}}
\author[1,2]{Martin Burger}
\author[3]{Matthias J. Ehrhardt}
\author[2,$\dagger$]{Lorenz Kuger}
\author[2]{Lukas Weigand}
\affil[1]{\small Helmholtz Imaging, Deutsches Elektronen-Synchrotron DESY, Hamburg, Germany.}
\affil[2]{\small Fachbereich Mathematik, Universität Hamburg, Germany.}
\affil[3]{\small Department of Mathematical Sciences, University of Bath, United Kingdom.}
\hypersetup{pdftitle={Analysis of Primal-Dual Langevin Algorithms}, pdfauthor={Martin Burger, Matthias J. Ehrhardt, Lorenz Kuger, Lukas Weigand}}

\begin{document}

\maketitle
{\renewcommand{\thefootnote}{\fnsymbol{footnote}}%
\footnotetext[2]{Corresponding author: \href{mailto:mail@lorenzkuger.de}{mail@lorenzkuger.de}. Work carried out while at Universität Hamburg.}}

\begin{abstract}
    We analyze a recently proposed class of algorithms for the problem of sampling from probability distributions $\mu^\ast$ in $\mathbb{R}^d$ with a Lebesgue density of the form $\mu^\ast(x) \propto \exp(-f(Kx)-g(x))$, where $K$ is a linear operator and $f,g$ convex and non-smooth. The method is a generalization of the primal-dual hybrid gradient optimization algorithm to a sampling scheme. We give the iteration's continuous time limit, a stochastic differential equation in the joint primal-dual variable, and its mean field limit Fokker--Planck equation. Under mild conditions, the scheme converges to a unique stationary state in continuous and discrete time. Contrary to purely primal overdamped Langevin diffusion, the stationary state in continuous time does not have $\mu^\ast$ as its primal marginal. Thus, further analysis is carried out to bound the bias induced by the partial dualization, and potentially correct for it in the diffusion. Time discretizations of the diffusion lead to implementable algorithms, but, as is typical in Langevin Monte Carlo methods, introduce further bias. We prove bounds for these discretization errors, which allow to give convergence results relating the produced samples to the target. We demonstrate our findings numerically first on small-scale examples in which we can exactly verify the theoretical results, and subsequently on typical examples of larger scale from Bayesian imaging inverse problems.
\end{abstract}



Communicated by Tony Lelièvre

\input{sections/1intro}
\input{sections/2problemsetting}
\input{sections/3continuoustime}

\input{sections/4discretetime}
\input{sections/5numerics}
\input{sections/6conclusion}

\section*{Acknowledgements}
We thank Jan Modersitzki (Lübeck) for fruitful discussions and predicting non-concentration of the dual variable. We thank Laura Kanzler for pointing out an error in an earlier version of Section 3.

MB, LK and LW acknowledge support from DESY (Hamburg, Germany), a member of the Helmholtz Association HGF. 
This work was supported through the Maxwell computational resources operated at Deutsches Elektronen-Synchrotron DESY, Hamburg, Germany.
MB, LK and LW acknowledge funding by the German Federal Ministry of Research, Technology and Space (BMFTR) under grant agreement Nos. 16IS24072B and 16IS24072D (COMFORT).
MB and LW acknowledge support from the German Research Foundation, projects BU 2327/19-2 and BU 2327/20-1.
MJE acknowledges support from the EPSRC (EP/V026259/1, EP/Y037286/1), STFC through the UKRI Digital Research Infrastructure programme and the European Union Horizon 2020 research and innovation programme under the Marie Skłodowska-Curie grant agreement REMODEL.

\printbibliography

\appendix
\crefalias{section}{appendix}
\crefalias{subsection}{appendix}
\input{sections/appendix}
\end{document}

%% file: packages.tex
\usepackage[T1]{fontenc}
\usepackage[utf8]{inputenc}

\usepackage[margin=3cm]{geometry}
\usepackage{graphicx}
\usepackage{authblk}

\usepackage{amsmath,amssymb,amsthm}
\usepackage{thmtools}
\usepackage{mathtools}

\usepackage[
natbib,
maxcitenames=3, 
maxbibnames=10,  
sorting=nyt,
url=false,
isbn=false,
sortcites,
defernumbers,
backref,
backend=biber
]{biblatex}
\AtBeginBibliography{\emergencystretch=2em}

\makeatletter
\let\blx@rerun@biber\relax
\makeatother

\usepackage[hypertexnames=false]{hyperref}
\hypersetup{colorlinks=true, linkcolor=blue!50!black, citecolor=green!40!black, urlcolor=blue!60!black}
\usepackage{subcaption}
\usepackage{url}

\usepackage[nameinlink, capitalise, noabbrev]{cleveref}
\crefname{assumption}{Assumption}{Assumptions}
\crefname{proposition}{Proposition}{Propositions}
\crefname{corollary}{Corollary}{Corollaries}
\crefname{lemma}{Lemma}{Lemmata}
\crefname{definition}{Definition}{Definitions}
\crefname{theorem}{Theorem}{Theorems}
\crefname{remark}{Remark}{Remarks}
\crefname{algocf}{Algorithm}{Algorithms}

\crefname{appendix}{Appendix}{Appendices}

\usepackage[ruled]{algorithm2e}
\usepackage{enumitem}

\usepackage{tikz}
\usetikzlibrary{matrix, positioning, fit, arrows.meta, calc}

%% file: commands.tex
\DeclareMathOperator{\prox}{prox}

\DeclareMathOperator{\id}{id}
\DeclareMathOperator{\sign}{sign}

\DeclareMathOperator*{\argmin}{arg\,min}

\DeclareMathOperator{\TV}{TV}
\DeclareMathOperator{\TGV}{TGV}
\DeclareMathOperator{\Law}{Law}

\DeclareMathOperator{\spanop}{span}
\DeclareMathOperator{\tr}{tr}
\DeclareMathOperator{\Var}{Var}

\renewcommand{\epsilon}{\varepsilon}
\renewcommand{\phi}{\varphi}

\newcommand{\norm}[1]{\lVert #1  \rVert}

\newcommand{\inpro}[2]{\langle #1,#2 \rangle}
\newcommand{\abs}[1]{\left\vert #1 \right\vert}

\newcommand{\bbN}{\mathbb{N}}
\newcommand{\bbR}{\mathbb{R}}
\newcommand{\bbE}{\mathbb{E}}

\newcommand{\calA}{\mathcal{A}}

\newcommand{\calF}{\mathcal{F}}

\newcommand{\calL}{\mathcal{L}}
\newcommand{\calM}{\mathcal{M}}

\newcommand{\calO}{\mathcal{O}}
\newcommand{\calP}{\mathcal{P}}
\newcommand{\calT}{\mathcal{T}}
\newcommand{\calU}{\mathcal{U}}
\newcommand{\calV}{\mathcal{V}}
\newcommand{\calW}{\mathcal{W}}
\newcommand{\calX}{\mathcal{X}}
\newcommand{\calY}{\mathcal{Y}}
\newcommand{\calZ}{\mathcal{Z}}

\newcommand{\rmd}{\,\mathrm{d}}

\newcommand{\tup}[1]{\text{\textup{#1}}}

\newtheorem{theorem}{Theorem}[section]
\newtheorem{proposition}[theorem]{Proposition}
\newtheorem{lemma}[theorem]{Lemma}
\newtheorem{corollary}[theorem]{Corollary}
\newtheorem{remark}[theorem]{Remark}
\newtheorem{example}[theorem]{Example}
\newtheorem{definition}[theorem]{Definition}

\numberwithin{equation}{section}

%% file: sections/1intro.tex
\section{Introduction}
We analyze a recently proposed class of algorithms for sampling from a probability distribution on $\calX \coloneqq \bbR^d$ with log-concave, non-smooth density.
We are interested in target distributions with Lebesgue densities of the form
\begin{equation}\label{eq:target}
    \mu^\ast(x) = \frac{\exp(-f(Kx)-g(x))}{\int \exp(-f(Kw)-g(w))\rmd w},
\end{equation}
where we assume that the integral in the denominator exists and is finite. The operator $K : \calX \to \bbR^m \eqqcolon \calY $ is assumed to be linear and $f:\calY \to \bbR \cup \{\infty\}$, $g:\calX \to \bbR\cup \{\infty\}$ will be called potential terms and are both assumed convex and lower semi-continuous but not necessarily differentiable. We will denote the full potential by $ f \circ K + g \eqqcolon h : \calX \to \bbR \cup \{\infty\}$.

A standard class of algorithms to draw samples in Bayesian inference are Markov chain Monte Carlo (MCMC) methods. 
Compared to other types of MCMC methods, sampling schemes based on Langevin diffusion processes have proved efficient in inverse problems and imaging applications due to the preferable scaling behaviour to high dimensions \cite{Roberts1996,Durmus2022,Durmus2019}. 
Typically, MCMC sampling algorithms based on Langevin diffusion perform a time discretisation of the stochastic differential equation (SDE)
\begin{equation}\label{eq:overdamped-langevin}
    \mathrm d X_t = \nabla \log \mu^\ast(X_t)\rmd t + \sqrt{2}\rmd W_t,
\end{equation}
where $W_t$ is standard $\bbR^d$-valued Brownian motion. Assuming $\mu^\ast$ satisfies a log-Sobolev inequality, it is the unique invariant distribution of the SDE and the law of $X_t$ will converge to $\mu^\ast$ as $t$ grows.
The standard Euler--Maruyama discretization of \eqref{eq:overdamped-langevin} takes the form
\begin{equation}\label{eq:ula}
    \left\{ \begin{array}{l}
    \xi^{n+1} \sim \mathrm{N}(0,I_d)\\
    X^{n+1} = X^n + \tau \nabla \log \mu^\ast(X^n) + \sqrt{2\tau} \xi^{n+1},
    \end{array}\right.
\end{equation}
this iteration is also known as unadjusted Langevin algorithm (ULA). ULA and other time discretizations of \eqref{eq:overdamped-langevin} introduce a bias between the law of the samples and the target $\mu^\ast$, hence many works have been concerned with the characterization of the invariant distribution of the Markov chain $(X^n)_n$ and its bias \cite{Roberts1996,Durmus2017,Wibisono2018,Dalalyan2019}. 
Other typical questions concern preconditioning and convergence speed of the Markov chain and the behaviour of Metropolis--Hastings correction steps to overcome the bias and draw unbiased samples from $\mu^\ast$.

In this work, we are concerned with a new type of Langevin sampling algorithm which has been proposed as \textit{Unadjusted Langevin Primal-Dual Algorithm} (ULPDA) in \cite{Narnhofer2024} and has since been mentioned in \cite{Lau2024,Habring2024}.
Unlike ULA or its variants, ULPDA can not be interpreted as a time discretization of \eqref{eq:overdamped-langevin}. 
It rather consists in steps of a primal-dual optimization method, where the same stochastic term $\sqrt{2\tau}\xi$ as in ULA is added to the primal iterate. The iteration takes the form
\begin{equation}\label{eq:pd_sampling_update}
    \left\{ \begin{array}{l}
        \xi^{n+1} \sim \mathrm{N}(0,I_d)\\
        Y^{n+1} = \prox_{\sigma f^\ast}\left(Y^n+\sigma K X_\theta^n \right)\\
        X^{n+1} = \prox_{\tau g}(X^n - \tau K^T Y^{n+1}) + \sqrt{2\tau} \xi^{n+1}\\
        X_\theta^{n+1} = X^{n+1} + \theta(X^{n+1}-X^n),
    \end{array}\right.
\end{equation}
where $\tau, \sigma > 0$ are step sizes, $\theta > 0$ is an overrelaxation parameter, $f^\ast$ denotes the convex conjugate of $f$ and the proximal mapping of a convex function $\phi$ is the unique value $\prox_{\gamma \phi} (z) \coloneqq (id + \gamma \partial \phi)^{-1}(z)$. 

The motivation for this update comes from the optimization perspective: The ULA iteration \eqref{eq:ula} can be interpreted as gradient descent with an additive stochastic term. While ULA requires to know the gradients of the potentials, the advantage of primal-dual algorithms is that the iteration can be applied without the need to evaluate any gradients or even the proximal mapping of the composition $f \circ K$. $f$ and $g$ need only be convex, not necessarily differentiable. Hence, it appears desirable to obtain more general sampling algorithms by adding stochasticity to other first-order optimization methods like the primal-dual update.

This is, of course, not a rigorous argument for convergence, and until now, to the best of the authors' knowledge, no convergence theory for ULPDA has been provided. Indeed, the situation seems to be more intricate: When we run ULPDA on a simple toy example, see \cref{fig:1d-gaussian-motivational-example}, the distribution of the samples seems to depend more critically on the choice of step size hyperparameters $\tau,\sigma$ of the algorithm than it is known from ULA.

\begin{figure}[ht]
    \centering
    \includegraphics[width=\linewidth]{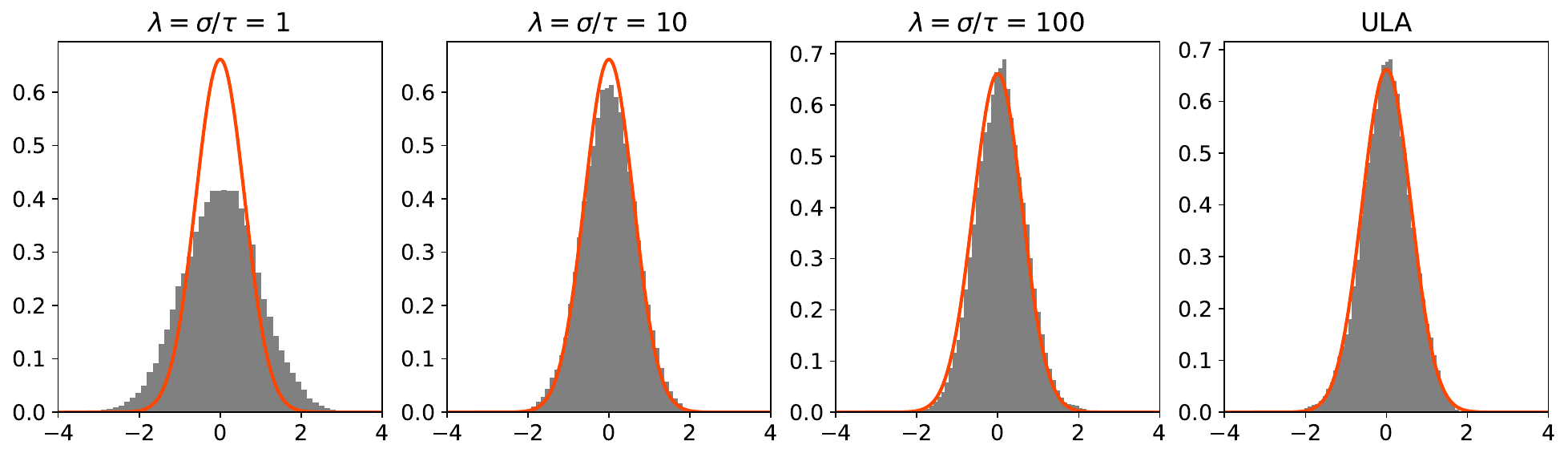}
    \caption{In a simple toy example, we choose $f,g : \bbR \to \bbR$ to be quadratic functions, $K >0$ a real number, so that the target $\mu^\ast$ is Gaussian with the density depicted in orange.
    We run ULPDA with very small primal step size $\tau$ and sufficient burn-in period and repeat the experiment for three different ratios $\sigma/\tau$. The gray histograms show the distribution of samples produced by ULPDA. 
    For comparison, the right figure is produced by ULA with the same primal step size $\tau$.}
    \label{fig:1d-gaussian-motivational-example}
\end{figure}

\subsection{Related Works}
There is a vast literature on unadjusted Langevin algorithms, where we use the term to refer not only to \eqref{eq:ula}, but any time discretization of \eqref{eq:overdamped-langevin}. The convergence analysis is based on the simple fact that the target distribution is the stationary distribution of the SDE \eqref{eq:overdamped-langevin} and that the dynamics converge to the target in continuous time \cite{Roberts1996}. As an underlying concept, this convergence can be explained by drawing the connection between the diffusion dynamics and, equivalently, the gradient flow of relative entropy with respect to the target in Wasserstein space \cite{Jordan1998,Wibisono2018}. Standard convergence results on ULA typically require a Lipschitz continuous gradient of the potential in order to bound the error introduced by time discretization \cite{Dalalyan2019,Durmus2019,Durmus2019a}. 

Since non-smooth densities do not satisfy this gradient Lipschitz condition, many works have treated the extension to the non-smooth case. Generally, the approaches can be divided into two classes. The first regularizes non-smooth terms, resulting in a modified, smooth target and applies typical ULA to this modified distribution. The challenge then typically consists in bounding the terms between the target and the modified distribution, and numerically in trading off the introduced bias and convergence speed via the smoothing parameter \cite{Pereyra2016,Durmus2022,Durmus2018,Lau2022,Cai2024}. The second approach consists in finding different discretizations of \eqref{eq:overdamped-langevin}, i.e. implicit or semi-implicit schemes, that are implementable and provably convergent. These schemes are usually referred to as proximal or implicit Langevin methods, see e.g. \cite{Pereyra2016,Bernton2018,Salim2020,Klatzer2024,Ehrhardt2024}. A work of particular relevance for our setting is \cite{Habring2024}, since the authors assume the same type of splitting of the potential $h = f\circ K + g$ as we do. Their method `Prox-Sub' can be understood as a first-order forward-backward discretization of \eqref{eq:overdamped-langevin}, with the forward step in $f\circ K$ realized as a subgradient since $f$ is assumed non-smooth. We will see that this method is intimately connected to the primal-dual sampling algorithm since it forms the limiting case in which the primal-dual stationary distribution can be related to the target. In \cite{Narnhofer2024}, the authors proposed the primal-dual scheme that we consider here, in order to generalize from the class of methods that employ purely primal discretizations of \eqref{eq:overdamped-langevin}. A partial dualization of the potential is also considered in \cite{Cai2024}, but the scheme differs from the one considered here in that it runs an inner iteration to solve the dual optimality condition. Further, the algorithm in \cite{Cai2024} assumes smoothness of $g$ since the updates require evaluation of $\nabla g$. We also refer to \cite{Lau2024} for a recent comparison of current methods applicable in the non-smooth case.

Constraints on the state space can be viewed as a special case of non-smooth potentials, where the constraint is enforced by adding an indicator function to the potential. Several works have considered this special case usually by incorporating a proximal (which amounts to a projection in the case of an indicator) or a reflection step \cite{Brosse2017,Bubeck2018,Melidonis2023}. Mirroring the Langevin dynamics can be another approach to precondition the dynamics as well as to enforce geometrical constraints \cite{Girolami2011,Ahn2021,Li2021,Lau2022}.

We also mention that due to the dualization, the primal-dual Langevin sampling strategy can be interpreted as an alternating stochastic descent-ascent scheme. There is a more general interest in the analysis of Langevin descent-ascent schemes in the context of saddle point problems, e.g. for the recovery of Nash equilibria, see \cite{Wang2024}. Our results give answers on the convergence in continuous and discretized time for the case where the saddle point functional is of a Lagrangian form obtained from rewriting a convex optimization problem as a saddle point problem.

\subsection{Contributions}
Our main contributions are the following (see also \cref{fig:visualization-contributions} for an overview).
\begin{itemize}
    \item For the algorithm's continuous time limit, we prove a contraction result using a coupling argument. This guarantees the dynamics are asymptotically stable and converge to a stationary distribution $\pi_\lambda(x,y)$, where $(x,y)\in\calX\times \calY \eqqcolon \calZ$ in the joint space of primal and dual variable. Crucially, the distribution depends on $\lambda > 0$, a constant coupling primal and dual time scales.
    \item We show that the marginal of $\pi_\lambda$ in the primal variable $x$ is not the target $\mu^\ast$. This can be interpreted as being a consequence of $\pi_\lambda$ not concentrating its mass on a lower-dimensional submanifold of $\calX\times\calY$, although the diffusion coefficient is degenerate with diffusion only acting directly on $x$. In the case of smooth $f^\ast,g$, we prove this property called hypoellipticity.
    \item We consider two potential frameworks for ``correcting'' the bias between the stationary distribution $\pi_\lambda$ and $\pi^\ast$ (a concentrated push-forward of the target $\mu^\ast$ to $\calZ$). Firstly, a proposed modification of the dynamics with non-homogeneous diffusion coefficient can be employed to force $\pi^\ast$ to be the stationary solution. The target becomes the stationary distribution, but the discretization of the diffusion SDE is impractical. Secondly, we consider the unmodified scheme, but with the ratio $\lambda$ of dual and primal step sizes in the algorithm growing to infinity. The limit recovers the purely primal overdamped Langevin diffusion. Under regularity assumptions on the dualized potential $f$, we show convergence of $\pi_\lambda$ to $\pi^\ast$.
    \item In discrete time, we show stability of the algorithm in the sense that Wasserstein distances between two different instances are bounded if the potentials are convex. Under stronger regularity assumptions on $f^\ast, g$, there exists a stationary solution $\pi_{\lambda,\tau}$ to which the distribution of the iterates converges. We then provide a bound of the bias, i.e. a bound on the Wasserstein distance between $\pi_\lambda$ and $\pi_{\lambda,\tau}$. This allows distance bounds between the generated samples and $\mu^\ast$.
\end{itemize}

\newlength{\topcellheight}\setlength{\topcellheight}{2.6cm}%
\newlength{\topcellboxheight}\setlength{\topcellboxheight}{2.0cm}%
\newlength{\bottomcellheight}\setlength{\bottomcellheight}{2cm}%
\newlength{\bottomcellboxheight}\setlength{\bottomcellboxheight}{1cm}%
\newlength{\cellwidth}\setlength{\cellwidth}{0.38\linewidth}%
\newlength{\cellboxwidth}\setlength{\cellboxwidth}{0.34\linewidth}%
\newlength{\topheight}\setlength{\topheight}{0.6cm}%
\newlength{\leftwidth}\setlength{\leftwidth}{0.2\linewidth}%
\newlength{\leftboxwidth}\setlength{\leftboxwidth}{0.16\linewidth}%
\begin{figure}
    \centering
    \begin{tikzpicture}[
    s1top/.style={align=center,
    minimum height=\topcellheight,
    minimum width=\cellwidth,
    inner xsep=0.02\linewidth,
    inner ysep=0.3cm},
    s1bottom/.style={align=center,
    minimum height=\bottomcellheight,
    minimum width=\cellwidth,
    inner xsep=0.02\linewidth,
    inner ysep=0.3cm},
    s2top/.style={align=center,
    minimum height=\topcellheight,
    minimum width=\leftwidth,
    inner xsep=0.02\linewidth,
    inner ysep=0.3cm},
    s2bottom/.style={align=center,
    minimum height=\bottomcellheight,
    minimum width=\leftwidth,
    inner xsep=0.02\linewidth,
    inner ysep=0.3cm},
    s3/.style={align=center,
    minimum height=\topheight,
    minimum width=\cellwidth,
    inner xsep=0.02\linewidth,
    inner ysep=0.3cm}
    ]
        \matrix (m) [column sep=0cm, row sep=0cm]
        {
            & \node[s3](m12){Primal-dual ($0 < \lambda < \infty$)}; & \node[s3](m13){Primal ($\lambda = \infty$)}; \\
            \node[s2top](m21){\parbox[t][\topcellboxheight][t]{\leftboxwidth}{\centering Continuous time}};&
            \node[s1top](m22){\parbox[t][\topcellboxheight][t]{\cellboxwidth}{\centering Contraction to $\pi_\lambda$ \{\ref{thm:convergence-pd-ct}\},\\Inconsistency \{\ref{thm:1d-gauss-example},\ref{thm:1d-gauss-no-homogeneous-correction}\},\\Non-concentration \{\ref{thm:pd-hypoellipticity}\},\\Corrected SDE \{\ref{thm:stationary-soln-of-modified-fp}\}.}};
            & \node[s1top](m23){\parbox[t][\topcellboxheight][t]{\cellboxwidth}{\centering Stationary solution\\ $\pi^\ast = (\id,\nabla f \circ K)_\# \mu^\ast $\\with the target $\mu^\ast$.}};\\
            \node[s2bottom](m31){\parbox[t][\bottomcellboxheight][b]{\leftboxwidth}{\centering Discretized time\vspace{0.1em}}};&
            \node[s1bottom](m32){\parbox[t][\bottomcellboxheight][b]{\cellboxwidth}{\centering Stability \{\ref{thm:dt-stability-convex}\},\\Contraction to $\pi_{\lambda,\tau}$ \{\ref{thm:dt-contraction-strongly-convex}\}.}};& 
            \node[s1bottom](m33){\parbox[t][\bottomcellboxheight][b]{\cellboxwidth}{\centering ULA and variants from\\ other discretizations of \eqref{eq:overdamped-langevin}.}}; \\
        };
    
        \draw[thick] (m21.north west) -- (m23.north east);
        \draw[thick] (m31.north west) -- (m33.north east);
        \draw[thick] (m12.north west) -- (m32.south west);
        \draw[thick] (m13.north west) -- (m33.south west); 
    
        \node[draw, thick, inner sep=4pt, anchor=center,align=center,fill=white,minimum height=1cm] (box1) at ($(m22.center)!0.5!(m23.center)$) {Limit\\$\lambda \to \infty$\\\{\ref{thm:ct-convergence-to-overdamped-langevin}\}\\};
        \draw[->] ($(box1.west)!0.6!(box1.south)$) -- ($(box1.east)!0.6!(box1.south)$);
    
        \node[draw, thick, inner sep=4pt, anchor=center, fill=white, align=right] (box2) at ($(m22.center)!0.55!(m32.center)$) {\qquad Limit $\tau \to 0$ \{\ref{cor:distance-invariants-discrete-to-continuous}\}};
        \node (aux1) at ($(box2.north)!0.5!(box2.north west)$) {};
        \node (aux2) at ($(box2.south)!0.5!(box2.south west)$) {};
        \draw[<-] ($(aux1)!0.5!(box2.west)$) -- ($(aux2)!0.5!(box2.west)$);
    
        \node[draw, thick, inner sep=4pt, anchor=center, fill=white, align=right] (box3) at ($(m23.center)!0.5!(m33.center)$) {\qquad other works,\\ e.g. \cite{Habring2024} for \\ \qquad $h = f\circ K + g$};
        \node (aux3) at ($(box3.north)!0.5!(box3.north west)$) {};
        \node (aux4) at ($(box3.south)!0.5!(box3.south west)$) {};
        \draw[<-] ($(aux3)!0.5!(box3.west)$) -- ($(aux4)!0.5!(box3.west)$);
    \end{tikzpicture}
    \caption{Overview of the results in this paper. The central theorems are indicated in curled brackets.}
    \label{fig:visualization-contributions}
\end{figure}

The manuscript is organized as follows. In \cref{sec:problemsetting}, we give the considered sampling procedure based on primal-dual optimization and derive its continuous time limit SDE. Since the continuous time limit of the algorithm does not coincide with overdamped Langevin diffusion, we devote \cref{sec:continuous-time} to the analysis of the continuous time setting and its mean field limit Fokker--Planck equation.
The stability and convergence analysis of the algorithm in discrete time is carried out in \cref{sec:discrete-time}. We present numerical illustrations of our most important findings in \cref{sec:numerics} and conclude in \cref{sec:conclusion}. The code for reproducing all numerical experiments is available at \url{https://github.com/lokuger/primal-dual-langevin}.

%% file: sections/2problemsetting.tex
\section{Primal-Dual Langevin Sampling Schemes}\label{sec:problemsetting}
Although the algorithmic approach is not limited to this setting, in the context of Bayesian inference, the target distribution $\mu^\ast$ is typically a posterior distribution. Such a posterior target is defined through the likelihood $p(z|x)$ of observing $z$ given the true parameter $x$, combined with a prior distribution $p(x)$. It is given by
$$p(x|z) = \frac{1}{Z} p(z | x) p(x),$$
where the normalization constant $Z = \int p(z | w) p(w) \rmd w$ is typically intractable in high dimensions. 
As an example class of problems that we consider, in Bayesian inference for imaging inverse problems the distributions are typically high-dimensional and involve non-smooth terms. This is because model-based priors on image distributions often assume sparse coefficients in some basis, the recovery of which corresponds to minimization of non-smooth functions \cite{Candes2006}. Therefore, many relevant posteriors $p(x|z)$ fall into the class of targets given in \eqref{eq:target}. We start by explaining the motivation for primal-dual sampling from the point of view of convex, non-smooth optimization.

\subsection{Motivation from Primal-Dual Optimization}
A common task in Bayesian inverse problems is the computation of the maximum a posteriori estimate (MAP), i.e. the mode of the distribution $\mu^\ast = p(\cdot|z)$. This task forms the connection between Bayesian formulations of inverse problems and variational regularization approaches, since it amounts to the problem of computing the solution to the following minimization problem
\begin{equation}\label{eq:composite-opt-problem}
    \argmin_x f(Kx) + g(x).
\end{equation}
In order to solve \eqref{eq:composite-opt-problem}, depending on the structure of $f$, $K$ and $g$ it can be helpful to employ techniques from convex duality on parts of the potential $h = f\circ K + g$. A popular approach consists in reformulating the variational problem \eqref{eq:composite-opt-problem} as the saddle point problem
\begin{align}
    \min_x h(x) &= \min_x \big(g(x) + \max_y ( \langle Kx, y\rangle - f^*(y))\big) \notag \\
    &= \min_x \max_y s(x,y) \label{eq:map-problem-saddle-point-form}
\end{align}
with $s(x,y) := g(x) + \langle Kx, y\rangle - f^*(y) $. Here $f^\ast$ denotes the convex conjugate of $f$ defined by
\begin{equation*}
    f^\ast(y) := \sup_{w\in\calY} \left\{ \langle w,y \rangle - f(w) \right\}.
\end{equation*}
Under mild assumptions on $g$ and $f^\ast$, one can prove existence of a saddle point $(x^\ast,y^\ast)$ (see section 19 in \cite{Bauschke2011} for details), the primal part $x^\ast$ of which is a solution to the original problem \eqref{eq:composite-opt-problem}. Instead of minimizing $h$, one can then employ primal-dual type algorithms to solve the saddle point problem \eqref{eq:map-problem-saddle-point-form}. For the context of imaging inverse problems, we refer to \cite{Chambolle2016,Burger2016a} for an overview.

The primal-dual optimization method (primal descent and dual ascent with overrelaxation in the primal variable) for \eqref{eq:map-problem-saddle-point-form} takes the form
\begin{equation}\label{eq:pd_optimization_update}
\left\{ \begin{array}{l}
    Y^{n+1} = \prox_{\sigma f^\ast}\left(Y^n+\sigma K X_\theta^{n}\right)\\
    X^{n+1} = \prox_{\tau g}(X^n - \tau K^T Y^{n+1}) \\
    X_\theta^{n+1} = X^{n+1} + \theta(X^{n+1}-X^n),
\end{array}\right.
\end{equation}
with step sizes $\tau,\sigma$, overrelaxation parameter $\theta$ and an initial guess $(X^0,Y^0)$. By convention, we set $X_\theta^0 = X^0$. For convergence theory of the scheme \eqref{eq:pd_optimization_update} and possible accelerations we refer to \cite{Chambolle2011,Chambolle2016}.

In order to draw the parallels between computing the MAP of $\mu^\ast$ and drawing samples from $\mu^\ast$, note that ULA \eqref{eq:ula} can be viewed as a scheme alternating a gradient descent step in $h = - \log \mu^\ast$ and a diffusion step. The diffusion is scaled to ensure the samples have the correct covariance up to a bias which is due to the discretization error and is controlled by the step size $\tau$. Similar to ULA, several first-order optimization algorithms have been transformed into sampling algorithms by inserting a scaled diffusion step at some point in the iteration \cite{Lau2024}. The choice of optimization step is fairly flexible as long as the resulting method can be viewed as an alternating scheme combining a descent step (i.e. a discretized step along the gradient flow of a potential energy) with the diffusion (an exact step along the gradient flow of entropy) \cite{Wibisono2018}.

It appears natural to take a similar perspective on primal-dual optimization methods and add a diffusion step at some point in the iteration. Such a scheme, in the form of \eqref{eq:pd_sampling_update}, was proposed under the name unadjusted Langevin primal-dual algorithm (ULPDA) in \cite{Narnhofer2024}. In this variant of the algorithm, the diffusion step is added only in the primal variable. We will consider both \eqref{eq:pd_sampling_update} and the more general scheme
\begin{equation}\label{eq:general_pd_sampling_update}
\left\{ \begin{array}{l}
    \xi^{n+1} \sim \mathrm{N}(0,I_{d+m})\\
    Y^{n+1} = \prox_{\sigma f^\ast}\left(Y^n+\sigma K X_\theta^{n}\right) + \sqrt{\tau} B_Y\xi^{n+1}\\
    X^{n+1} = \prox_{\tau g}(X^n - \tau K^T Y^{n+1}) + \sqrt{\tau} B_X \xi^{n+1}\\
    X_\theta^{n+1} = X^{n+1} + \theta(X^{n+1}-X^n),
\end{array}\right.
\end{equation}
with general constant coefficients $B_X \in \bbR^{d \times (d+m)}$ and $B_Y \in \bbR^{m \times (d+m)}$
for the stochastic component. Note the increased dimension of the stochastic term compared to \eqref{eq:pd_sampling_update}, which is a special case of \eqref{eq:general_pd_sampling_update} with $B_X = (\sqrt{2}I_d, 0)$ and $B_Y = 0$.

As a further variant of the algorithm, we remark that in \eqref{eq:pd_sampling_update}, we add the stochastic term $\sqrt{2\tau}\xi^{n+1}$ after applying the proximal operator. Similar to other proximal Langevin algorithms, it may instead be added inside the argument of the proximal mapping, effectively making the primal update read
$$ X^{n+1} = \prox_{\tau g}(X^n - \tau K^T Y^{n+1} + \sqrt{2\tau} \xi^{n+1}). $$
Since the continuous time limit of both versions is the same SDE, our results on asymptotic behaviour in continuous time in \cref{sec:continuous-time} hold for either variant. In the discrete time analysis in \cref{sec:discrete-time}, we will focus on the version \eqref{eq:pd_sampling_update} with diffusion as the outer step. The same results hold for the other variant; we explain in \cref{cor:stability-modified-algorithm,cor:distances-modified-algorithm} how our arguments can be modified to prove them.

\subsection{Continuous Time Perspective}
In the analysis of ULA's convergence behaviour, it is helpful to view the update step as an Euler--Maruyama time discretization of the overdamped Langevin diffusion \eqref{eq:overdamped-langevin}. Convergence results for ULA are based on the fact that the target measure $\mu^\ast$ is the unique invariant distribution of the Markov process $X_t$.

In order to analyze ULPDA, we want to define a similar time-continuous limit of \eqref{eq:pd_sampling_update}. We assume that the step sizes $\tau, \sigma$ are coupled via the constant $\lim_{\tau, \sigma \to 0} \frac{\sigma}{\tau} = \lambda$ where $ 0 < \lambda < \infty$. Taking both step sizes to zero at this ratio, the formal limiting Markov process $(X_t,Y_t)$ satisfies the diffusion equation
\begin{equation}\label{eq:ct-pd-diffusion}
    \left\{ \begin{array}{l}
    \rmd X_t \in -(\partial g(X_t) + K^TY_t) \rmd t + \sqrt{2} \rmd W_t\\
    \rmd Y_t \in - \lambda (\partial f^*(Y_t) - K X_t) \rmd t,
    \end{array}\right.
\end{equation}
which we can write in a joint variable $Z_t = (X_t,Y_t)$ as
\begin{equation}\label{eq:ct-pd-diffusion-jointvar}
    \rmd Z_t \in - A_{\tup{PD}}(Z_t) \rmd t + B_{\tup{PD}} \rmd W_t,
\end{equation}
where we introduced the primal-dual coefficients
\begin{equation}\label{eq:coeff-ct-pd-diffusion}
    A_{\tup{PD}}(Z_t) \coloneqq \begin{pmatrix}
    \partial g(X_t) + K^TY_t\\
    \lambda (\partial f^\ast(Y_t) - KX_t)
\end{pmatrix},\qquad B_{\tup{PD}} \coloneqq \begin{pmatrix}
    \sqrt 2 I_d \\ 0
\end{pmatrix}.
\end{equation} 
Note in particular that the diffusion coefficient $\Sigma_{\tup{PD}} \coloneqq \frac12 B_{\tup{PD}}B_{\tup{PD}}^T$ is only positive semidefinite (or degenerate) with an $m$-dimensional null space, where $m$ is the dimension of the dual variable $Y$. The standard existence and uniqueness theory for Fokker--Planck equations with positive definite diffusion coefficient therefore does not apply.

In the more general scheme \eqref{eq:general_pd_sampling_update}, the diffusion equation takes the form
\begin{equation}\label{eq:ct-general-pd-diffusion-jointvar}
    \rmd Z_t \in - A_{\tup{PD}}(Z_t) \rmd t + B_{\tup{gPD}} \rmd W_t, \qquad B_{\tup{gPD}} \coloneqq \begin{pmatrix}
    B_X \\ B_Y
\end{pmatrix}.
\end{equation}
If $B_{\tup{gPD}}$ (where `gPD' denotes `general primal-dual') is invertible, then $\Sigma_{\tup{gPD}} \coloneqq \frac12 B_{\tup{gPD}}B_{\tup{gPD}}^T$ becomes positive definite, and the corresponding Fokker--Planck equation has a unique stationary solution \cite{Pavliotis2014}. We will check whether the stationary solution can be related to the density of the target distribution $\mu^\ast$.

%% file: sections/3continuoustime.tex
\section{The Stationary Solution in Continuous Time}\label{sec:continuous-time}

We start by proving the existence of a unique stationary solution of the primal-dual diffusion SDE \eqref{eq:ct-pd-diffusion-jointvar}. We further show that the corresponding primal-dual Fokker--Planck equation is contractive in $\calP_2(\calZ)$, the space of distributions over $\calZ$ with finite second moment. This will imply that any solution converges in distribution to a stationary state. Convergence is shown in weighted 2-Wasserstein distance: Given any norm $\norm{\cdot}$ on $\calX$, a transport distance on $\calP_2(\calX)$ is defined by
\begin{equation}\label{eq:transport-distance}
    \calT(\mu, \tilde \mu) := \inf_{\calM \in \Gamma(\mu,\tilde \mu)} \left( \bbE_{(x,\tilde x)\sim \calM} \left[ \norm{x - \tilde x}^2 \right] \right)^{1/2},
\end{equation}
where $\Gamma(\mu,\tilde \mu)$ denotes the set of all couplings with marginals $\mu$ and $\tilde \mu$, i.e.
$$ \Gamma(\mu,\tilde \mu) = \left\{ \calM \in \calP(\calX \times \calX) ~:~ \mu = \int_{\calX} \calM(\cdot,\tilde x)\,d\tilde x  ~\textrm{and}~ \tilde\mu = \int_{\calX} \calM(x, \cdot)\,dx \right\}. $$
$\calT$ defines a metric on $\calP_2(\calX)$. If $\norm{\cdot}$ is the Euclidean norm, then $\calT$ is the 2-Wasserstein distance $\calW_2$.

\subsection{Fokker--Planck Equation and Existence of a Stationary Solution}\label{subsec:mean-field-limit}
The contraction result follows an argument from the literature on overdamped Langevin diffusion, see \cite{Fournier2021,Durmus2019}. Since this simpler case is instructive, we repeat it here and then generalize to the setting of an SDE with any strongly monotone operator as drift coefficient.

\subsubsection{Contraction of Overdamped Langevin Diffusion via a Coupling Argument}\label{subsubsec:contraction-overdamped-langevin-diffusion}
Suppose for now that the potential $h : \bbR^d \to \bbR$ is differentiable and $\omega_h$-strongly convex. We repeat here an instructive calculation that has been carried out, e.g. in \cite{Fournier2021,Durmus2019}, to show stability of the overdamped Langevin diffusion process
\begin{equation}\label{eq:Langevin_diffusion_gradient}
    dX_t = - \nabla h(X_t)  dt + \sqrt{2} dW_t.
\end{equation}
The distribution $\mu_t$ of the process $X_t$ in \eqref{eq:Langevin_diffusion_gradient}, denoted $\mu_t = \Law(X_t)$, satisfies the Fokker--Planck equation (or forward Kolmogorov equation) \cite{Pavliotis2014} given by
$$ \frac{\partial \mu_t}{\partial t} = \nabla_x \cdot  ( \mu_t \nabla h(x)) + \Delta_x \mu_t. $$
The target distribution $\mu^\ast$ is the unique stationary solution of the Fokker--Planck equation. Furthermore, the PDE is contractive in Wasserstein distance, which can be shown using the following coupling argument. We will call two processes coupled if both solve \eqref{eq:Langevin_diffusion_gradient} with the same Brownian motion. Let $X_t, \tilde X_t$ be coupled with $X_0 \sim \mu_0$ and $\tilde X_0 \sim \tilde \mu_0$. Then the distribution $\calM_t$ of the joint variable $(X_t,\tilde X_t)$ satisfies the Fokker--Planck equation
\begin{equation}\label{eq:joint_FP_gradient}
    \frac{\partial \calM_t}{\partial t} = \nabla_x \cdot ( \calM_t \nabla h (x)) + \nabla_{\tilde x}\cdot (\calM_t \nabla h (\tilde x)) + (\nabla_x + \nabla_{\tilde x}) \cdot (\nabla_x + \nabla_{\tilde x})\calM_t.
\end{equation}
Note that we can rewrite this using $(\nabla_x + \nabla_{\tilde x}) \cdot (\nabla_x + \nabla_{\tilde x}) = \Delta_{x+\tilde x}$ by a coordinate change $(x,\tilde x) \mapsto (\frac{x+\tilde x}{2}, x - \tilde x)$. For any distance function of the form $d(x,\tilde x) = \rho(\norm{x-\tilde x}_2)$, we directly see $\Delta_{x+\tilde x} d(x,\tilde x) = 0$. Using this and the property \eqref{eq:joint_FP_gradient} of the joint distribution $\calM_t$, a contraction in Wasserstein distance can hence be shown in the following way. We start by estimating
\begin{align*}
    \frac{1}2 \frac{d}{dt} \bbE_{(x,\tilde x) \sim \calM_t} \left[ \norm{x-\tilde x}_2^2 \right] &= \frac{1}2 \frac{d}{dt} \int \norm{x-\tilde x}_2^2 \calM_t(\mathrm d x,\mathrm d \tilde x) \\
    &= \frac{1}2 \int \norm{x-\tilde x}_2^2 \left[\nabla_x \cdot ( \calM_t \nabla h (x)) + \nabla_{\tilde x}\cdot (\calM_t \nabla h (\tilde x)) + \Delta_{x+\tilde x} \calM_t \right] \rmd x \rmd \tilde x \\
    &= - \int (x-\tilde x) \cdot (\nabla h(x) - \nabla h(\tilde x)) \calM_t(\mathrm d x,\mathrm d \tilde x) \\
    &\le - \omega_h \bbE_{(x,\tilde x) \sim \calM_t} \left[ \norm{x-\tilde x}_2^2 \right],
\end{align*}
where we used that due to strong convexity of $h$, we have $ (x-\tilde x) \cdot (\nabla h(x) - \nabla h(\tilde x)) \ge \omega_h \norm{x-\tilde x}_2^2$. Applying now Gronwall's lemma, we obtain
$$ \bbE_{(x,\tilde x) \sim \calM_t} \left[ \norm{x-\tilde x}_2^2 \right] \le e^{-2\omega_h t} \bbE_{(x,\tilde x) \sim \calM_0} \left[ \norm{x-\tilde x}_2^2 \right]. $$
The left side is bounded from below by $\calW_2^2(\mu_t,\tilde \mu_t)$ by definition. By taking the infimum over all couplings $\calM_0$ (with marginals $\mu_0$ and $\tilde \mu_0$) on the right hand side, we obtain
$$ \calW_2(\mu_t,\tilde \mu_t) \leq e^{- \omega_h t} \calW_{2}(\mu_0,\tilde \mu_0). $$
In particular, if we choose one of the processes to be initialized at the target distribution as $\tilde \mu_0 = \mu^\ast$, then the distribution of $\tilde X_t$ is stationary at $\tilde \mu_t = \mu^\ast$ and by the above we obtain the convergence result
$$ \calW_2(\mu_t,\mu^\ast) \leq e^{- \omega_h t} \calW_{2}(\mu_0,\mu^\ast). $$

\subsubsection{Contraction and Stationary Solution of the Primal-Dual Diffusion Equation}\label{subsubsec:contraction-general-diffusion}
We can generalize the above argument allowing an analysis of the time-continuous limit \eqref{eq:ct-pd-diffusion-jointvar} of the primal-dual sampling scheme \eqref{eq:pd_sampling_update}. Although the primal-dual dynamics are not the gradient flow of a corresponding relative entropy anymore, contraction can equivalently be shown for a coupling of two processes. For that, let $Z_t$ satisfy a general stochastic differential inclusion with non-homogeneous drift coefficient
\begin{equation}\label{eq:SDI_general_strong_monotone_op}
    \mathrm{d}Z_t \in - A(Z_t)\rmd t + B \rmd W_t.
\end{equation}
Here $W_t$ is Brownian motion in $\bbR^k$ for some $k>0$ and $B \in \bbR^{(d+m)\times k}$ is constant. We assume the drift $A$ to be a set-valued monotone operator, meaning
$$ \inpro{w-\tilde w}{z - \tilde z} \ge 0, \qquad \forall w \in A(z), \tilde w  \in A(\tilde z). $$
We say that $A$ is $\omega$-strongly monotone if the following stronger condition is satisfied:
$$ \inpro{w-\tilde w}{z - \tilde z} \ge \omega \norm{z - \tilde z}^2, \qquad \forall w \in A(z), \tilde w  \in A(\tilde z). $$

\begin{proposition}\label{prop:general_FP}
    Let $Z_t$ be a solution to \eqref{eq:SDI_general_strong_monotone_op} and denote by $\pi_t$ the density w.r.t. the Lebesgue measure of the distribution of $Z_t$. Then $\pi_t$ solves the Fokker--Planck equation
    \begin{equation}\label{eq:FP_general_strong_monotone_op}
        \frac{\partial \pi_t}{\partial t} = \nabla_z \cdot \left( w \pi_t + \Sigma \nabla_z \pi_t\right),\qquad w\in A(z),
    \end{equation}
    with $\Sigma = \frac{1}{2} B B^T$.
\end{proposition}
We refer to \cite[ch. 4]{Pavliotis2014} for general background on the Fokker--Planck equation of SDEs and \cite{Pettersson1995} for existence and uniqueness results in the special case of inclusions with multivalued, maximally monotone drift term.

Since we are repeatedly going to make use of coupling arguments, note that by applying \cref{prop:general_FP} to a coupled diffusion
\begin{equation}\label{eq:SDI_coupling_general_strong_monotone_op}
    \begin{pmatrix} \mathrm{d} Z_t\\ \mathrm{d} \tilde Z_t \end{pmatrix} \in - \begin{pmatrix} A(Z_t) \\ A(\tilde Z_t) \end{pmatrix} \mathrm d t + \begin{pmatrix} B \\ B \end{pmatrix} \mathrm{d} W_t,
\end{equation}
we can also obtain a Fokker--Planck equation for the joint density of the coupled processes.
\begin{proposition}\label{prop:general_FP_joint}
    Let $Z_t, \tilde Z_t$ be two coupled processes that both solve \eqref{eq:SDI_general_strong_monotone_op} with the same Brownian motion, i.e. $(Z_t,\tilde Z_t)$ solves \eqref{eq:SDI_coupling_general_strong_monotone_op}. Then the joint distribution $\Pi_t(z,\tilde z)$ of $(Z_t,\tilde Z_t)$ satisfies the Fokker--Planck equation
    \begin{equation}\label{eq:joint_FP_general_strong_monotone_op}
        \frac{\partial \Pi_t}{\partial t} = \nabla_z \cdot \left( w \Pi_t \right) + \nabla_{\tilde z} \cdot \left( \tilde w \Pi_t \right) + (\nabla_z+\nabla_{\tilde z}) \cdot (\Sigma (\nabla_z+\nabla_{\tilde z}) \Pi_t),\qquad w\in A(z),\tilde w \in A(\tilde z).
    \end{equation}
\end{proposition}

For showing the contractivity of overdamped Langevin dynamics in \cref{subsubsec:contraction-overdamped-langevin-diffusion}, the relevant requirement was the strong convexity of $h$, implying that $\nabla h$ is a strongly monotone operator. The computation can be directly generalized under a strong monotonicity assumption on $A$.
\begin{lemma}\label{lem:stability_general_dynamics_cont_time}
     Suppose $A$ is $\omega$-strongly monotone with respect to some inner product $\inpro{\cdot}{\cdot}$ with induced norm $\norm{\cdot}$. Then \eqref{eq:SDI_general_strong_monotone_op} is contractive in the following sense: If $Z_t, \tilde Z_t$ are two coupled processes that solve \eqref{eq:SDI_general_strong_monotone_op} with the same Brownian motion, with associated distributions $\pi_t, \tilde \pi_t$ and initial conditions $\pi_0,\tilde\pi_0 \in \calP_2(\calZ)$, then for $t \ge 0$ it holds
    \begin{equation}\label{eq:semigroup-contraction}
        \calT(\pi_t,\tilde \pi_t) \le e^{- \omega t} \calT(\pi_0,\tilde \pi_0),
    \end{equation}
    where $\calT$ is the transport distance \eqref{eq:transport-distance} induced by the norm $\norm{\cdot}$. As $(\calP_2(\calZ), \calW_2)$ is complete \cite[ch. 7]{Ambrosio2008} and $\calT$ and $\calW_2$ are equivalent for any norm $\norm{\cdot}$, this implies that the Fokker--Planck equation \eqref{eq:FP_general_strong_monotone_op} has a stationary solution $\pi$, and $\pi_t, \tilde{\pi}_t$ weakly converge to $\pi$.
\end{lemma}

For the proof, note that the exemplary computation for overdamped Langevin diffusion in \cref{subsubsec:contraction-overdamped-langevin-diffusion} is agnostic to both the choice of (Hilbert) norm $\norm{\cdot}$ and the fact that the drift $\nabla h$ is the gradient of a potential. It only required the monotonicity condition $ (x-\tilde x) \cdot (\nabla h(x) - \nabla h(\tilde x)) \ge \omega_h \norm{x-\tilde x}^2$. The proof therefore follows the exact same computations for a general drift $A$, strongly monotone in any $\norm{\cdot}$. The full proof can be found in \cref{subsec:appendix-continuous-time:mean-field-limit}.

Consider now the general primal-dual type diffusion \eqref{eq:ct-general-pd-diffusion-jointvar} and denote by $\pi_t$ the law of $Z_t$. By \cref{prop:general_FP}, $\pi_t$ solves the Fokker--Planck equation
\begin{equation}\label{eq:fp-pd}
\frac{\partial \pi_t}{\partial t} = \nabla_z \cdot \left( w \pi_t + \Sigma_{\tup{gPD}} \nabla_z \pi_t\right),\qquad w\in A_{\tup{PD}}(z),
\end{equation}
with $\Sigma_{\tup{gPD}} = \frac{1}{2} B_{\tup{gPD}} B_{\tup{gPD}}^T $ as in \eqref{eq:ct-general-pd-diffusion-jointvar} and $A_{\tup{PD}}$ given in \eqref{eq:coeff-ct-pd-diffusion}. The drift coefficient $A_{\tup{PD}}$ is a monotone operator with respect to the weighted inner product $\inpro{\cdot}{\cdot}_{1,\lambda^{-1}}$ introduced below.

Under a strong convexity assumption on $g$ and $f^\ast$, the drift is even strongly monotone in a weighted norm, so that we can apply \cref{lem:stability_general_dynamics_cont_time} to the primal-dual diffusion equation. In the following, for some $a,b>0$ we equip the space $\calZ = \calX \times \calY$ of joint primal and dual variable with the weighted inner product $\inpro{z_1}{z_2}_{a,b} \coloneqq \inpro{x_1}{x_2}_a + \inpro{y_1}{y_2}_b \coloneqq a\inpro{x_1}{x_2} + b \inpro{y_1}{y_2}$, where $\inpro{\cdot}{\cdot}$ here denotes the Euclidean inner product on $\calX$ and $\calY$. The induced norms will be denoted $\norm{\cdot}_a$ (on $\calX$ and $\calY$) and $\norm{\cdot}_{a,b}$ (on $\calZ$).\\
The induced transport distance \eqref{eq:transport-distance} on $\calP_2(\calZ)$ will be denoted by
\begin{equation}\label{eq:weighted-transport-distance}
    \calT_{a,b}(\pi, \tilde \pi) \coloneqq \inf_{\Pi \in \Gamma(\pi,\tilde \pi)} \left( \bbE_{(z,\tilde z)\sim \Pi} \left[ \norm{z - \tilde z}_{a,b}^2 \right] \right)^{1/2}.
\end{equation}
Note that the equivalence of norms on $\calZ$ implies equivalence of all $\calT_{a,b}$ for $a,b>0$.

\begin{theorem}\label{thm:convergence-pd-ct}
    Assume that $g,f^\ast$ are $\omega_g$-strongly convex and $\omega_{f^\ast}$-strongly convex, respectively, with $\omega_g, \omega_{f^\ast} > 0$. Let $Z_t, \tilde Z_t$ be two coupled processes that solve the general primal-dual diffusion equation \eqref{eq:ct-general-pd-diffusion-jointvar} with corresponding densities $\pi_t$ and $\tilde \pi_t$. Then it holds 
    $$ \calT_{1,\lambda^{-1}}(\pi_t, \tilde \pi_t) \le e^{-\omega t} \calT_{1,\lambda^{-1}} (\pi_0, \tilde \pi_0), $$
    where $\omega = \min\{\omega_g, \lambda\omega_{f^\ast}\} > 0$. By equivalence of $\calT_{1,\lambda^{-1}}$ and $\calW_2$, and completeness of $(\calP_2(\calZ), \calW_2)$, the dynamics are contractive in $\calW_2$ and the primal-dual diffusion has a unique stationary distribution. Making the dependence on $\lambda$ explicit, we denote the stationary distribution by $\pi_\lambda$.
\end{theorem}

The proof applies \cref{lem:stability_general_dynamics_cont_time}: We observe that the scale $\lambda$ between primal and dual timescales provides a natural weighting of the metric on the joint space. In particular, strong convexity of the potential terms $g, f^*$ implies a strongly monotone drift in $\norm{\cdot}_{1,\lambda^{-1}}$. The full proof can be found in \cref{subsec:appendix-continuous-time:mean-field-limit}.

In the following, we will also use the $x$-marginal $\mu_\lambda$ of $\pi_\lambda$, defined as $\mu_\lambda(x) \coloneqq \int_{\calY} \pi_\lambda(x,y) \rmd y.$

\subsection{Non-Concentration of the Stationary Solution}\label{subsec:non-concentration}
Having established the existence of $\pi_\lambda$ and its marginal $\mu_\lambda$, we need to check how they relate to the target $\mu^\ast$ from which we wish to draw samples. While overdamped Langevin dynamics is designed to have the target $\mu^\ast$ as its stationary solution, this is unfortunately not the case in the primal-dual setting. In general, $\mu^\ast \neq \mu_\lambda$. We demonstrate this on a simple example of a Gaussian target, where we can solve the Fokker--Planck equation \eqref{eq:fp-pd} in closed form.

\begin{theorem}\label{thm:1d-gauss-example}
    Consider a one-dimensional setting with the target $\mu^\ast(x) \propto \exp(-h(x))$ given by
    \begin{equation*}\label{eq:1d-gauss-example}
        f(y) \coloneqq \frac{1}{2c_f}y^2,\ g(x) \coloneqq \frac{1}{2c_g}x^2,\ h(x) \coloneqq f(kx) + g(x) = \frac{k^2c_g + c_f}{2c_fc_g} x^2, \qquad x,y\in\bbR,
    \end{equation*}
    where $k \neq 0$ and $c_g, c_f > 0$. It holds $f^\ast(y) = \frac{c_f}{2}y^2$ for all $y\in \bbR$. The primal-dual diffusion equation \eqref{eq:ct-pd-diffusion-jointvar} for the joint variable $Z_t = (X_t, Y_t)^T \in \bbR^2$ reads
    \begin{equation*}
        \mathrm{d}Z_t = -A_{\tup{PD}}Z_t \mathrm{d}t + B_{\tup{PD}} \mathrm{d}W_t, \quad \text{with } A_{\tup{PD}} = \begin{pmatrix} c_g^{-1} & k\\ -\lambda k & \lambda c_f \end{pmatrix}, B_{\tup{PD}} = \begin{pmatrix} \sqrt{2}\\ 0\end{pmatrix}
    \end{equation*}
    where $W_t$ is one-dimensional Brownian motion and the weight $\lambda >0$ relates primal and dual time scales. The following statements hold:
    \begin{enumerate}[label=(\roman*)]
        \item\label{item:1d-gauss-example-smoothness-of-soln} For any (possibly singular) initialization $\pi_0$ such that $Z_0 \sim \pi_0$, the distribution $\pi_t \coloneqq \Law(Z_t)$ has a $C^\infty$-smooth Lebesgue density for $t>0$.
        \item\label{item:1d-gauss-example-soln-for-gaussinit} If $\pi_0$ is Gaussian (including degenerate Gaussians with singular covariance matrix), then
        $\pi_t$ is (non-degenerate) Gaussian for all $t>0$. As $t \to \infty$, the solution $\pi_t$ converges to the Gaussian stationary distribution $\pi_\lambda$ with mean zero and covariance matrix
        $$ \frac{1}{(c_f + k^2c_g)(1+\lambda c_f c_g)} \begin{pmatrix}
            c_g (c_f + k^2c_g) + \lambda c_f^2 c_g^2 & k \lambda c_f c_g^2 \\ k \lambda c_f c_g^2 & k^2 \lambda c_g^2
        \end{pmatrix}. $$
        The $x$-marginal $\mu_\lambda$ is not equal to the target $\mu^\ast$. However, the limiting joint distribution as $\lambda \to \infty$ is the degenerate Gaussian $(\id, f'(k\,\cdot))_{\#} \mu^\ast$ whose $x$-marginal is the target $\mu^\ast$.
    \end{enumerate}
\end{theorem}

The proof solves the Fokker--Planck equation via its characteristic equations and derives its stationary solution exactly. The computations can be found in \cref{subsec:appendix-continuous-time:non-concentration}.

The negative result \cref{thm:1d-gauss-example} naturally raises the question whether the diffusion equation can be modified such that the target becomes its stationary distribution. While we will consider the more general case of non-homogeneous, problem-dependent diffusion coefficients in \cref{subsec:bias-corrected-diffusion}, the (algorithmically) most straightforward approach would be to add homogeneous diffusion not only in the primal variable $X^n$ of the algorithm, but also in the dual $Y^n$, as in \eqref{eq:general_pd_sampling_update}. Unfortunately, the next theorem shows that the time-continuous limit \eqref{eq:ct-general-pd-diffusion-jointvar} always remains asymptotically inconsistent.

\begin{theorem}\label{thm:1d-gauss-no-homogeneous-correction}
    Consider the same target $\mu^\ast$ as in \cref{thm:1d-gauss-example}, but the more general diffusion \eqref{eq:ct-general-pd-diffusion-jointvar}
    \begin{equation}\label{eq:ctpd-1dgaussex-generaldiffcoeff}
        \mathrm{d}Z_t = -A_{\tup{PD}}Z_t \mathrm{d}t + B_{\tup{gPD}} \mathrm{d}W_t, \quad \text{with } A_{\tup{PD}} = \begin{pmatrix} c_g^{-1} & k\\ -\lambda k & \lambda c_f \end{pmatrix}, B_{\tup{gPD}} = \begin{pmatrix} b_{11} & b_{12}\\ b_{21} & b_{22}\end{pmatrix}
    \end{equation}
    where $W_t$ is Brownian motion in $\bbR^2$. Then there exists no choice for the entries of $B_{\tup{gPD}} \in \bbR^{2\times 2}$ independent of $c_g, c_f$ and $k$ such that the stationary solution $\pi_\lambda$ of \eqref{eq:ctpd-1dgaussex-generaldiffcoeff} has $x$-marginal $\mu^{\ast}$. 
\end{theorem}

The proof can be found in \cref{subsec:appendix-continuous-time:non-concentration}.

We want to elaborate more on the concentration properties of $\pi_\lambda$ in the scheme \eqref{eq:ct-pd-diffusion-jointvar} with primal diffusion. \cref{thm:1d-gauss-example} showed that in the limiting case $\lambda \to \infty$, the stationary distribution does not have a Lebesgue density anymore since the mass concentrates on the optimality condition $\partial f^\ast(y) = Kx$. In that case, the primal marginal actually converges to the target distribution, which is intuitively clear since inserting $Y_t \in \partial f(KX_t)$ into the equation for $X_t$ reduces the system to overdamped Langevin diffusion. We will rigorously generalize these properties of the toy example of \cref{thm:1d-gauss-example} in two ways. Specifically, we show in \cref{subsec:convergence-to-ula} under which conditions the stationary solution converges, as $\lambda \to \infty$, to a singular push-forward of the target $\mu^\ast$ in the space of distributions over $\calZ$. Also, we next show that the concentration at the optimal dual value generally does not occur for finite $\lambda$, even though the diffusion coefficient is degenerate. 

In the case of smooth coefficients, the Fokker--Planck operator can be shown to be hypoelliptic, implying that any solution to the dynamics has a smooth density with respect to the Lebesgue measure for all times $t>0$, even in the case of a concentrated initialization.

\begin{definition}
    Let $\calL$ be a partial differential operator with smooth coefficients. If, for every open set $\Omega$ and every distribution $u$ on $\Omega$, $\calL u \in C^\infty(\Omega)$ implies $u\in C^\infty(\Omega)$, then $\calL$ is called hypoelliptic.
\end{definition}

Every elliptic operator is hypoelliptic, and equally any operator of the form $-\partial_t + \calL$ where $\calL$ is the generator of a diffusion process with uniformly positive definite diffusion coefficient. A special case is overdamped Langevin diffusion with smooth drift, which naturally has a smooth transition density.

The degenerate diffusion coefficient in \eqref{eq:ct-pd-diffusion-jointvar} results in a constant but only positive semidefinite diffusion coefficient $\Sigma_{\tup{PD}}$ in the Fokker--Planck equation \eqref{eq:fp-pd}. In order to prove that the Fokker--Planck equation is hypoelliptic, we hence have to employ a weaker condition for hypoellipticity, usually termed `Hörmander's condition' \cite{Hoermander1967}. In order to state it, we need the definition of the commutator operation and Lie algebras spanned by smooth vector fields.

\begin{definition}
    For two vector fields $A,B$ on $\bbR^n$, the Lie bracket or commutator $[A,B]$ is the vector field on $\bbR^n$ defined by
    $$ [A,B](x) = DB(x) A(x) - DA(x) B(x), $$
    with the derivative matrices $(DA)_{i,j}(x) = \partial_j A_i(x)$. 
\end{definition}

\begin{lemma}[\cite{Hoermander1967}]\label{lem:hoermander-condition-hypoellipticity}
    Assume that the generator $\calL$ of a diffusion process with state space $\bbR^n$ has $C^\infty$-smooth coefficients and can be written in the form
    $$ \calL = A_0 + \sum_{j=1}^k A_j^2$$
    with first-order differential operators $A_0, \dots, A_k$. Interpreting the differential operators as vector fields, define collections of vector fields $\calA_l$ by
    \begin{equation*}
    \calA_0 \coloneqq \{A_j\,:\, j=1,\dots,k\},\quad \calA_{l+1} = \calA_l \cup \{ [A,A_j]\,:\,A\in \calA_l, j=0,\dots,k \}.
    \end{equation*}
    If the vector spaces $\calA_l(x) \coloneqq \spanop\left(\{A(x)\,:\, A\in \calA_l \}\right)$ satisfy $\bigcup_{l\ge 0}\calA_l(x) = \bbR^n$ for every $x \in \bbR^n$, then $\calL$ is said to satisfy the parabolic Hörmander condition. In that case, $\partial_t - \calL$ and $\partial_t-\calL^\ast$ are both hypoelliptic, hence any process generated by $\cal L$ has $C^\infty$-smooth transition probability. In particular, any solution to the Fokker--Planck equation $ \partial_t p = \calL^\ast p $ is $C^\infty((0,\infty)\times \bbR^n)$-smooth.
\end{lemma}

We refer to \cite[ch. 6]{Pavliotis2014} for a discussion in the context of diffusion SDEs and an application to the Langevin diffusion SDE with friction. 
For the diffusion equation arising in our primal-dual setting, the generator is given by
$$\calL_{\tup{PD}}u = - A_{\tup{PD}}^T(z)\nabla_z u + \Delta_x u. $$
The following result gives a sufficient condition under which the primal-dual dynamics considered here satisfy Hörmander's condition, so that the corresponding Fokker--Planck operator is hypoelliptic.

\begin{theorem}\label{thm:pd-hypoellipticity}
    Assume that $g \in C^{\infty}(\calX)$ and $f^\ast \in C^\infty(\calY)$, so that the diffusion \eqref{eq:ct-pd-diffusion-jointvar} has smooth coefficients. If $\dim(\calY) = m \le d = \dim(\calX)$ and $K$ has full rank $m$, then $\calL_{\tup{PD}}$ satisfies Hörmander's condition. As a consequence, any stationary solution of \eqref{eq:fp-pd}, in particular $\pi_\lambda$ under the assumptions of \cref{thm:convergence-pd-ct}, has a $C^\infty(\calZ)$-smooth Lebesgue density.
\end{theorem}

The proof consists in verifying the assumptions of \cref{lem:hoermander-condition-hypoellipticity}; it can be found in \cref{subsec:appendix-continuous-time:non-concentration}.

While the smoothness assumption on $g,f^\ast$ is necessary for technical reasons, our numerical test validated that the dispersion in the dual variable is a more general phenomenon. Concentration in the dual variable can generally not be expected in the primal-dual dynamics.

\subsection{Convergence to Overdamped Langevin Diffusion}\label{subsec:convergence-to-ula}
We now continue by showing that in the limiting case $\lambda \to \infty$ and under some smoothness assumption on $f$, the stationary solution $\pi_\lambda$ of \eqref{eq:ct-pd-diffusion-jointvar} becomes singular. In the limit, the primal marginal distribution coincides with the target.

Consider a solution $(X_t,Y_t)$ of \eqref{eq:ct-pd-diffusion-jointvar} with distribution $(X_t,Y_t) \sim \pi_t$, denoting the $x$-marginal by $X_t\sim \mu_t$. We are interested in the behaviour of $\mu_t$ in the limit $\lambda \rightarrow \infty$. Intuitively, as $\lambda$ becomes large, it is clear that $Y_t$ is forced to concentrate, with the (formally) limiting SDE taking the form
\begin{equation}\label{eq:ld-artificial-pd}
\begin{split}
    \rmd X_t &\in - (\partial g ( X_t) + K^T Y_t) \rmd t + \sqrt{2}\rmd W_t,\\
    Y_t &\in \partial f(K X_t).
\end{split}
\end{equation}
This coincides with overdamped Langevin diffusion on the potential $g+f\circ K$, with the target $\mu^\ast$ as invariant distribution.

This suggests we can use a coupling argument between \eqref{eq:ld-artificial-pd} and the primal-dual dynamics \eqref{eq:ct-pd-diffusion-jointvar} to bound the Wasserstein distance between the stationary state of \eqref{eq:ct-pd-diffusion-jointvar} and the target.

\begin{theorem}\label{thm:ct-convergence-to-overdamped-langevin}
    Assume that $f \in C^3$ and $g \in C^1$, that $g,f^\ast$ are $\omega_g$- and $\omega_{f^\ast}$-strongly convex, respectively, and that the target distribution $\mu^\ast$ satisfies
    \begin{gather*}
        \bbE_{X \sim \mu^\ast} \left[ \norm{\nabla (\nabla f \circ K) (X)}_F^2 \right] =: D_1 < \infty,\\
        \bbE_{X \sim \mu^\ast} \left[ \norm {\nabla \cdot \nabla (\nabla f \circ K) (X)}^2 \right] =: D_2 < \infty,\\
        \bbE_{X \sim \mu^\ast} \left[ \norm{ \nabla (\nabla f \circ K) (X) \left(\nabla g(X) + K^T \nabla f (KX)\right)}^2 \right] =: D_3 < \infty,
    \end{gather*}
    with constants $D_i$, where in the second line the Laplacian is understood as acting component-wise on the $\calY$-valued function $\nabla f \circ K$. For $\omega_g/\omega_{f^\ast} < \lambda$ and a constant $C_1(\lambda) = (\lambda \omega_g / D_1)^{-1/2} + \calO(\lambda^{-3/2})$ as $\lambda \to \infty$, the following two statements hold:
    \begin{enumerate}
        \item If the primal-dual dynamics \eqref{eq:ct-pd-diffusion-jointvar} are initialized with $X_0 \sim \mu_0$ for some $\mu_0 \in \calP_2(\calX)$ and $Y_0$ concentrated at the optimality condition $Y_0 = \nabla f (KX_0)$, then for every $t > 0$ it holds
        $$ \calW_2(\mu_t, \mu^\ast) \le (1+ \lambda^{-1}\omega_{f^\ast}^{-2} \norm{K}^2)^{1/2} \calW_2(\mu_0, \mu^\ast) e^{-\omega_g t} + C_1(\lambda). $$
        \item The $x$-marginal $\mu_\lambda$ of the stationary solution $\pi_\lambda$ of \eqref{eq:ct-pd-diffusion-jointvar} satisfies
        $$ \calW_2(\mu_\lambda, \mu^\ast) \le C_1(\lambda). $$
    \end{enumerate}
\end{theorem}

The idea of this result's proof is a synchronous coupling, similar to the computations in \cref{subsubsec:contraction-overdamped-langevin-diffusion} and \cref{lem:stability_general_dynamics_cont_time}. This time, however, we couple two different processes \eqref{eq:ld-artificial-pd} and \eqref{eq:ct-pd-diffusion-jointvar}. The difference results in an offset term in the Gronwall inequality, which remains bounded under the assumptions, and leads to the bias $C_1(\lambda)$. The second part follows from the first by simply taking the limit $t \to \infty$. The full proof can be found in \cref{subsec:appendix-continuous-time:convergence-to-ula}.

Note that in the above result, the convergence rate is controlled by the constant $\omega_g$. This could lead one to believe that $\omega_{f^*}$ can be arbitrarily small without having an effect on the convergence properties. Inspecting the proof more closely, this is because we shift energy terms corresponding to the dual objective $\nabla f \circ K$ to the bias term $C_1(\lambda)$ (to be precise, by the specific choice of $\delta$ after \eqref{eq:proof-inequality-3}). In the formulation of \cref{thm:ct-convergence-to-overdamped-langevin}, reducing $\omega_{f^*}$ implicitly increases the Lipschitz constant of $\nabla f$, leading to higher variance of the dual objective and therefore increased constants $D_i$. The residual terms $D_i$ all gather in the bias term $C_1(\lambda)$, which grows arbitrarily as $\omega_{f^*}$ shrinks to zero.

As the constants $D_i$ are generally not available in closed form, the result does not give a clear answer to how the convergence behaviour changes under reparametrizations of the potential. For example, if $K$ is invertible, one could sample from the modified density $\exp(-f(u)-g(K^{-1}u))$ and push the samples forward under $K^{-1}$. The effect of this on convergence behaviour of the SDE presents an interesting future direction of research.

\subsection{A Bias-Corrected Modification of Primal-Dual Diffusion}\label{subsec:bias-corrected-diffusion}
In this subsection, we want to follow up on the negative result of \cref{thm:1d-gauss-no-homogeneous-correction}, where we saw that there is no way to construct a diffusion SDE with the primal-dual type drift coefficient and homogeneous, problem-independent diffusion that has the correct stationary distribution. Naturally, we can also consider modifications of the diffusion SDE with problem-dependent and/or non-homogeneous diffusion coefficient, so that the samples become (in continuous time) unbiased.

We aim to correct the stochastic process \eqref{eq:ct-pd-diffusion} by adding terms such that $\pi^\ast = (\id,\nabla f \circ K)_{\#} \mu^\ast $ becomes its stationary solution. Throughout this subsection, we will assume that $g \in C^1(\calX)$ and $f \in C^3(\calY)$, and that $f^\ast$ is differentiable and strictly convex, so that $\nabla f^\ast = (\nabla f)^{-1}$.
We define $M : \bbR^d \to \bbR^{d\times m}$ by $M(x)\coloneqq \nabla (\nabla f \circ K)^T(x) = K^T H_f (Kx)$, where $H_f$ denotes the Hessian of $f$, and $r: \bbR^d \to \bbR^m $ componentwise by $r_j(x) \coloneqq \sum_i \partial_{x_i} M_{ij}(x) = \nabla_x \cdot M_{\cdot j}(x)$, the divergence of the $j$-th column of $M$. Consider now the modified stochastic differential equation
\begin{equation}\label{eq:ct-modified-pd-diffusion-joint-var}
    \rmd Z_t = - A_{\tup{mod}}(Z_t) \rmd t + B_{\tup{mod}}(Z_t) \rmd W_t,
\end{equation}
where the modified drift and diffusion coefficients are given by
\begin{align*}
    A_{\tup{mod}}(z) &\coloneqq%
    \begin{pmatrix}%
        \nabla g(x)+K^Ty\\ 
        \lambda (\nabla f^*(y) - Kx) + M(x)^T (\nabla g(x) + K^T \nabla f(Kx))-r(x)
    \end{pmatrix}\\
    B_{\tup{mod}}(z) &\coloneqq \sqrt 2\, \nabla (\id, \nabla f \circ K)(x) = \sqrt 2 \begin{pmatrix}
    I\\M(x)^T
\end{pmatrix}.
\end{align*}
In contrast to \eqref{eq:ct-pd-diffusion}, $B_{\tup{mod}}$ depends on $x$ and is therefore not a special case of $B_{\tup{gPD}}$. The drift $A_{\tup{mod}}$ contains the same terms as $A_{\tup{PD}}$, and an additional term $- M(x)^T \nabla \log \mu^\ast(x) - r(x)$. The Fokker--Planck equation under non-constant diffusion coefficients like in \eqref{eq:ct-modified-pd-diffusion-joint-var} is
\begin{equation}\label{eq:fp-modifed-pd}
\frac{\partial \pi_t}{\partial t} = \nabla_z \cdot \left(A_{\tup{mod}}(z)\pi_t + \nabla_z \cdot \left(\Sigma_{\tup{mod}}(z) \pi_t\right) \right),
\end{equation}
where $\Sigma_{\tup{mod}}(z) \coloneqq \frac{1}{2} B_{\tup{mod}}(z) B_{\tup{mod}}^T(z)$, see \cite{Pavliotis2014}.

\begin{remark}
    As an intuitive explanation why \eqref{eq:ct-modified-pd-diffusion-joint-var} has the desired stationary distribution, note that we want the stationary solution to concentrate its mass in $\calZ = \bbR^{d+m}$ on the $d$-dimensional manifold $\calU$ described by $y = \nabla f(Kx)$. The tangent space to $\calU$ at any point $(x,\nabla f(Kx))$, denoted in the following by $\mathrm{T}_x\calU$, is spanned by the columns of $\nabla (\id, \nabla f \circ K) = (I,M(x))^T$. This coincides with the column space of $\Sigma_{\tup{mod}}$, hence locally on $\calU$, the diffusion is confined to the manifold. The term $\nabla f^\ast(y) - Kx$ in the drift vanishes upon concentration of the mass on $\calU$; of the remaining terms, $M(x)^T(\nabla g(x) + K^T\nabla f(Kx))$ is just the gradient of the negative log-density $- \nabla \log \mu^\ast(x) = \nabla g(x) + K^T\nabla f(Kx)$ in the primal variable $x$ mapped to the tangent space of $\calU$ by the map $(I,M(x))^T$, i.e. tangential to $\calU$, and by itself would leave $\calU$ invariant just as in overdamped Langevin diffusion. The remaining term $-r(x)$ is not tangential: since $\Sigma_{\tup{mod}}$ varies along $\calU$, the diffusion alone would let mass leak off the manifold to second order, at a rate governed by the curvature of $\calU$, i.e. by the divergence $r(x)$ of the columns of $M$. The term $-r(x)$ exactly cancels this leakage. Locally on $\calU$, \eqref{eq:ct-modified-pd-diffusion-joint-var} coincides with overdamped Langevin diffusion, up to this curvature correction.
\end{remark}

In order to show formally that the partially singular distribution $\pi^\ast$ is a stationary solution of \eqref{eq:ct-modified-pd-diffusion-joint-var}, we have to consider the weak formulation of the corresponding Fokker--Planck equation, given by
\begin{equation}\label{eq:modified-weak-fp}
   \frac{\rmd}{\rmd t} \int \phi \rmd\pi_t = \int \Sigma_{\tup{mod}}(z) : \nabla_z^2 \phi \rmd \pi_t - \int \nabla \phi \cdot A_{\tup{mod}}(z) \rmd \pi_t \quad \forall \phi\in C_c^\infty(\calZ),
\end{equation}
where $\Sigma:\Phi = \sum_{ij} \Sigma_{ij} \Phi_{ij}$ denotes the matrix scalar product.

\begin{theorem}\label{thm:stationary-soln-of-modified-fp}
    The distribution $\pi^\ast \coloneqq (\id, \nabla f \circ K)_{\#}\mu^* $ is a stationary solution of \eqref{eq:modified-weak-fp}.
\end{theorem}

In principle, this theorem might suggest that a time discretization of the modified diffusion \eqref{eq:ct-modified-pd-diffusion-joint-var} may be a better idea than the primal-dual sampling algorithm \eqref{eq:pd_sampling_update}. Note, however, that apart from the higher-order derivative terms in $f$ (which counter the motivation to use a primal-dual scheme in the first place), the term $A_\tup{mod}$ contains just the gradient of the target log-density, i.e. we explicitly add overdamped Langevin diffusion in the dual variable. This suggests that, algorithmically, the dualization is somewhat redundant and any sampler based on \eqref{eq:ct-modified-pd-diffusion-joint-var} would be inferior to a standard unadjusted Langevin algorithm. 

The proof of \cref{thm:stationary-soln-of-modified-fp} can be found in \cref{subsec:appendix-continuous-time:bias-corrected-diffusion}, along with an auxiliary vector calculus identity (\cref{lem:weird-chain-rule}) that is needed in the proof. It consists in plugging the terms into the weak Fokker--Planck equation \eqref{eq:modified-weak-fp} and verifying equality.

\begin{remark}
    While the last result showed that the target is now an invariant solution, the asymptotic stability of the new diffusion process is unclear. In general, it is hard to generalize a contraction result like \cref{lem:stability_general_dynamics_cont_time} to the case of non-homogeneous diffusion coefficients. In section 2 of \cite{Fournier2021}, this question is addressed in more detail for the heat equation with non-homogeneous diffusion. The authors show that it becomes possible to prove stability in a Wasserstein distance with a modified underlying metric, where the metric depends on the diffusion coefficient via an elliptic PDE.\\
    Since the modified equation \eqref{eq:ct-modified-pd-diffusion-joint-var} is of limited numerical interest, we only sketch some main ideas and omit a detailed proof. At every point $x \in \calX$ one can define an inner product $\langle u, v \rangle_x \coloneqq u^T V^{-1}(x) v$ where $V(x)=W(x)W(x)^T$ is symmetric positive definite. The first $d$ columns of $W$ form a basis of the tangent space $\mathrm{T}_x\calU$ of $\calU$ at $x$, and the remaining $m$ columns span the normal space $(\mathrm{T}_x\calU)^\perp$. We then need to equip $\calZ$ with the induced Riemannian metric
    $$ d_{\tup{mod}}(z,\tilde z) = \inf_{\zeta(0)=z,\,\zeta(1)=\tilde z} \int_0^1 \left( \dot\zeta(t)^T V^{-1}(\zeta(t)) \dot \zeta(t) \right)^{1/2} \rmd t. $$
    The difficulty then usually consists in finding a more convenient form of $d_{\tup{mod}}(z,\tilde z)$, which allows one to integrate the Fokker--Planck equation against $d_{\tup{mod}}^2(z,\tilde z)$ in order to show contraction. We leave this to future work.
\end{remark}

\begin{example}
    We can check that in the one-dimensional quadratic case that was considered as an illustrative counterexample in \cref{thm:1d-gauss-example,thm:1d-gauss-no-homogeneous-correction}, the modified equation now has the desired invariant distribution $\pi^\ast$. Furthermore, if $\lambda > k^2 c_f^{-2}$, every solution $\pi_t$ of the Fokker--Planck equation converges to $\pi^\ast$ as $t\to \infty$. This can be seen as follows: We carry out the same computations as in the proof of \cref{thm:1d-gauss-example} to obtain the Fourier transform $p_t$ of the solution $\pi_t$, given by
    \begin{gather*}
        p_t(\zeta) = p_0(\exp(-tA_{\tup{mod}}^T)\zeta) \exp \left( - \frac{1}{2} \int_0^t (B_{\tup{mod}}^T \exp(-sA_{\tup{mod}}^T) \zeta)^2 \rmd s \right),\\
        A_{\tup{mod}} = \begin{pmatrix}
            c_g^{-1} & k\\
            -\lambda k + k c_f^{-1} c_g^{-1} + k^3 c_f^{-2} & \lambda c_f
        \end{pmatrix}, \quad
         B_{\tup{mod}} = \sqrt{2}\begin{pmatrix}
            1\\
            kc_f^{-1}
        \end{pmatrix}.
    \end{gather*}
    This time, the quadratic form under the integral is not positive definite in $\zeta$ since $B_{\tup{mod}}$ is an eigenvector of $A_{\tup{mod}}$ with eigenvalue $v_h \coloneqq c_g^{-1} + k^2c_f^{-1}$. Instead, it holds
    $$\int_0^t (B_{\tup{mod}}^T \exp(-sA_{\tup{mod}}^T) \zeta)^2 \rmd s = (2v_h)^{-1}\left(1-e^{-2tv_h}\right) (B_{\tup{mod}}^T\zeta)^2. $$
    Taking the inverse Fourier transform shows that for every (possibly degenerate) initialization $\pi_0$ the solution $\pi_t$ converges to the invariant distribution, which is a degenerate Gaussian with mean zero and covariance matrix $(2v_h)^{-1} B_{\tup{mod}} B_{\tup{mod}}^T$. This is precisely the target distribution $\pi^\ast$.
\end{example}

%% file: sections/4discretetime.tex
\section{Convergence of the Discrete Time Algorithm}\label{sec:discrete-time}
We want to apply similar techniques as in \cref{sec:continuous-time} to the time discretized setting of the algorithm \eqref{eq:pd_sampling_update}. Indeed, couplings like the ones we employed in the continuous time setting above have been used before in the stability and convergence analysis of Langevin algorithms. We begin this section by recalling such a computation in the simple case of ULA as an instructive example, and then expand to the primal-dual algorithm.

For a convex functional $h : \calX \to \bbR \cup \{\infty \}$, points $x,\tilde x \in \calX$ and subgradients $p\in \partial h(x), \tilde p \in \partial h(\tilde x)$, we will use the symmetric Bregman distance $D_{h,\text{sym}}^{p,\tilde p}(x,\tilde x) = \inpro{x-\tilde x}{p - \tilde p}$. If $h$ is differentiable with unique subgradients, we will drop the reference to $p,\tilde p $ and just write $D_{h,\text{sym}}(x,\tilde x)$. Recall that, if $h$ is convex, $D_{h,\text{sym}}$ is always non-negative.

Suppose now that the potential $h : \bbR^d \to \bbR$ is $\omega_h$-strongly convex and differentiable, and $\nabla h$ is $L_{\nabla h}$-Lipschitz continuous. We consider two coupled Markov chains $(X^n)_n, (\tilde X^n)_n$ defined by initial values $X^0 \sim \mu^0, \tilde X^0 \sim \tilde \mu^0$ and
\begin{equation*}
    \left\{\begin{array}{l}
    \xi^{n+1} \sim \mathrm{N}(0,I_d) \vspace{0.4em}\\
    X^{n+1} = X^n - \tau \nabla h (X^n) + \sqrt{2\tau}\xi^{n+1} \vspace{0.4em}\\
    \tilde X^{n+1} = \tilde X^n - \tau \nabla h (\tilde X^n) + \sqrt{2\tau}\xi^{n+1}.
    \end{array}\right.
\end{equation*}
Then we can compute that (see \cite[Prop.~3]{Durmus2019}, where this calculation has been carried out)
\begin{align*}
    \norm{X^{n+1}-\tilde X^{n+1}}^2 &= \norm{X^n - \tau \nabla h(X^n) - \tilde X^n + \tau \nabla h (\tilde X^n)}^2\\
    &= \norm{X^{n}-\tilde X^{n}}^2 - 2\tau D_{h,\textrm{sym}}(X^n,\tilde X^n) + \tau^2 \norm{\nabla h (X^n) - \nabla h (\tilde X^n)}^2.
\end{align*}
We can use the following bound for strongly convex functions with Lipschitz continuous gradient \cite[Thm 2.1.12]{Nesterov2004}
$$ D_{h,\textrm{sym}}(X^n,\tilde X^n) \ge \frac{\omega_h L_{\nabla h}}{\omega_h + L_{\nabla h}} \norm{X^n-\tilde X^n}^2 + \frac{1}{\omega_h + L_{\nabla h}}\norm{\nabla h(X^n) - \nabla h(\tilde X^n)}^2, $$
and we obtain
\begin{equation*}
    \norm{X^{n+1}-\tilde X^{n+1}}^2 \le \left(1-\tau \frac{2\omega_h L_{\nabla h}}{\omega_h + L_{\nabla h}}\right) \norm{X^n-\tilde X^n}^2 + \tau \left( \tau - \frac{2}{\omega_h+L_{\nabla h}} \right) \norm{\nabla h(X^n) - \nabla h(\tilde X^n)}^2.
\end{equation*}
If the step size is chosen to be $\tau = 2/(\omega_h+L_{\nabla h})$, we obtain the convergence rate
\begin{equation*}
    \norm{X^{n}-\tilde X^{n}}^2 \le \left( \frac{L_{\nabla h}-\omega_h}{L_{\nabla h}+\omega_h} \right)^{2n} \norm{X^0-\tilde X^0}^2.
\end{equation*}
Denoting the distribution of $X^n$ by $\mu^n$ and that of $\tilde X^n$ by $\tilde \mu^n$, we obtain by taking the infimum over all couplings of $\mu^0$ and $\tilde \mu^0$ on the right side that
\begin{equation*}
    \calW_2(\mu^n,\tilde \mu^n) \le \left( \frac{L_{\nabla h}-\omega_h}{L_{\nabla h}+\omega_h} \right)^{n} \calW_2(\mu^0,\tilde \mu^0).
\end{equation*}
Since $(\calP_2(\bbR^d), \calW_2)$ is complete, $\mu^n, \tilde{\mu}^n$ converge to a unique stationary distribution $\mu_\tau$ in $\calW_2$ distance. We concentrate on Wasserstein convergence in the present work, but convergence in e.g. total variation distance can be derived \cite{Durmus2019}. Similar computations can be carried out for other time-discretizations of Langevin diffusion, e.g. time-implicit ones, leading to proximal sampling schemes \cite{Salim2020,Durmus2019a}. Since the stationary distribution $\mu_\tau$ of the time-discrete scheme is not the target distribution $\mu^\ast$, the missing step to relate the Markov chain samples to the target is to prove an error bound using a coupling of the time-discrete chain and the time-continuous process. Typically, this is more challenging and the introduced bias (in Wasserstein distance) between the two scales with a power of the step size $\tau$ \cite{Dalalyan2017a,Dalalyan2019,Durmus2019a,Salim2020,Lau2024}.

\subsection{Stability of Primal-Dual Langevin Sampling}\label{subsec:contraction-primal-dual-langevin}
We now turn to the primal-dual setting of \eqref{eq:pd_sampling_update}. Consider the following coupling in discrete time with initial values $(X^0,Y^0)\sim\pi^0$ and $(\tilde X^0, \tilde Y^0) \sim \tilde \pi^0$ and identical stochastic input to both iterations:
\begin{equation}\label{eq:coupling-ulpda}
    \left\{\begin{array}{l}
        \xi^{n+1} \sim \mathrm N(0,I_d)  \vspace{0.8em} \\
        Y^{n+1} = \prox_{\sigma f^\ast}(Y^n + \sigma K X_\theta^n) \\
        X^{n+1} = \prox_{\tau g}(X^n - \tau K^T Y^{n+1}) + \sqrt{2\tau}\xi^{n+1} \\
        X_\theta^{n+1} = X^{n+1} + \theta(X^{n+1} - X^n) \vspace{0.8em} \\
        \tilde Y^{n+1} = \prox_{\sigma f^\ast}(\tilde Y^n + \sigma K \tilde X_\theta^n) \\
        \tilde X^{n+1} = \prox_{\tau g}(\tilde X^n - \tau K^T \tilde Y^{n+1}) + \sqrt{2\tau} \xi^{n+1} \\
        \tilde X_\theta^{n+1} = \tilde X^{n+1} + \theta (\tilde X^{n+1} - \tilde X^n).
    \end{array}\right.
\end{equation}
By convention, we set $X^{-1} = X^0$ and $\tilde X^{-1} = \tilde X^0$. In the following, we denote the distribution of $(X^n,Y^n)$ by $\pi^n$ and the $X^n$-marginal by $\mu^n$, and analogously by $\tilde\pi^n, \tilde\mu^n$ those of $(\tilde X^n,\tilde Y^n)$. Denote further $L \coloneqq \norm{K}$ where $\norm{\cdot}$ is the operator norm induced by the Euclidean norm on $\calX$ and $\calY$. We assume that also in the discrete setting, the step sizes are coupled via the relation $\lambda = \sigma /\tau$.

We start by giving a stability result for the case where $g$ and $f^\ast$ are merely convex (not necessarily strongly convex).
\begin{theorem}\label{thm:dt-stability-convex}
    Let $(X^n,Y^n)$ and $(\tilde X^n, \tilde Y^n)$ be given by the coupling \eqref{eq:coupling-ulpda} with overrelaxation parameter $\theta = 1$ and $\tau \sigma L^2 \le 1$. Then for the primal variable it holds
    $$ \calW_2(\mu^N, \tilde \mu^N) \le \calT_{1, \lambda^{-1}}(\pi^0, \tilde \pi^0), $$
    where $\calT_{1,\lambda^{-1}}$ is the transport distance over the joint primal-dual variable defined in \eqref{eq:weighted-transport-distance}. If $\tau\sigma L^2 < 1$, then the joint variable obeys 
    $$ \calT_{1, \lambda^{-1}}(\pi^N, \tilde \pi^N) \le C \calT_{1, \lambda^{-1}}(\pi^0, \tilde \pi^0), $$
    with the constant $ C= (1-\tau\sigma L^2)^{-1/2}$.
\end{theorem}

In the case of strongly convex $g$ and $f^\ast$, we can prove the following contraction result.
\begin{theorem}\label{thm:dt-contraction-strongly-convex}
    Assume that $g$ and $f^\ast$ are $\omega_g$- and $\omega_{f^\ast}$-strongly convex respectively. Let $(X^n,Y^n)$ and $(\tilde X^n, \tilde Y^n)$ be given by the coupling \eqref{eq:coupling-ulpda} with overrelaxation parameter $\theta$ satisfying 
    $$\max\{(1+2\omega_g \tau)^{-1},(1+2\omega_{f^\ast}\sigma)^{-1}\} \le \theta < 1,$$ and further $\theta \sigma \tau L^2 \le 1$. Then with the constant $\tilde \lambda \coloneqq \frac{1+2\omega_g \tau}{1+2\omega_{f^\ast}\sigma} \, \lambda$, it holds
    $$\calW_2(\mu^N,\tilde \mu^N) \le \theta^{N/2} \calT_{1,\tilde \lambda^{-1}}(\pi^0,\tilde \pi^0) $$
    and if $\theta \sigma \tau L^2 < 1$, then with $ \tilde C \coloneqq (1-\theta \tau \sigma L^2)^{-1/2}$ we have for the joint variables
    $$\calT_{1,\tilde \lambda^{-1}}(\pi^N,\tilde \pi^N) \le \tilde C \theta^{N/2} \calT_{1,\tilde \lambda^{-1}}(\pi^0,\tilde \pi^0). $$
    It follows directly that in the latter case a stationary distribution $\pi_{\lambda,\tau}$ exists and the distribution $\pi^N$ of samples converges to $\pi_{\lambda,\tau}$ as $N \to \infty$.
\end{theorem}

The proofs of \cref{thm:dt-stability-convex,thm:dt-contraction-strongly-convex} can be found in \cref{subsec:appendix-discrete-time:contraction}. In order to prove both results, we derive an auxiliary inequality akin to a descent lemma. Both proofs then result from iterating the descent inequality over $N$ steps and are in spirit very similar to existing convergence proofs of primal-dual optimization algorithms \cite{Chambolle2011,Chambolle2018}.

Having proven the stability theorems for the iteration \eqref{eq:pd_sampling_update}, we briefly comment on the version of the algorithm that adds the stochastic term $\xi$ inside the argument of the proximal operator.

\begin{corollary}\label{cor:stability-modified-algorithm}
    Instead of \eqref{eq:pd_sampling_update}, consider the iteration procedure 
    \begin{equation}\label{eq:pd_sampling_update_xi_in_prox}
        \left\{\begin{array}{l}
            \xi^{n+1} \sim \mathrm N(0,I_d) \\
            Y^{n+1} = \prox_{\sigma f^\ast}(Y^n + \sigma K X_\theta^n) \\
            X^{n+1} = \prox_{\tau g}(X^n - \tau K^T Y^{n+1} + \sqrt{2\tau}\xi^{n+1}) \\
            X_\theta^{n+1} = X^{n+1} + \theta(X^{n+1} - X^n)
        \end{array}\right.
    \end{equation}
    with the stochastic term $\sqrt{2\tau} \xi$ added inside the proximal mapping. For the iterates of this algorithm, under the same assumptions, the results of \cref{thm:dt-stability-convex,thm:dt-contraction-strongly-convex} also hold. Under the assumptions of \cref{thm:dt-contraction-strongly-convex} with $\theta \sigma \tau L^2 < 1$, the algorithm therefore has a stationary distribution $\tilde{\pi}_{\lambda, \tau}$.
\end{corollary}

The proof can be found in \cref{subsec:appendix-discrete-time:contraction}.

Note that this result does not state anything about the relation between the stationary distributions $\tilde{\pi}_{\lambda, \tau}$ of \eqref{eq:pd_sampling_update_xi_in_prox} and $\pi_{\lambda,\tau}$ of \eqref{eq:pd_sampling_update}; they will generally be different distributions. To see this, consider an example where $g = \norm{\cdot}_1 + \frac{\omega_g}{2}\norm{\cdot}^2$ (and $f^\ast$ strongly convex) has the proximal map $\prox_{\tau g}(x)_j = (1+\tau\omega_g)^{-1}\sign(x_j)\max(\abs{x_j}-\tau, 0)$. Then some sets of Lebesgue measure zero like the singleton $\{0\}$ necessarily have positive mass under the law of $X^{n+1}$ for all $n\in\mathbb{N}_0$ in \eqref{eq:pd_sampling_update_xi_in_prox}. Plugging in the stationary distribution $\tilde{\pi}_{\lambda, \tau}$ implies that such sets have positive mass under $\tilde{\pi}_{\lambda, \tau}$. Therefore $\tilde{\pi}_{\lambda, \tau}$ does not have a Lebesgue density.\\
Under the iteration \eqref{eq:pd_sampling_update}, the samples are always the result of a diffusion step applied to the output of the proximal mapping, so samples can not be concentrated on a single value. A direct bound of the distance of the two distributions can likely be proven, but is outside the scope of the present work.

\subsection{Convergence to the Continuous Time Stationary Solution}\label{subsec:discrete-time:convergence}
Having shown that the algorithm is stable with a limiting stationary distribution $\pi_{\lambda,\tau}$ in the strongly convex case, we need to find bounds on the distance between $\pi_{\lambda,\tau}$ and the continuous time stationary solution $\pi_\lambda$, in order to relate the drawn samples to the target. We approach this by considering again a coupling, but this time between the algorithm \eqref{eq:pd_sampling_update} and the stationary state of its continuous time limiting SDE \eqref{eq:ct-pd-diffusion-jointvar}. 

Assume that the time steps $\tau,\sigma$ are constant with $\lambda = \sigma/\tau$. Let $(X^0,Y^0) \sim \pi^0$ for some initial distribution $\pi^0$, and set $X_\theta^0 = X^0$. Now consider a process $\tilde Z_t = (\tilde X_t,\tilde Y_t)$, $t \in \bbR_{\ge 0}$ which is initialized at $\tilde Z_0 \sim \pi_\lambda$. Let $W_t$ be $\bbR^d$-valued Brownian motion and denote $\xi^{n+1} := \tau^{-1/2} \int_{n\tau}^{(n+1)\tau} \mathrm d W_t$ for all $n \in \bbN_0$. We now consider $Z^n = (X^n,Y^n)$, $n \in \bbN_0$ and $\tilde Z_t = (\tilde X_t,\tilde Y_t)$, $t \in \bbR_{\ge 0}$ defined by the coupling
\begin{equation}\label{eq:coupling-ulpda-SDE}
    \left\{\begin{array}{l}
        Y^{n+1} = \prox_{\sigma f^\ast}(Y^n + \sigma K X_\theta^n) \\
        X^{n+1} = \prox_{\tau g}(X^n - \tau K^T Y^{n+1}) + \sqrt{2\tau} \xi^{n+1} \\
        X_\theta^{n+1} = X^{n+1} + \theta(X^{n+1} - X^n) \vspace{0.8em} \\
        \rmd \tilde X_t = -(\nabla g(\tilde X_t) + K^T\tilde Y_t) \rmd t + \sqrt{2} \rmd W_t\\
        \rmd \tilde Y_t = - \lambda (\nabla f^\ast(\tilde Y_t) - K \tilde X_t) \rmd t.
    \end{array}\right.
\end{equation}
It immediately follows that the marginal distribution of $\tilde Z_t$ is stationary with $\tilde Z_t \sim \pi_\lambda$, $t \in \bbR_{\ge 0}$. For notational convenience, we also define an overrelaxation variable in continuous time by $\tilde X_{t,\theta} := \tilde X_t + \theta (\tilde X_t - \tilde X_{t-\tau})$ and set $\tilde X_{t-\tau} = \tilde X_0$ when $t \in [0,\tau)$.
\begin{theorem}\label{thm:convergence-to-pi-lambda}
    Assume that $g \in C^1(\calX)$ and $\nabla g$ is Lipschitz continuous with constant $L_{\nabla g} > 0$, that $f$ is $\omega_f$-strongly convex with $\omega_f > 0$, and that $g, f^\ast$ are $\omega_g$- and $\omega_{f^\ast}$-strongly convex, respectively. Then $\nabla f^\ast$ is $\omega_f^{-1}$-Lipschitz continuous, so that $A_{\tup{PD}}$ grows at most linearly, and since $\pi_\lambda \in \calP_2(\calZ)$ by \cref{thm:convergence-pd-ct}, the constant $D_4 \coloneqq \bbE_{Z \sim \pi_\lambda} \left[ \norm{A_{\tup{PD}}(Z)}^2 \right]$ is finite. Let the step sizes $\sigma,\tau$ and the overrelaxation parameter $\theta$ be chosen such that $\theta \sigma \tau L^2 \le 1$ and 
    $$\max\{(1+\omega_g \tau)^{-1},(1+\omega_{f^\ast}\sigma)^{-1}\} \eqqcolon \theta_{\min} \le \theta < 1.$$
    Let the iterates $(X^n,Y^n)$ be generated by \eqref{eq:pd_sampling_update}. Denoting by $\mu^n$ the marginal distribution of $X^n$, we have
    \begin{equation*}
        \calW_2(\mu^N, \mu_\lambda) \le \theta^{N/2} \calT_{1,\bar \lambda^{-1}}(\pi^0,\pi_\lambda) + C_2(\tau, \theta).
    \end{equation*}
    If further $\theta \sigma \tau L^2 < 1$, then the joint distribution $\pi^n$ of $(X^n,Y^n)$ satisfies
    \begin{equation*}
        \calT_{1,\bar\lambda^{-1}}(\pi^N, \pi_\lambda) \le \frac{\theta^{N/2}}{(1-\theta \sigma \tau L^2)^{1/2}} \calT_{1,\bar \lambda^{-1}}(\pi^0,\pi_\lambda) + \frac{C_2(\tau, \theta)}{(1-\theta \sigma \tau L^2)^{1/2}},
    \end{equation*}
    with the constants
    $$\bar \lambda \coloneqq \frac{1+\omega_g \tau}{1+\omega_{f^\ast}\sigma}\lambda, \quad C_2(\tau, \theta) \coloneqq \frac{1}{(1-\theta)^{1/2}}\left((4 d L_{\nabla g}^2 \omega_g^{-1} + 6 d L^2 (1+2\theta^2)\omega_{f^\ast}^{-1})^{1/2} \tau + \calO(\tau^{2})\right).$$
    Note that $\theta_{\min}$ depends on $\tau$. Assuming the minimal choice $\theta = \theta_{\min}$ to achieve the optimal convergence rate, it holds $1-\theta = \min\{\omega_g, \omega_{f^*} \lambda\} \tau + \calO(\tau^2)$ as $\tau \to 0$. This implies $C_2(\tau) \coloneqq C_2(\tau,\theta_{\min}) = \calO(\tau^{1/2})$.
\end{theorem}

The proof of \cref{thm:convergence-to-pi-lambda} can be found in \cref{subsec:appendix-discrete-time:convergence}.
This result is ``typical'' for Langevin algorithms in the sense that it uses a standard condition on the drift and results in the same order of bias as conventional, primal Langevin algorithms:
\begin{remark}
    The finiteness of $D_4$ corresponds to a standard assumption on the drift in Langevin sampling algorithms. In the case of overdamped Langevin diffusion, it is satisfied if the drift coefficient, i.e. $\nabla h$ in \eqref{eq:Langevin_diffusion_gradient}, is Lipschitz continuous, see Lemma 2 of \cite{Dalalyan2017a}.\\
    In the case of primal-dual diffusion, it puts growth requirements on $\nabla g$ and $\nabla f^\ast$ instead: if there are constants $\alpha,\beta$ such that $\norm{\nabla g(x)} \le \alpha + \beta \norm{x}$ and $\norm{\nabla f^\ast(y)} \le \alpha + \beta \norm{y}$ for all $x\in\calX$ and $y\in\calY$, then $D_4 < \infty$. Under the assumptions of \cref{thm:convergence-to-pi-lambda}, this holds automatically.
\end{remark}

\begin{remark}
    The Wasserstein bias term of order $\calO(\tau^{1/2})$ between discretized scheme and continuous time limit is the typical order known from purely primal Langevin algorithms under strong log-concavity \cite{Dalalyan2019,Durmus2019a,Salim2020}.
\end{remark}

Apart from the bounded drift condition, the strong convexity and Lipschitz assumptions in \cref{thm:convergence-to-pi-lambda} restrict the target distribution $\mu^*$ to be supported everywhere on $\mathbb{R}^d$. This is in contrast to existing theory and practice of primal-dual methods in optimization: Due to their proximal nature, convergence guarantees for these optimization algorithms can be derived for suitable objectives on bounded domains \cite{Chambolle2016}. Extending the theory for sampling to constrained domains $\calX \subsetneq \mathbb{R}^d$ likely requires more work, as the underlying SDE would have to change. This is an interesting future direction of research, as recent works have shown that primal Langevin algorithms can be restricted to some convex, constrained domains either by approximating the target or by modifying the SDE to a boundary-reflected process \cite{Melidonis2023,Klatzer2026}. Our numerical results in \cref{subsec:numerics:inpainting} suggest that primal-dual sampling appears to be effective for a target on the constrained domain $\mathbb{R}^d_{\ge 0}$.

Under the assumptions of \cref{thm:convergence-to-pi-lambda} with $\theta\sigma\tau L^2 < 1$, \cref{thm:dt-contraction-strongly-convex} shows that the discrete time algorithm converges to its stationary distribution $\pi_{\lambda, \tau}$ as $N\to \infty$. From the nonasymptotic bounds in \cref{thm:convergence-to-pi-lambda}, we obtain the following bound between the invariant measures in discrete and continuous time.
\begin{corollary}\label{cor:distance-invariants-discrete-to-continuous}
    Let the assumptions of \cref{thm:convergence-to-pi-lambda} be satisfied with $\theta \sigma \tau L^2 < 1$ and $\theta = \theta_{\min}$. Then, denoting by $\mu_{\lambda,\tau}$ the $x$-marginal of $\pi_{\lambda,\tau}$, the following bounds hold:
    \begin{align*}
        \calW_2(\mu_{\lambda,\tau}, \mu_\lambda) &\le C_2(\tau, \theta_{\min}) = \calO(\tau^{1/2}),\\
        \calT_{1,\bar \lambda^{-1}}(\pi_{\lambda,\tau},\pi_\lambda) &\le \frac{C_2(\tau, \theta_{\min})}{(1-\theta \sigma \tau L^2)^{1/2}} = \calO(\tau^{1/2}).
    \end{align*}
\end{corollary}

\Cref{thm:convergence-to-pi-lambda} provides a bound between the samples produced by the algorithm and its continuous time limit. While this result is typical in the Langevin algorithms literature, we here have to combine this with \cref{thm:ct-convergence-to-overdamped-langevin} in order to get a relation between samples and target distribution. The following result closes this final gap.
\begin{corollary}\label{cor:distance-samples-to-target}
    Let the assumptions of \cref{thm:ct-convergence-to-overdamped-langevin,thm:convergence-to-pi-lambda} be satisfied with $\theta \sigma \tau L^2 < 1$. Let $\pi^n$ be the distribution of samples $(X^n,Y^n)$ generated by \eqref{eq:pd_sampling_update} with $\theta=\theta_{\min}$. Then the $x$-marginal $\mu^n$ of $\pi^n$ obeys
    $$ \calW_2(\mu^N, \mu^\ast) \le \theta^{N/2} \calT_{1,\bar \lambda^{-1}}(\pi^0, \pi_\lambda) + C_1(\lambda) + C_2(\tau), $$
    with the constants $C_1, C_2$ from \cref{thm:ct-convergence-to-overdamped-langevin,thm:convergence-to-pi-lambda}, respectively, which decay asymptotically like
    \begin{align*}
        C_1(\lambda) &= D_5 \lambda^{-1/2} + \calO(\lambda^{-3/2}), \quad \textrm{as } \lambda \to \infty,\\
        C_2(\tau) &= D_6 \tau^{1/2} + \calO(\tau^{3/2}), \quad \textrm{as } \tau \to 0,
    \end{align*}
    with constants gathered in $D_5, D_6$. It immediately follows that $\calW_2(\mu_{\lambda,\tau}, \mu^\ast) \le C_1(\lambda) + C_2(\tau)$.
\end{corollary}

Before moving on, let us comment on the implication of this result for numerical practice.
\begin{remark}
    The final form of the bias $C_1(\lambda) + C_2(\tau)$ exhibits two practical trade-offs. One is typical for Langevin samplers: Ideally, $\tau$ is large so that subsequent samples are less correlated, but this at the same time increases the discretization bias $C_2(\tau)$, worsening the target approximation. Secondly, and specific to the primal-dual scheme, $\lambda$ should be large, in order to guarantee small dualization bias, but $\sigma = \lambda \tau$ is bounded since we have to control $\sigma \tau L^2$ in order to guarantee stability of the iterates.\\
    The constants $D_i$ controlling the bias in \cref{thm:ct-convergence-to-overdamped-langevin} are generally unknown, so we use numerical heuristics for the choice of $\lambda$ and ablate on this parameter in several of the experiments below. Given $\lambda$, when $L$ is known, we set $\tau = 1/(\sqrt{\lambda} L)$ and $\sigma = \sqrt{\lambda} / L$ as the biggest step sizes that achieve guaranteed stability. When $\omega_g, \omega_{f^*}$ are known, the limited budget of discarded burn-in iterations implies we should choose $\theta = \theta_{\min}$ in order to achieve rapid convergence to the target. These are of course limited heuristics motivated by our convergence results, and further studies on optimal parameter choices are needed.
\end{remark}

In the following, final result, we comment on how the previous statements transfer to the modified scheme \eqref{eq:pd_sampling_update_xi_in_prox}, which adds the random Gaussian increment before the proximal step.
\begin{corollary}\label{cor:distances-modified-algorithm}
    Let $(X^n,Y^n)$ be generated by the iteration \eqref{eq:pd_sampling_update_xi_in_prox}, where the stochastic term $\xi$ is added in the primal variable before applying the proximal mapping. By \cref{cor:stability-modified-algorithm}, this modified algorithm has a stationary state $\tilde{\pi}_{\lambda, \tau}$ under suitable assumptions. For this distribution and the samples generated by the algorithm, the results of \cref{thm:convergence-to-pi-lambda} and \cref{cor:distance-invariants-discrete-to-continuous,cor:distance-samples-to-target} hold unchanged.
\end{corollary}

The proof can be found in \cref{subsec:appendix-discrete-time:convergence}.

%% file: sections/5numerics.tex
\section{Numerical Examples}\label{sec:numerics}
In the following, we demonstrate our results on the primal-dual Langevin sampling scheme on several numerical examples, first small-scale ones to demonstrate the convergence behaviour in detail and subsequently larger imaging inverse problems. We always simulate \eqref{eq:ct-pd-diffusion-jointvar} in a time interval $[0,T]$ by running the discretized algorithm \eqref{eq:pd_sampling_update} with step sizes $\tau,\sigma$ for $T/\tau$ steps. The step sizes are coupled via $\lambda \tau = \sigma$ and are chosen small enough so that $\tau\sigma L^2 \le 1$, i.e. the assumptions of \cref{thm:dt-stability-convex} are satisfied and the continuous-time dynamics are approximated well.

\subsection{Revisiting the Quadratic Case}
We start by illustrating the toy example of two one-dimensional quadratic potentials that was considered in \cref{thm:1d-gauss-example,thm:1d-gauss-no-homogeneous-correction}. Recall that the target distribution has density $\exp(-h(x))/Z$ with the potential $h$ defined as
$$ h(x) = f(kx) + g(x) = \frac{k^2}{2c_f} x^2 + \frac{1}{2c_g} x^2 = \frac{c_g k^2 + c_f}{2c_fc_g} x^2, $$
hence the target distribution is Gaussian $\mathrm{N}(0,c_fc_g(c_gk^2+c_f)^{-1})$. For the numerical simulation, we choose $c_f = 1, c_g = 2, k=1.5$.
We simulate the primal-dual Langevin dynamics \eqref{eq:ct-pd-diffusion} for different values of $\lambda$ by running \eqref{eq:pd_sampling_update}. We set $\theta = 1$ and set the step sizes $\tau$ and $\sigma$ via the relations $\sigma/\tau = \lambda$ and $\sigma \tau k^2 = c$ where $c\le 1$ is required for stability. We set $c=10^{-4}$ in order to approximate closely the continuous time dynamics, resulting in step sizes $\tau = 10^{-2} /(\sqrt{\lambda}\, k)$ and $\sigma = 10^{-2} \sqrt{\lambda}/k$. \Cref{fig:1dgaussian-toy-joint-distributions} shows the joint distribution of primal and dual samples $(X^n,Y^n)$ produced by \eqref{eq:pd_sampling_update}. As $\lambda$ grows, the joint distribution of primal and dual variable concentrates along the linear subspace $y = kx/c_f$ and the primal marginal distribution converges to the target.

\begin{figure}[ht]
    \centering
    \begin{subfigure}{0.34\linewidth}%
        \includegraphics[width=\linewidth]{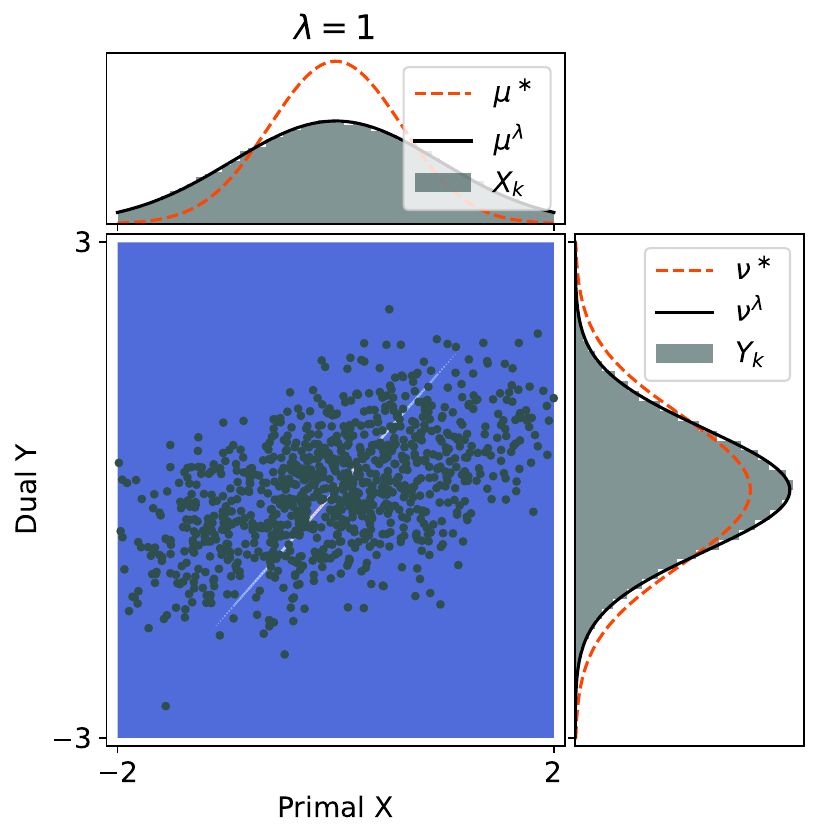}
    \end{subfigure}%
    \begin{subfigure}{0.325\linewidth}%
        \includegraphics[width=\linewidth]{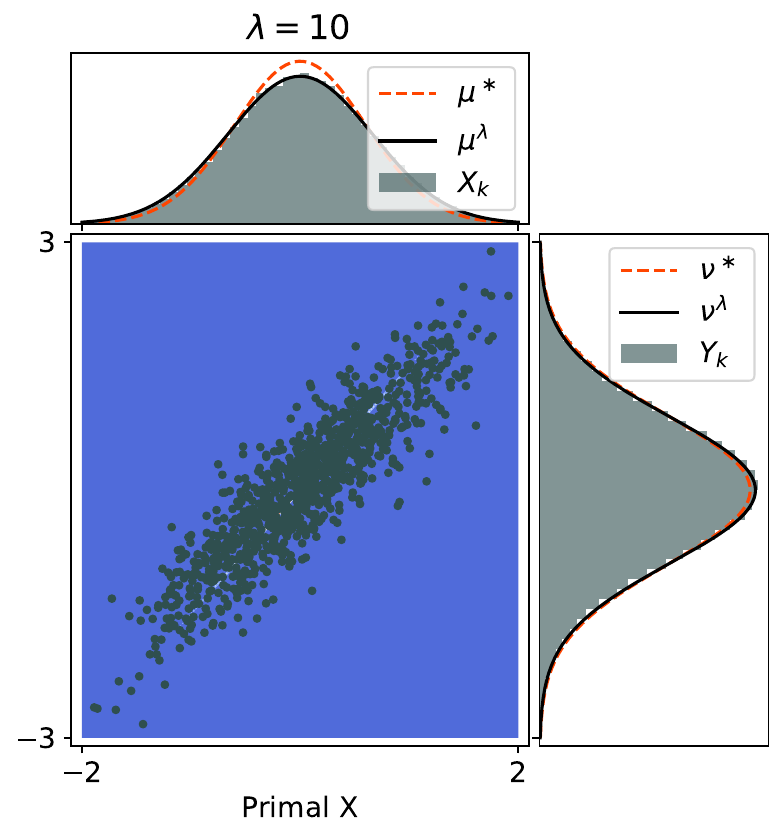}
    \end{subfigure}%
    \begin{subfigure}{0.325\linewidth}%
        \includegraphics[width=\linewidth]{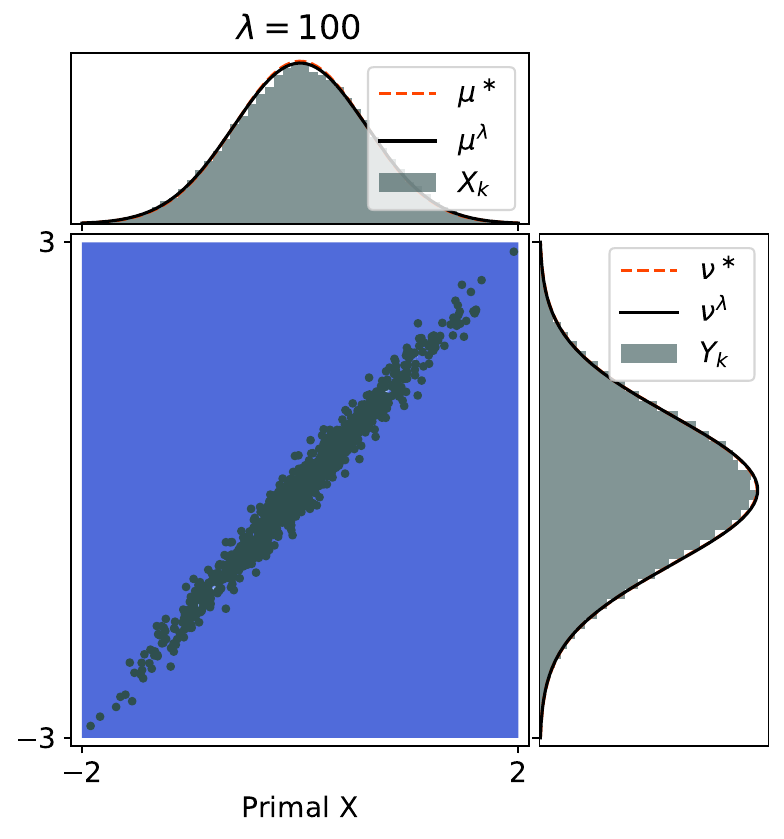}
    \end{subfigure}%
    \caption{Approximation of the target $\mu^\ast = \exp(-h)$ for a one-dimensional quadratic $h$ by \eqref{eq:pd_sampling_update} and its continuous time diffusion limit \eqref{eq:ct-pd-diffusion}. From left to right, results for $\lambda = 1$, $\lambda =10$ and $\lambda = 100$ are displayed. The center scatter plot shows pairs $(X^n,Y^n)$, the background color indicates the mass of the concentrated joint target $\pi^\ast$ (zero is blue, the mass is concentrated along the white line $y = kx/c_f$). The marginal boxes show the densities of the primal and dual targets $\mu^\ast$ and $\nu^\ast = (\partial f \circ K)_{\#}\mu^\ast$ (dashed red), the marginals $\mu^{\lambda}, \nu^{\lambda}$ of the invariant solution of primal-dual diffusion (black solid) and samples drawn using \eqref{eq:pd_sampling_update} (grey histograms). For increasing $\lambda$, the joint distribution of samples tends to concentrate and the primal marginal converges to the target, as proven in \cref{thm:ct-convergence-to-overdamped-langevin}.}
    \label{fig:1dgaussian-toy-joint-distributions}
\end{figure}

\subsection{Total Variation Based Image Denoising from Gaussian Noise}
For the potential $f\circ K$ that is dualized in the algorithm, we now consider the total variation (TV) functional, a typical model-based prior in imaging inverse problems. The (isotropic) TV functional is defined by 
$$\TV(x) := \sum_{i,j} \sqrt{(\nabla_{h}x)_{i,j}^2 + (\nabla_{v}x)_{i,j}^2},$$
where $\nabla_h, \nabla_v$ denote horizontal and vertical finite differences in the pixelated image $x$. We consider a Gibbs prior on $x$ with density $\exp(-\alpha \TV(x))$ where $\alpha$ is a fixed regularization parameter (Note that this only defines a probability distribution on a subspace orthogonal to constant functions since TV has a one-dimensional null space. With a log-likelihood that grows fast enough at infinity along this null space, the posterior is however well-defined.). The corresponding potential term can be written in the form $\alpha \TV(x) = \alpha \norm{Kx}_{2,1} = f(Kx)$ where $K = (\nabla_h,\nabla_v)$ and $f=\alpha \norm{\cdot}_{2,1}$ denotes a constant multiple of the nested $l_2$-$l_1$-norm (inner $l_2$ over vertical/horizontal component, outer $l_1$ over all pixels). We further assume that we make a noisy observation $\tilde x = x + \epsilon$ of $x$, where the noise is distributed like $\epsilon \sim \mathrm{N}(0,\sigma_\epsilon^2 I)$ for some observation variance $\sigma_\epsilon^2$. The corresponding negative log-likelihood is 
$$g(x) := l_{\tilde x}(x) = \frac{1}{2\sigma_\epsilon^2} \norm{x-\tilde x}_2^2,$$
which is $\sigma_\epsilon^{-2}$-strongly convex. We now want to sample from the posterior $\mu^\ast$, which by Bayes law has density proportional to $ \exp(-\alpha \norm{Kx}_{2,1} - l_{\tilde x}(x)) = \exp(-f(Kx)-g(x))$.

\subsubsection{Toy Example in Two Dimensions}
We start with a small-scale example in which we are able to compute accurate approximations of Wasserstein distances between the samples' distribution and the true posterior. Consider the artificial case where the `image' consists of $2\times 1$ pixels, i.e. $\calX = \bbR^2$. Then the operator $K$ reduces to $Kx = x_2-x_1 \in \bbR = \calY$ with $\norm{K} = \sqrt{2}$, and $f$ is given by $f(y) = \alpha \abs{y}$. The convex conjugate $f^\ast$ is the indicator function of the interval $[-\alpha,\alpha]$ and the proximal mapping $\prox_{\sigma f^\ast}$ is, for any $\sigma>0$, the projection onto $[-\alpha,\alpha]$.

The distribution $\mu_n$ of iterate $X^n$ in \eqref{eq:pd_sampling_update} is approximated by running finitely many, here $10^4$, independent Markov chains in parallel and using the resulting empirical distribution $\hat \mu_n$. In \cref{fig:l2tv-2d-scatterplots}, we show scatter plots of the samples at time $T$, where $T$ is chosen large enough that the empirical measures $\hat \mu_n$ remain stationary both qualitatively and in Wasserstein distance to the target (see below). As comparison, we show samples drawn using the algorithm Prox-Sub proposed in \cite{Habring2024}. Prox-Sub solves the optimality condition $y \in \partial f(Kx)$ exactly at every step (by assuming that the subgradient can be evaluated), and hence corresponds exactly to the concentrated limit $\lambda = \infty$ of \eqref{eq:pd_sampling_update} explored in \cref{subsec:convergence-to-ula}. The authors of \cite{Habring2024} showed the convergence of the distribution of Prox-Sub samples to the target in 2-Wasserstein distance, total variation distance and Kullback--Leibler divergence, where the bias depends as usual on the step size of the algorithm.

\begin{figure}[ht]
    \centering
    \includegraphics[width=\linewidth]{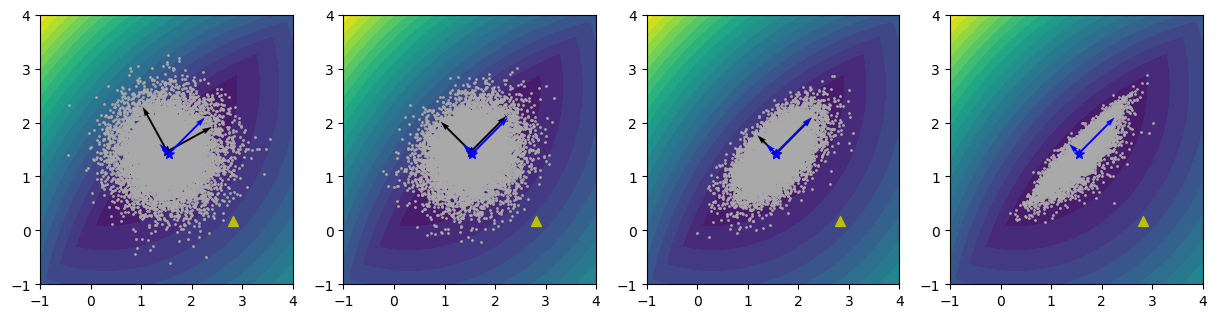}
    \caption{Scatter plots of $\hat \mu_n$ in the two-dimensional TV-denoising experiment. From left to right, the plots show approximations of the same posterior with values $\lambda = 10^1,10^2,10^3$ and `$\lambda = \infty$' (Prox-Sub). Horizontal and vertical axes correspond to the two coordinates in the primal variable $x \in \bbR^2$, the dual variable $y$ is not displayed. Grey dots show the samples at time $T$ produced by $10^4$ parallel chains, the black arrows indicate the samples' mean and the principal components of their covariance (i.e. eigenvectors of the covariance matrix scaled in length by the corresponding eigenvalues). For reference, the blue arrows are the principal components of the posterior that should be approximated. The noisy observation $\tilde x$ is marked by a yellow triangle and the background color is a contour plot of the posterior density $\mu^\ast(x)$.}
    \label{fig:l2tv-2d-scatterplots}
\end{figure}

Using numerical integration, in this low-dimensional example we are able to compute an estimate of the normalization constant $Z = \int_{\calX} \exp(-f(Kx)-g(x))\rmd x$. Using this, we can access an approximation of the posterior density $\mu^\ast(x)$. This allows us to compute approximations of the Wasserstein distances of current iterates to the target $\calW_2(\hat \mu_n,\mu^\ast)$ without the need to rely on a second sampler. We evaluate the distance at 20 equidistantly spaced points in $[0,T]$ (i.e. for $\hat \mu_{kN}, k=1,\dots,20$ where $N = T/20\tau$). The distances are computed using the Python optimal transport module \cite{Flamary2021} and are displayed in \cref{fig:l2tv-2d-wasserstein}. As expected, the samples drawn by the Prox-Sub algorithm are closest to the target distribution (their bias in Wasserstein distance is bounded by a constant multiple of the step size $\tau$). For finite $\lambda$, the distributions $\hat \mu_n$ exhibit an additional strong bias to the target due to the dualization, which becomes smaller as $\lambda$ grows.

\begin{figure}[ht]
    \centering
    \includegraphics[width=0.5\linewidth]{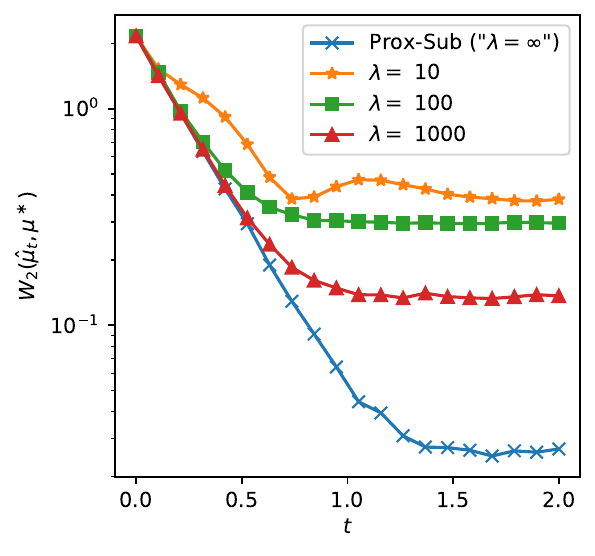}
    \caption{Wasserstein distances between the empirical measures $\hat \mu_n$ and the target $\mu^\ast$ in the two-dimensional TV-denoising experiment. We compare three different values of $\lambda$ and the algorithm Prox-Sub \cite{Habring2024} corresponding formally to $\lambda = \infty$. The discretization via partial dualization of $f\circ K$ introduces a bias between samples and target that is significantly larger than the pure discretization bias of Prox-Sub. As $\lambda$ grows, the iteration approximates the purely primal setting and the bias is reduced.}
    \label{fig:l2tv-2d-wasserstein}
\end{figure}

\subsubsection{Image Denoising}
We now sample from the posterior of an image denoising task. Due to the higher dimension, we can not compute accurate approximations of the Wasserstein distances to the target anymore. We therefore resort to illustrating the behaviour of the algorithm \eqref{eq:pd_sampling_update} using moment estimates. A ground truth image normalized to intensities $[0,1]$ is distorted by noise with standard deviation $0.25$ (PSNR $12.02\,$dB in the noisy observation). We then draw $10^6$ samples with a single chain and compute the sample mean and pixel-wise variance after a burn-in phase of $10^3$ samples. The experiment is repeated three times for values $\lambda = 1,10,100$, each with the same step size $\tau$. A fourth chain is computed using Prox-Sub with the same $\tau$. MMSE and pixel-wise variance estimates are reported in \cref{fig:l2tv-images-mmse-std-estimates}. The sample mean (MMSE) estimates produced by the different chains converge quickly and are very similar for all chains (PSNR values between $23.82\,$dB and $23.85\,$dB). The pixel-wise variance, however, shows the same effect of overdispersion for smaller levels of $\lambda$ as in the lower-dimensional example. 

While the `primal' samples $X^n$ are overdispersed for finite $\lambda$, the dual iterates $Y^n$ characteristically exhibit underdispersion (see \cref{fig:l2tv-images-dualstd-estimates}). This effect was consistent throughout all our numerical experiments (compare also \cref{fig:1dgaussian-toy-joint-distributions,fig:l2tv-2d-scatterplots}). A heuristic explanation for overdispersion in the primal and underdispersion in the dual variable lies in the fact that the step in the dual variable solves the optimality condition $y \in \partial f(Kx)$ only incompletely. On the particle level in continuous time, when the primal variable $X_t$ is driven towards the tail regions of its marginal distribution by the Brownian motion, the dual variable is not evolved entirely to the corresponding position $\partial f(KX_t)$ in the tails of its marginal, but rather remains near the mode of its marginal. The effect decreased with growing parameters $\lambda$ in all our experiments.

\newlength{\imageheight}%
\setlength{\imageheight}{0.22\linewidth}%
\begin{figure}[ht]
    \centering
    \begin{subfigure}{\imageheight}%
        \includegraphics[height=\imageheight]{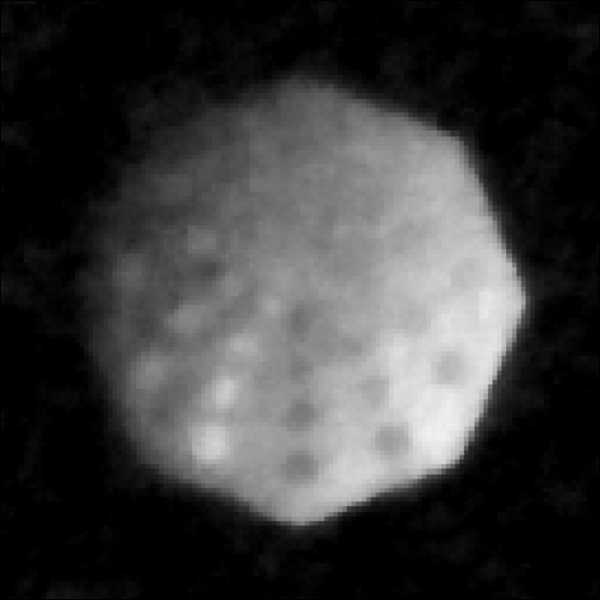}
    \end{subfigure}%
    \hspace{0.01\linewidth}%
    \begin{subfigure}{\imageheight}%
        \includegraphics[height=\imageheight]{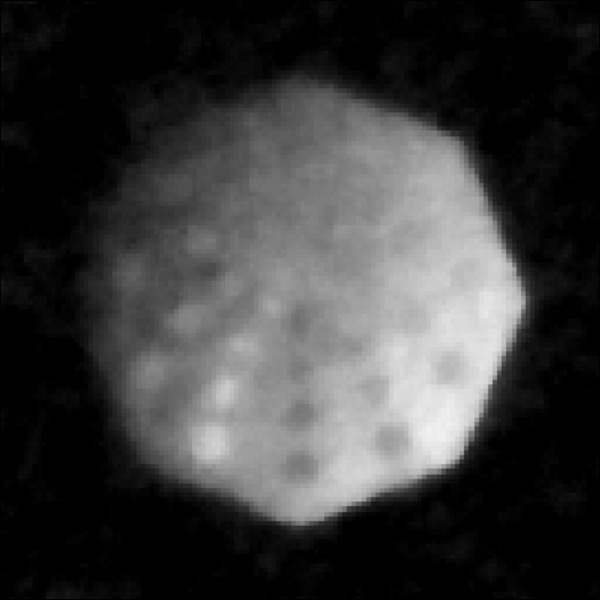}
    \end{subfigure}%
    \hspace{0.01\linewidth}%
    \begin{subfigure}{\imageheight}%
        \includegraphics[height=\imageheight]{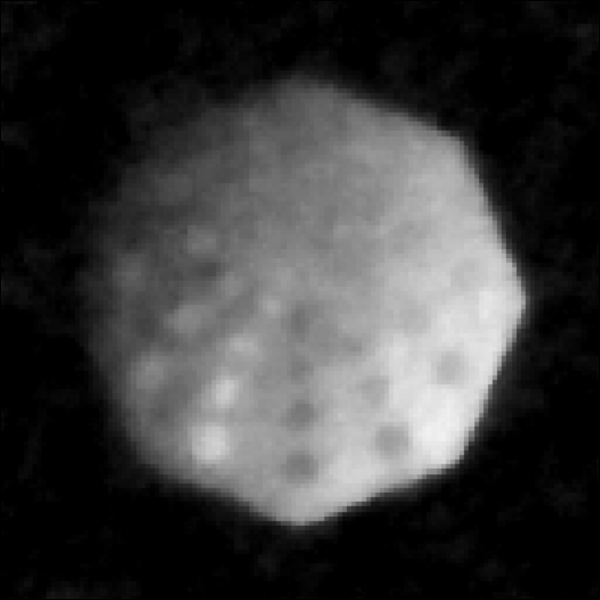}
    \end{subfigure}%
    \hspace{0.01\linewidth}%
    \begin{subfigure}{\imageheight}%
        \includegraphics[height=\imageheight]{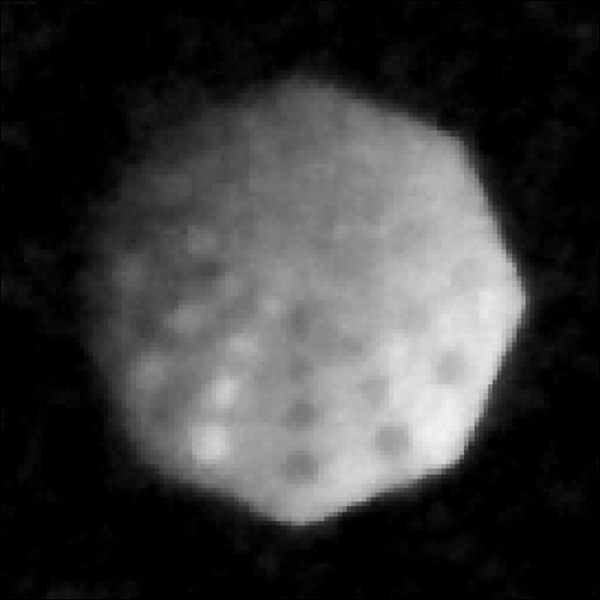}
    \end{subfigure}%
    \hspace{0.01\linewidth}%
    \begin{subfigure}{0.05\linewidth}%
        \includegraphics[height=\imageheight]{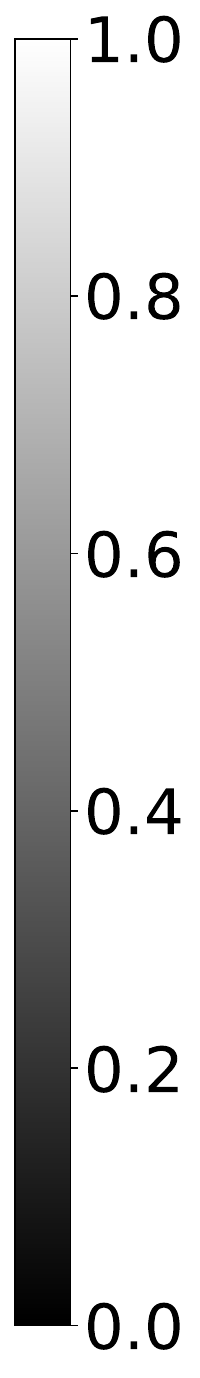}
    \end{subfigure}%
    \\%
    \vspace{0.005\linewidth}
    \begin{subfigure}{\imageheight}%
        \includegraphics[height=\imageheight]{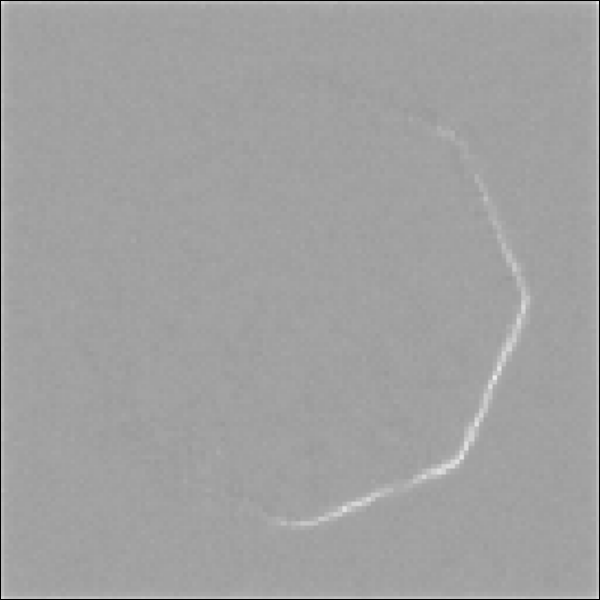}
    \end{subfigure}%
    \hspace{0.01\linewidth}%
    \begin{subfigure}{\imageheight}%
        \includegraphics[height=\linewidth]{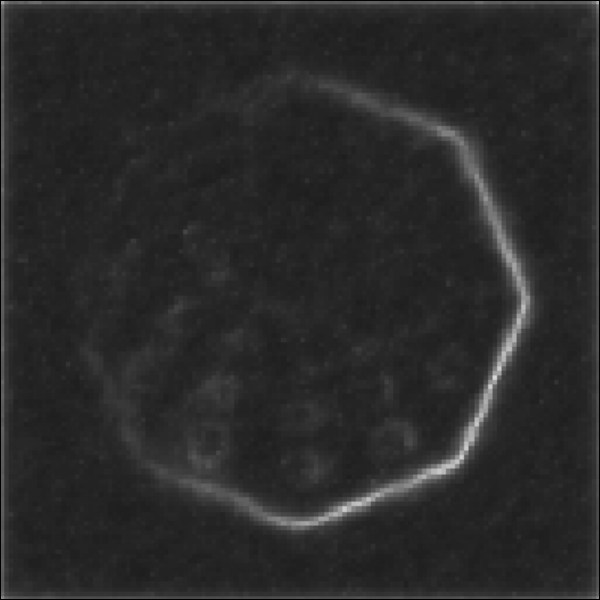}
    \end{subfigure}%
    \hspace{0.01\linewidth}%
    \begin{subfigure}{\imageheight}%
        \includegraphics[height=\linewidth]{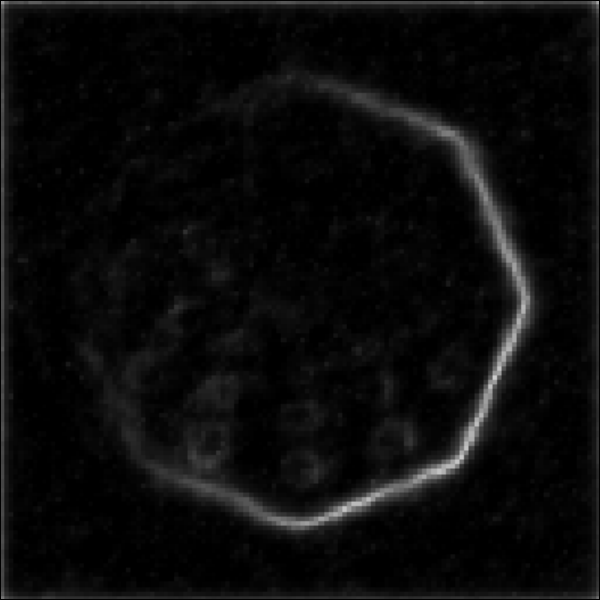}
    \end{subfigure}%
    \hspace{0.01\linewidth}%
    \begin{subfigure}{\imageheight}%
        \includegraphics[height=\linewidth]{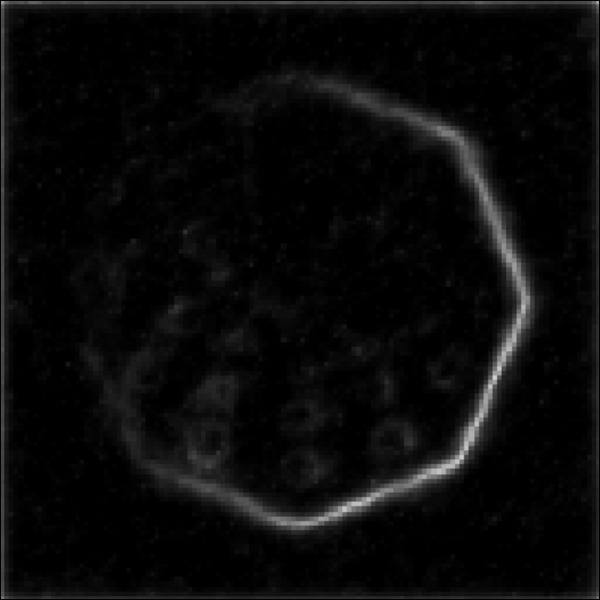}
    \end{subfigure}%
    \hspace{0.01\linewidth}%
    \begin{subfigure}{0.05\linewidth}%
        \includegraphics[height=\imageheight]{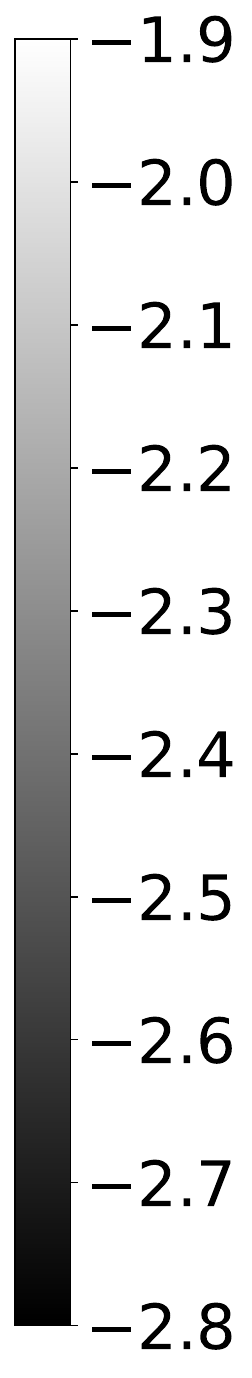}
    \end{subfigure}%
    \caption{MMSE estimates and pixel-wise variance computed via the posterior sample mean in the $l_2$-TV denoising experiment. Top row shows the MMSE estimates, bottom row the corresponding pixel-wise (logarithmically scaled) variances. From left to right results for $\lambda = 1,10,100$ and `$\lambda = \infty$' (using Prox-Sub) are shown.}
    \label{fig:l2tv-images-mmse-std-estimates}
\end{figure}

\begin{figure}[ht]
    \centering
    \begin{subfigure}{\imageheight}%
        \includegraphics[height=\imageheight]{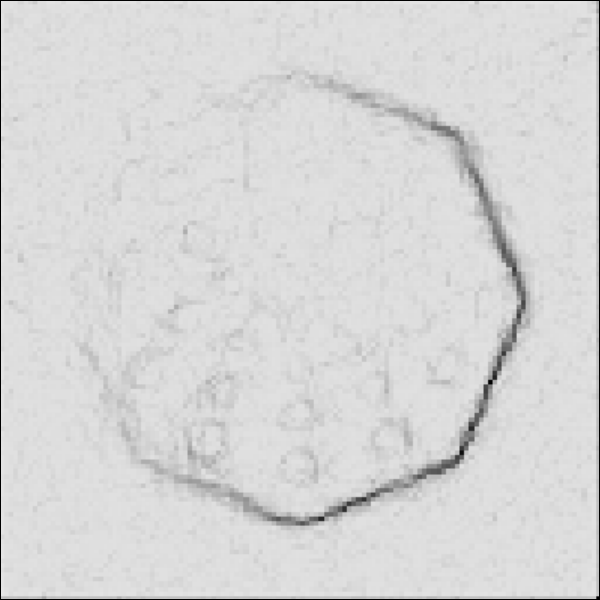}
    \end{subfigure}%
    \hspace{0.01\linewidth}%
    \begin{subfigure}{\imageheight}%
        \includegraphics[height=\imageheight]{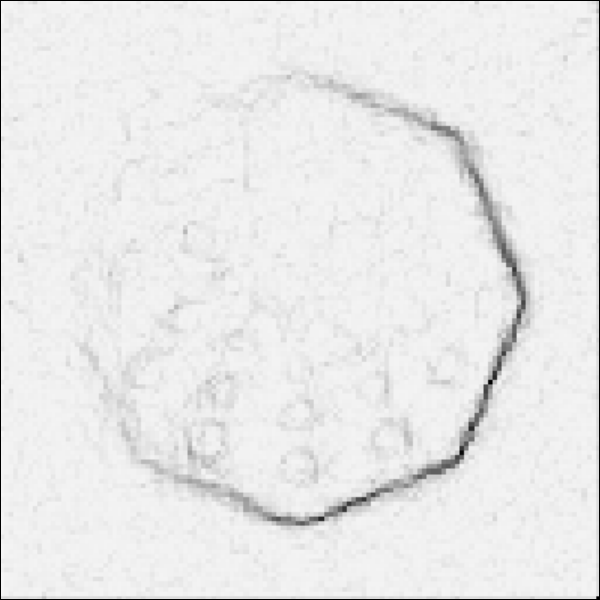}
    \end{subfigure}%
    \hspace{0.01\linewidth}%
    \begin{subfigure}{\imageheight}%
        \includegraphics[height=\imageheight]{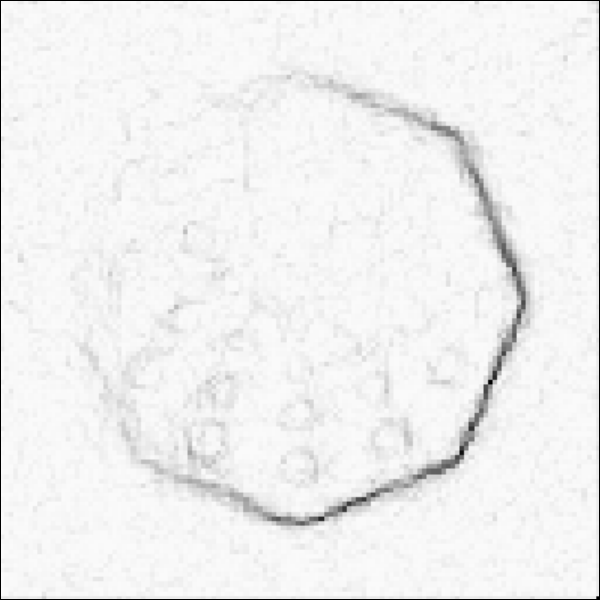}
    \end{subfigure}%
    \hspace{0.01\linewidth}%
    \begin{subfigure}{\imageheight}%
        \includegraphics[height=\imageheight]{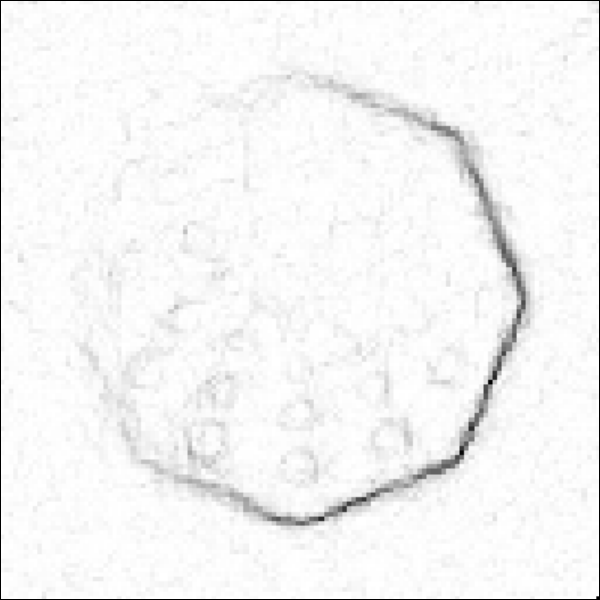}
    \end{subfigure}%
    \hspace{0.01\linewidth}%
    \begin{subfigure}{0.05\linewidth}%
        \includegraphics[height=\imageheight]{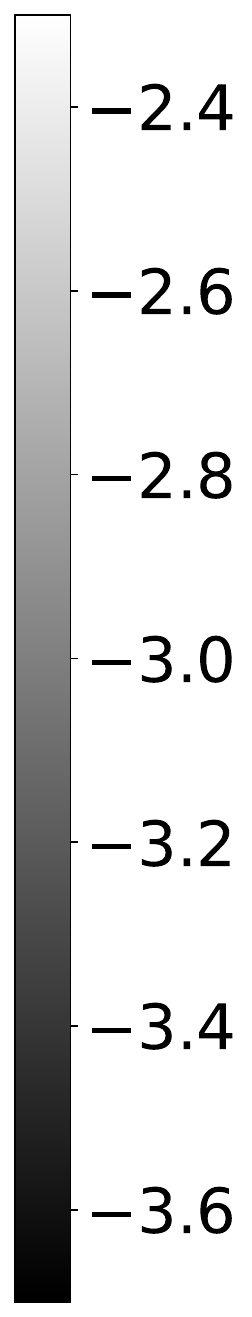}
    \end{subfigure}%
    \caption{Pixel-wise (logarithmic) standard deviation of the dual variables $Y^n$, where an $l_2$-norm over horizontal and vertical component was applied. From left to right, we show the results for $\lambda = 1,10,100$ and `$\lambda = \infty$' (Prox-Sub). For the latter case, the algorithm computes the dual variable as $Y^n \in \partial f(KX^n)$. For finite $\lambda$, the dual samples suffer from underdispersion, the effect vanishes as $\lambda$ grows and the dual values concentrate around the dual optimality condition.}
    \label{fig:l2tv-images-dualstd-estimates}
\end{figure}

\subsection{Total Generalized Variation Based Image Denoising}
In this experiment, we consider an example with an image prior based on second order total generalized variation (TGV). 
Like the TV functional, total generalized variation allows for discontinuous edges in images, but it generalizes TV in that it also employs higher-order derivatives. 
Acting as a penalty also on higher order image gradients, TGV-based MAP estimates avoid the typical staircasing effects of TV in smooth but non-constant regions of images. For optimization purposes, the second order TGV functional is usually defined as
\begin{gather}\label{eq:def-tgv}
    \TGV^2_\alpha(u) = \inf_{v \in \calV} J^2_\alpha(u,v), \\
    J^2_\alpha(u,v) = \alpha_1 \norm{\nabla u - v}_{2,1} + \alpha_0 \norm{E v}_{2,1}. \notag
\end{gather}
Here $u \in \calU$ is an image while $v$ is in the corresponding space $\calV = \calU\times \calU$ of (discretized) gradients of $u$. $\alpha_0,\alpha_1 > 0$ are regularization parameters. $\nabla = (\nabla_h,\nabla_v) : \calU \to \calV$ is a (forward discretized) gradient operator as in the case of TV, and $E : \calV \to \calW$ is a symmetrized (backward discretized) gradient operator that acts as second order derivative. We refer to Section 2.1 of \cite{Bredies2015a} for details on the implementation of the gradient operators.
In TGV-based variational minimization, the state space is typically expanded from the image space $\calU$ to the joint space $\calX = \calU \times \calV$ and the minimization over $u$ (in the variational problem) and $v$ (in \eqref{eq:def-tgv}) is carried out simultaneously. By applying primal-dual optimization methods, the iteration can be carried out efficiently without the need to evaluate the infimum in \eqref{eq:def-tgv} explicitly \cite{Bredies2015a}.
In the same way, we want to avoid evaluating $\TGV^2_\alpha$ or its gradients explicitly in a sampling procedure. In order to recover the image $u$, we therefore propose to sample from a posterior of the form
$$ \mu^\ast(x) \propto p(\tilde x | x) \exp\left( - J^2_\alpha(u,v)\right),\qquad x = (u,v) \in \calU \times \calV = \calX, $$
where $p(\tilde x|x)$ is the likelihood, for denoising with normally distributed noise given as before by
$$- \log p(\tilde x|x) = \frac{1}{2\sigma_\epsilon^2} \norm{u - \tilde u}_\calU^2 ,\qquad x = (u,v), \tilde x = (\tilde u, \tilde v) \in \calX$$
We note that this is a relaxed version of regularization with $\TGV_\alpha^2$, since we dropped the requirement that $v$ solves \eqref{eq:def-tgv}. Practically, the prior $\exp(-J^2_\alpha(u,v))$ seems to yield very similar smoothing behaviour to the effect of total generalized variation, we leave its detailed analysis to future work.
Note that the chosen target log-density is in the necessary form \eqref{eq:target}: $g(x)$ is, as before, the negative log-likelihood $\norm{u-\tilde u}^2_\calU/(2\sigma_\epsilon^2)$ for all $x = (u,v)$. We define the linear operator
$$ K = \begin{pmatrix}
    \nabla & -I\\
    0 & E
\end{pmatrix}: \calX \to \calV \times \calW \eqqcolon \calY,$$
where $I$ is the identity on $\calV$. The functional $f : \calY \to \bbR$ is given by $f(y) = \alpha_1\norm{v}_{2,1} + \alpha_0\norm{w}_{2,1}$ for all $y = (v,w)$. Due to the separable structure in $v$ and $w$, the proximal mapping of the convex conjugate $f^\ast$ is simply a componentwise projection in the dual $l_2$-$l_\infty$-norms in the respective spaces $\calV$, $\calW$.

We run the experiment with empirically set values of $\lambda$ that make sure the variance is not strongly overdispersed while keeping the computational expense moderate -- recall that as $\lambda$ grows, $\tau$ needs to be chosen smaller, which slows down the convergence.
$10^5$ samples are drawn in each run, we report the resulting sample mean and pixel-wise variance in \cref{fig:l2tgv-images}.

\setlength{\imageheight}{0.3\linewidth}
\begin{figure}[ht]
    \centering
    \begin{subfigure}{\imageheight}%
        \includegraphics[height=\imageheight]{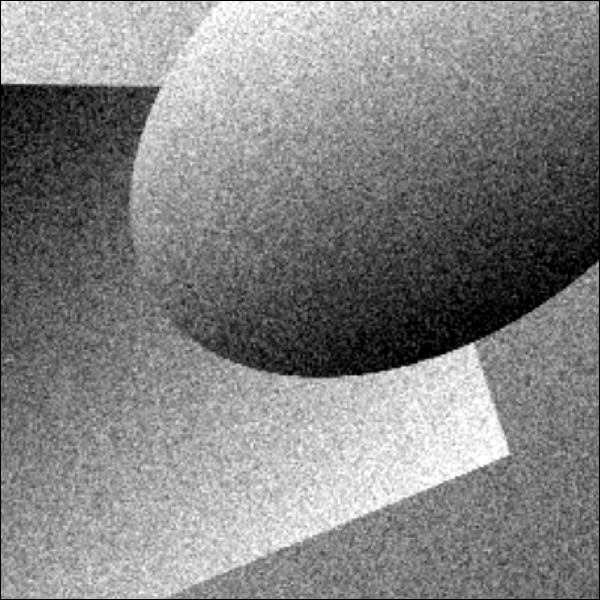}
    \end{subfigure}%
    \hfill%
    \begin{subfigure}{\imageheight}%
        \includegraphics[height=\imageheight]{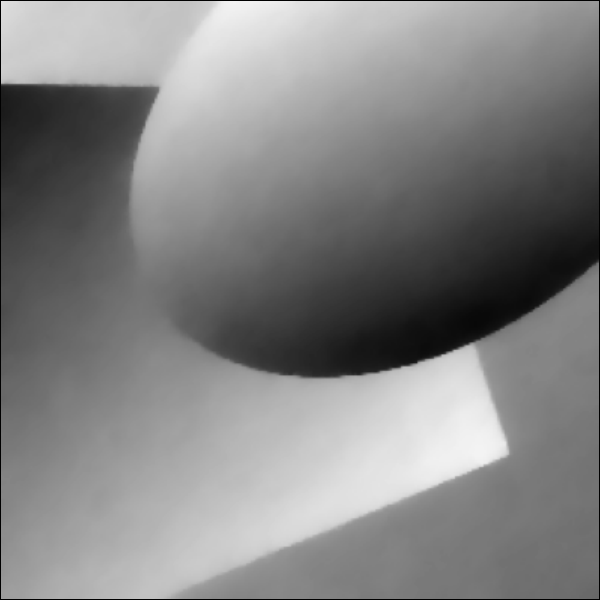}
    \end{subfigure}%
    \hspace{0.01\linewidth}%
    \begin{subfigure}{\imageheight}%
        \includegraphics[height=\imageheight]{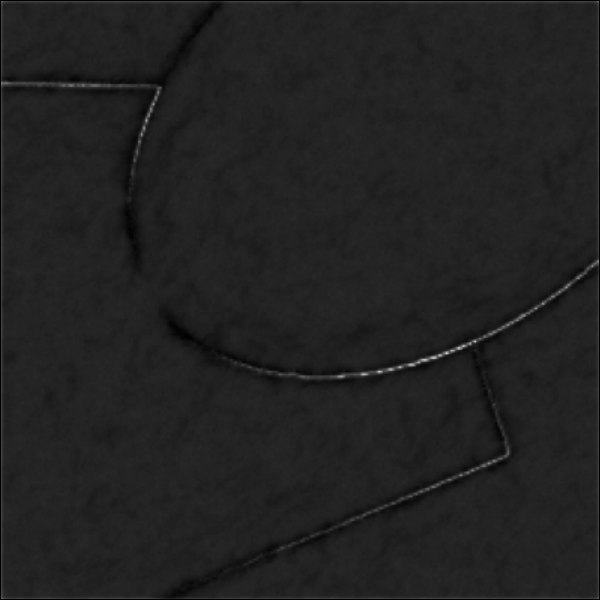}
    \end{subfigure}\\%
    \vspace{0.008\linewidth}
    \begin{subfigure}{\imageheight}%
        \includegraphics[height=\imageheight]{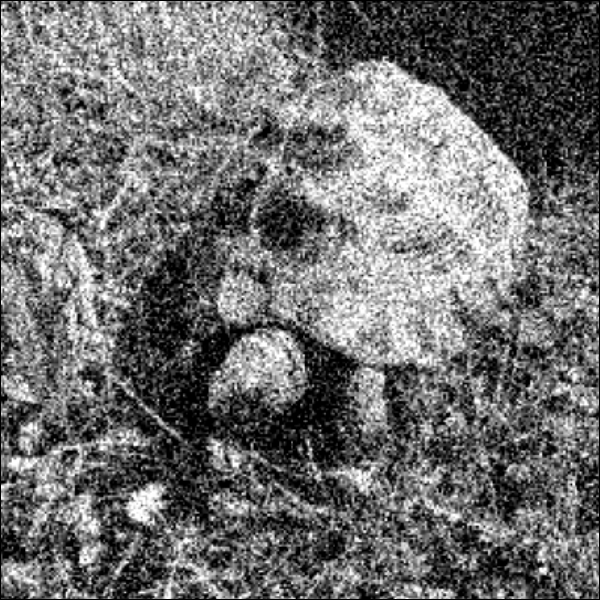}
    \end{subfigure}%
    \hfill%
    \begin{subfigure}{\imageheight}%
        \includegraphics[height=\imageheight]{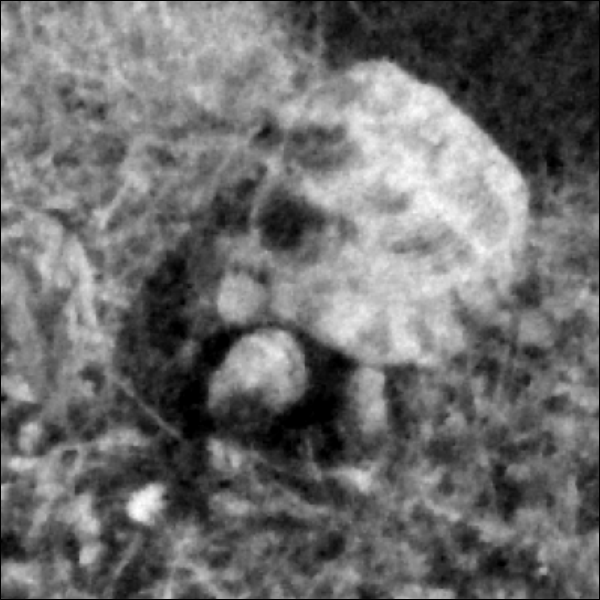}
    \end{subfigure}%
    \hspace{0.01\linewidth}%
    \begin{subfigure}{\imageheight}%
        \includegraphics[height=\imageheight]{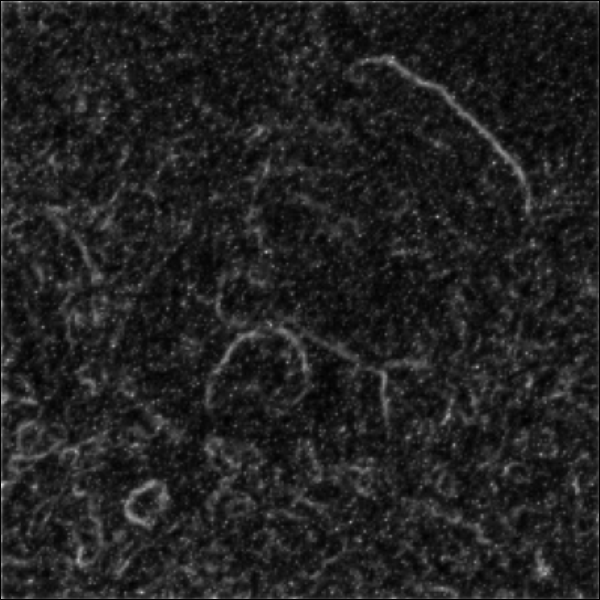}
    \end{subfigure}%
    \caption{Results from the TGV-based denoising experiment. Left: noisy observation $\tilde u$, Center: MMSE estimate via sample mean, Right: Log-scaled pixel-wise standard deviation. Top is a moderately noisy phantom designed for TGV-regularization ($\sigma_\epsilon=0.1$, PSNR of noisy image $19.98\,$dB, MMSE PSNR $34.00\,$dB), bottom a more strongly distorted, photorealistic image ($\sigma_\epsilon = 0.25$, PSNR of noisy image $12.02\,$dB, MMSE PSNR $20.08\,$dB).}
    \label{fig:l2tgv-images}
\end{figure}

\subsection{Comparison of Proximal Variants: Image Inpainting with Constraint}\label{subsec:numerics:inpainting}
We want to stress that the different variants of primal-dual and proximal schemes can have an effect on the quality of samples. Depending on the chosen splitting functions $f, g$ and the application, some versions could be preferred over the others. In the following, let us consider a constrained inpainting problem in which the true image has non-negative values, but a fraction of the image pixels are unobserved and the remaining observed pixels follow a Poisson distribution. Let $M \in \{0,1\}^{p \times d}$ (with $p=d/2$) denote the operator that masks 50\% randomly selected pixels. Let $y \in \mathbb{R}^p$ be the corresponding noisy observation $y \sim \mathrm{Poisson}(Mx^\ast)$. The resulting negative log-likelihood is
\begin{equation}\label{eq:inpainting-nll}
    l(x) := \sum_{i\in M} x_i - y_i \log(x_i),
\end{equation}
where $i \in M $ is shorthand for $i$ being an observed pixel. We set $l(x) = \infty$ if any $x_i \le 0$. 

This likelihood is relevant in various inverse problems connected to photonic imaging modalities, since photon observation models are typically based on a Poisson distribution, see \cite{Hohage2016} for an overview. In such modalities, the reconstruction space is often constrained to be physically plausible, for example, quantities like light intensities or absorption coefficients are usually non-negative. The prior therefore contains a non-negativity constraint which we formulate by the formal density factor $\exp(-\iota_{\ge 0}(x))$, with $\iota_{\ge 0}(x) = \infty$ if any $x_i < 0$ and $\iota_{\ge 0}(x) = 0$ otherwise. Since the likelihood only contains information about the observed pixels, we add total variation regularization to penalize large gradients. The image prior therefore has a (formal) density $\exp(-\alpha \TV(x) - \iota_{\ge 0}(x))$. This leads to the splitting
\begin{align*}
    g(x) &= l(x) + \iota_{\ge 0}(x)\\
    f(y) &= \alpha \norm{y}_{2,1},
\end{align*}
and $K = \nabla$ is a discretized gradient. The proximal mapping of $g$ is given componentwise by
$$ \prox_{\tau g}(x)_i = \left\{
\begin{array}{ll}
    \left(x_i-\tau + \sqrt{(x_i-\tau)^2 + 4 \tau y_i}\right)/2& \text{if }i \in M,\\
    \max(x_i,0)& \text{if }i \notin M.
\end{array}
\right. $$
Even though the convergence results from this work are not applicable here due to missing strong convexity or Lipschitz gradient properties of $f,g$, the log-density of the posterior is well-defined and convex, and the primal-dual algorithm applicable. We compare four different iterations:
(A) The primal-dual algorithm \eqref{eq:pd_sampling_update}; (B) its variant \eqref{eq:pd_sampling_update_xi_in_prox} with diffusion inside the proximal argument; (C) the Prox-Sub algorithm \cite{Habring2024} and (D) a variant of Prox-Sub with diffusion inside the proximal argument. The primal steps of these four algorithms all differ slightly - we recall them here for clarity:
\begin{align}
    X_{\text{A}}^{n+1} &= \prox_{\tau g}\left(X_{\text{A}}^n - \tau K^T Y_{\text{A}}^{n+1}\right) + \sqrt{2\tau} \xi_{\text{A}}^{n+1} \tag{A}\label{eq:primal-step-primal-dual-noise-outside}\\
    X_{\text{B}}^{n+1} &= \prox_{\tau g}\left(X_{\text{B}}^n - \tau K^T Y_{\text{B}}^{n+1} + \sqrt{2\tau} \xi_{\text{B}}^{n+1} \right) \tag{B}\label{eq:primal-step-primal-dual-noise-inside}\\
    X_{\text{C}}^{n+1} &= \prox_{\tau g}\left(X_{\text{C}}^n - \tau K^T \Theta_{\text{C}}^{n+1}\right) + \sqrt{2\tau} \xi_{\text{C}}^{n+1} \tag{C}\label{eq:primal-step-proxsub-noise-outside}\\
    X_{\text{D}}^{n+1} &= \prox_{\tau g}\left(X_{\text{D}}^n - \tau K^T \Theta_{\text{D}}^{n+1} + \sqrt{2\tau} \xi_{\text{D}}^{n+1}\right) \tag{D}\label{eq:primal-step-proxsub-noise-inside}
\end{align}
In the primal-dual variants \eqref{eq:primal-step-primal-dual-noise-outside} and \eqref{eq:primal-step-primal-dual-noise-inside}, the dual iterates $Y$ result from a dual proximal step, and for the Prox-Sub variants \eqref{eq:primal-step-proxsub-noise-outside}, \eqref{eq:primal-step-proxsub-noise-inside}, the iterates $\Theta$ are computed directly as subgradients of $f$ at $KX$. All four algorithms are run for $10^4$ iterations. 

Qualitative reconstruction examples and the log-standard deviation maps are shown in \cref{fig:inpainting-qualitative}. The samples are visually similar, but a relevant difference between the iterations lies in pixels with small values close to zero: For algorithms \eqref{eq:primal-step-primal-dual-noise-inside} and \eqref{eq:primal-step-proxsub-noise-inside}, the proximal step is applied \emph{after} the diffusion. For the above $g$, this implies that all samples remain non-negative. The additional projection of negative coordinates to zero also reduces variance of samples on dark pixels (third row of \cref{fig:inpainting-qualitative}), as it reduces the bias due to the unbalanced, isotropic diffusion term. In the present example, the variants \eqref{eq:primal-step-primal-dual-noise-inside} and \eqref{eq:primal-step-proxsub-noise-inside} could be preferred in practice, as the samples and moment estimates adhere to the physical non-negativity constraint. These results suggest that further studies on the different variants are needed.

\setlength{\imageheight}{0.22\linewidth}
\begin{figure}[ht]
    \centering
    \begin{subfigure}{\imageheight}%
    \end{subfigure}%
    \hspace{0.01\linewidth}%
    \begin{subfigure}{\imageheight}%
        \includegraphics[height=\imageheight]{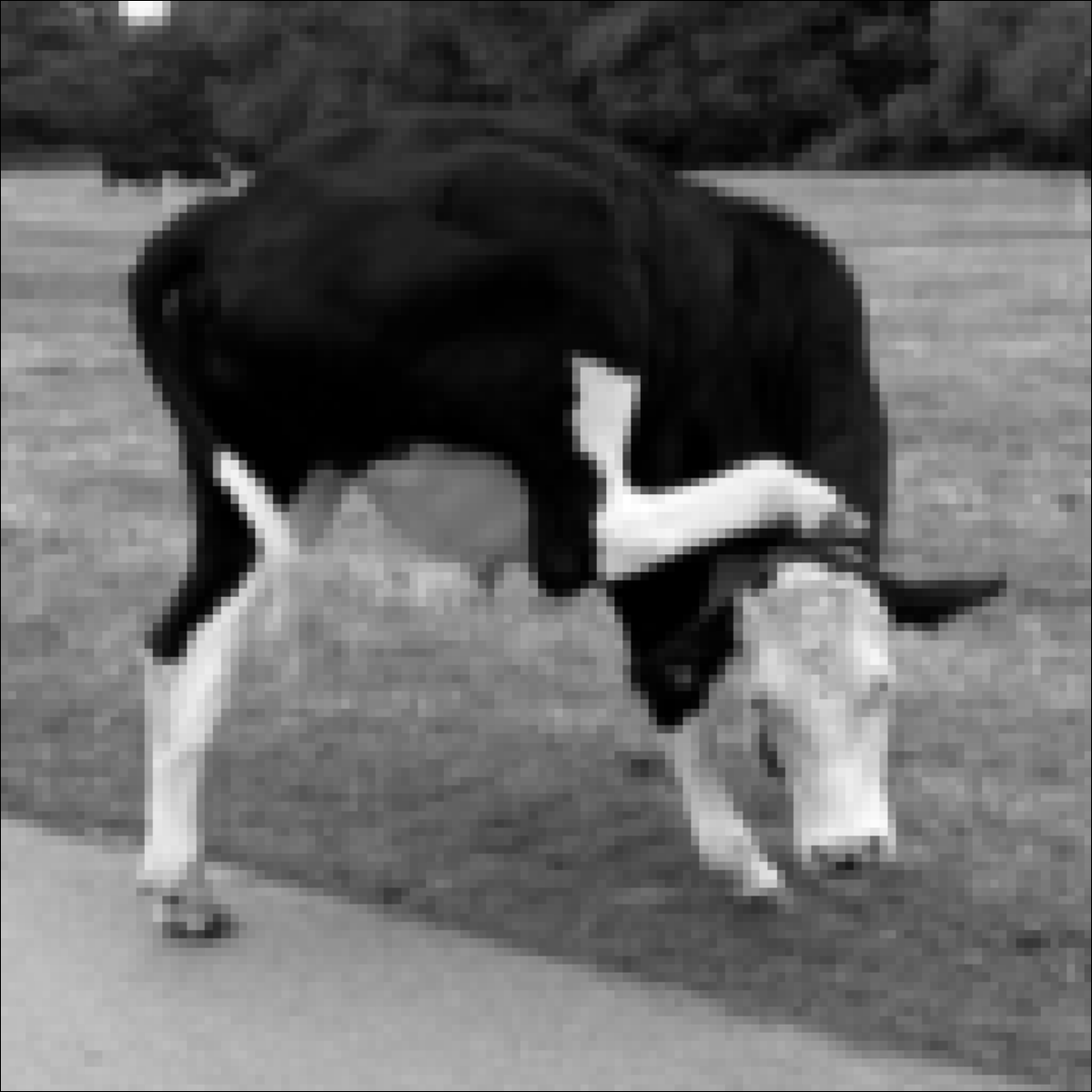}
    \end{subfigure}%
    \hspace{0.01\linewidth}%
    \begin{subfigure}{\imageheight}%
        \includegraphics[height=\imageheight]{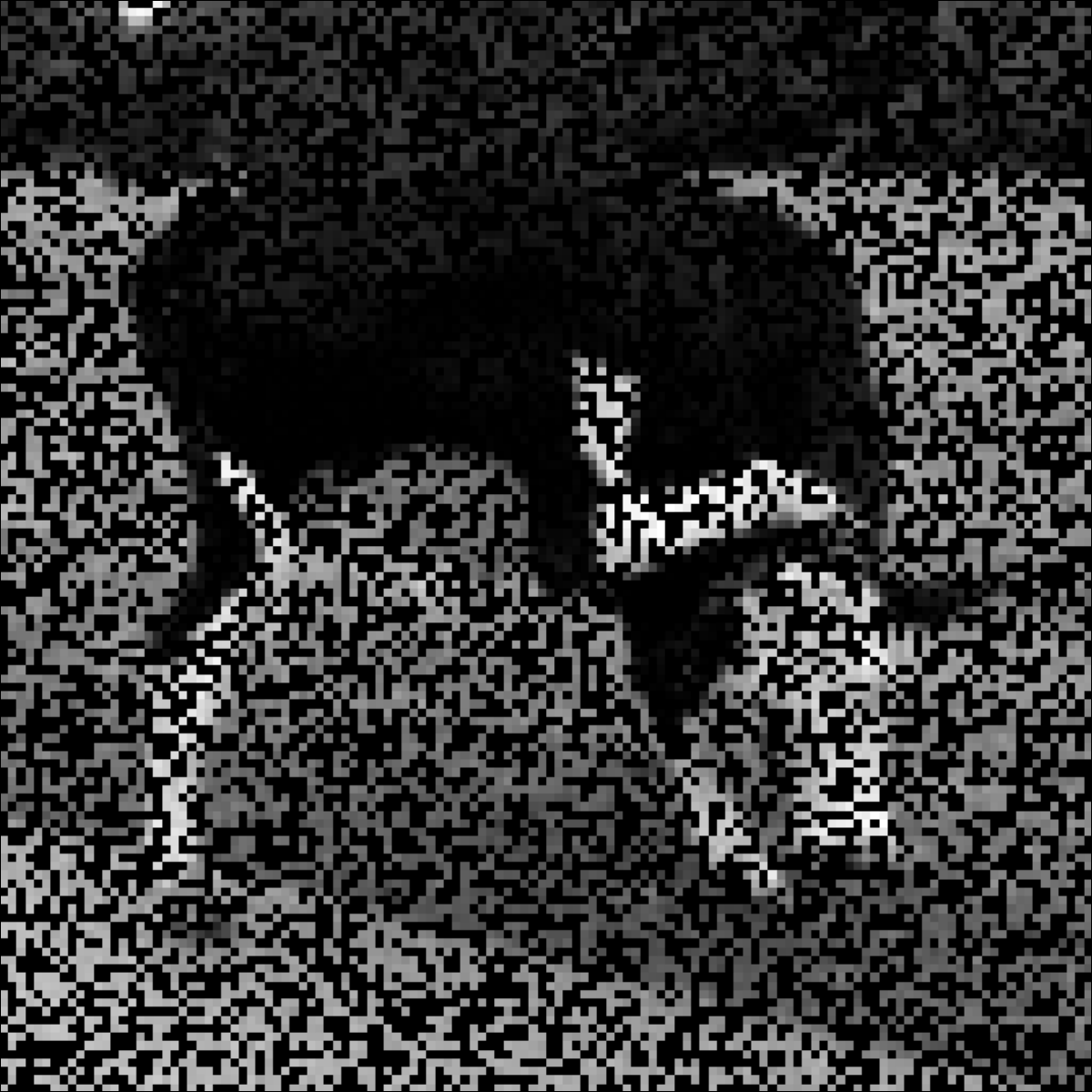}
    \end{subfigure}%
    \hspace{0.01\linewidth}%
    \begin{subfigure}{\imageheight}%
    \end{subfigure}%
    \hspace{0.01\linewidth}%
    \begin{subfigure}{0.05\linewidth}%
        \includegraphics[height=\imageheight]{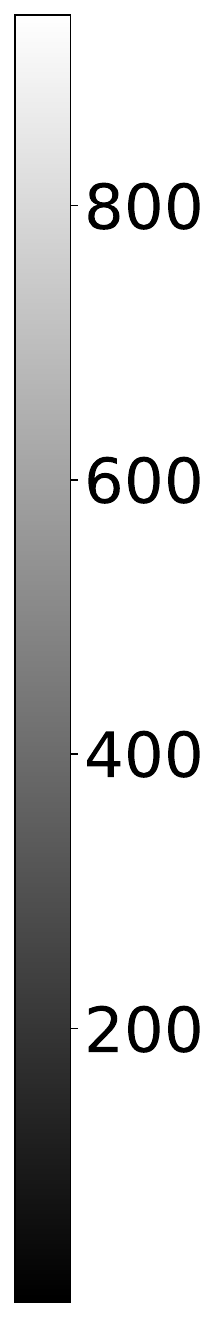}
    \end{subfigure}%
    \\%
    \vspace{0.008\linewidth}%
    \begin{subfigure}{\imageheight}%
        \includegraphics[height=\imageheight]{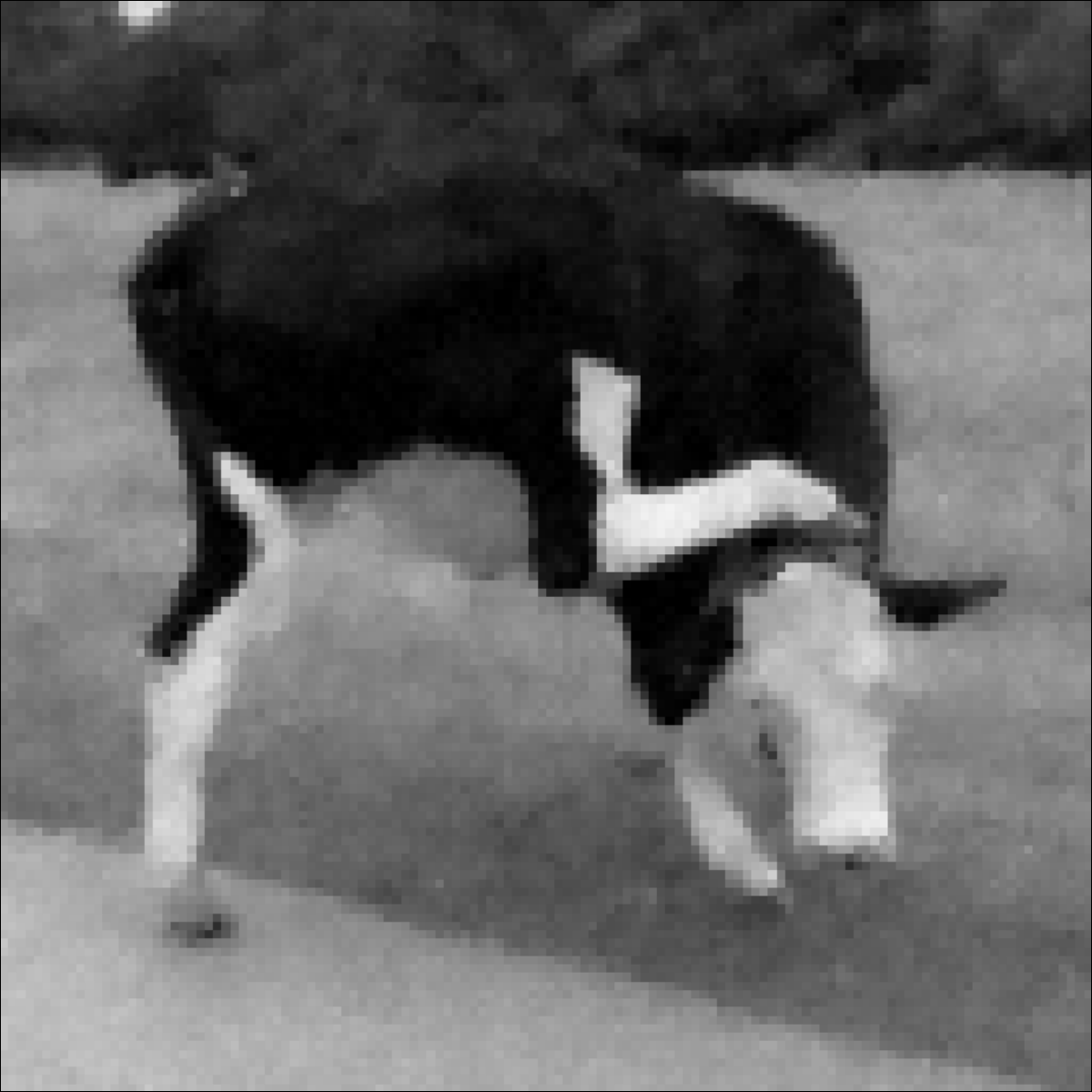}
        \subcaption*{(A): 1.3\%}
    \end{subfigure}%
    \hspace{0.01\linewidth}%
    \begin{subfigure}{\imageheight}%
        \includegraphics[height=\imageheight]{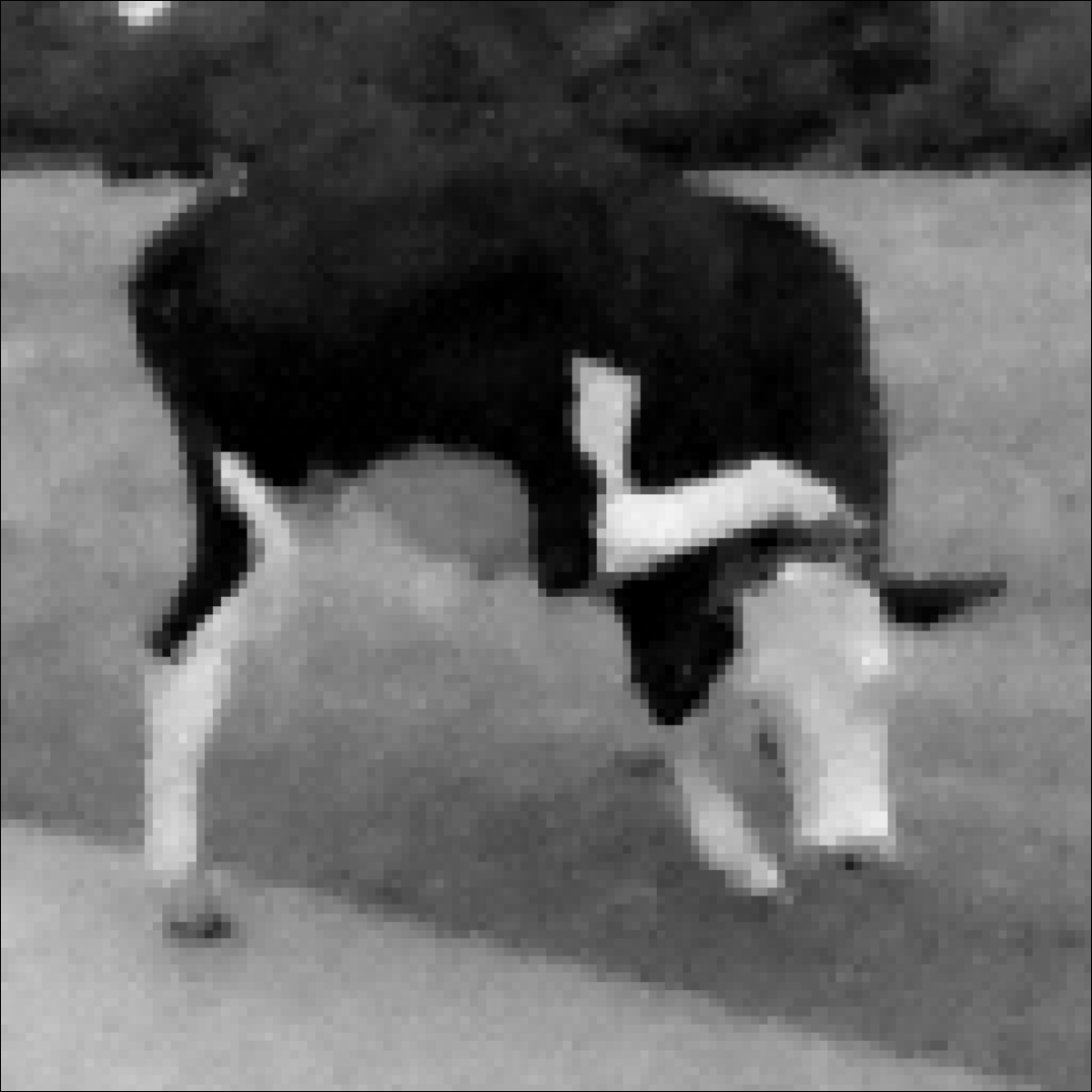}
        \subcaption*{(B): 0\%}
    \end{subfigure}%
    \hspace{0.01\linewidth}%
    \begin{subfigure}{\imageheight}%
        \includegraphics[height=\imageheight]{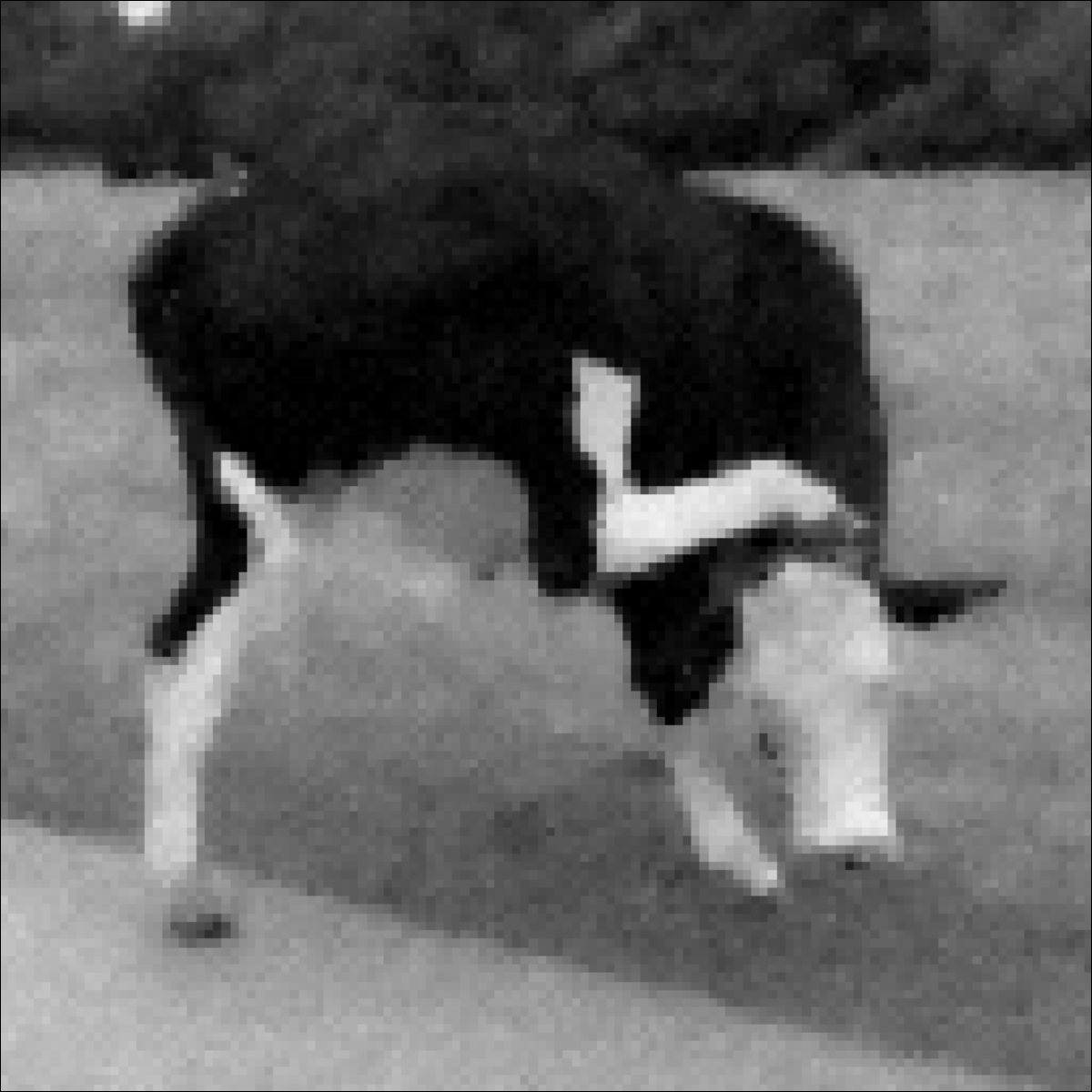}
        \subcaption*{(C): 1.0\%}
    \end{subfigure}%
    \hspace{0.01\linewidth}%
    \begin{subfigure}{\imageheight}%
        \includegraphics[height=\imageheight]{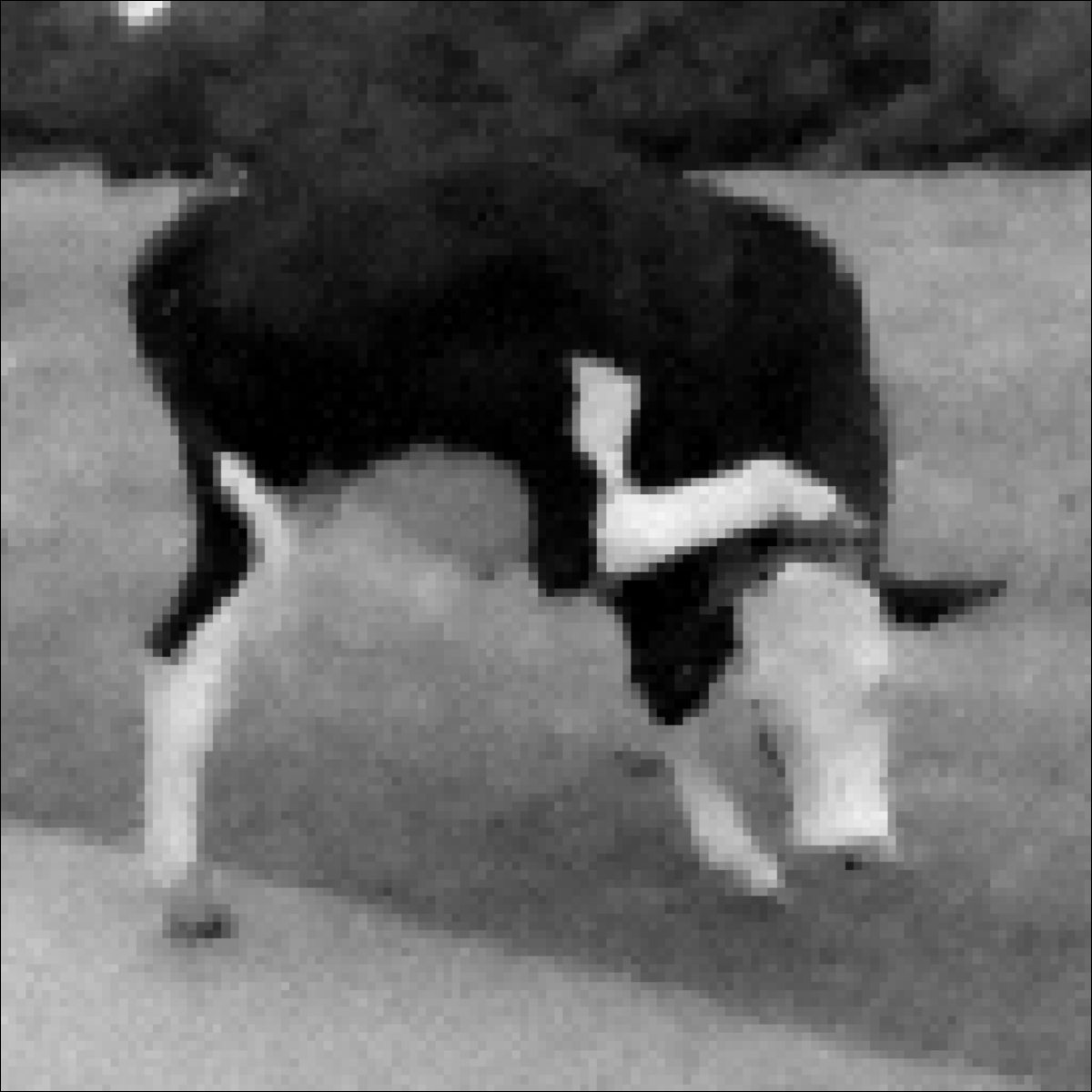}
        \subcaption*{(D): 0\%}
    \end{subfigure}%
    \hspace{0.01\linewidth}%
    \begin{subfigure}{0.05\linewidth}%
        \includegraphics[height=\imageheight]{results/kltv-inpainting/cbar_mean.pdf}
        \subcaption*{}
    \end{subfigure}%
    \\%
    \vspace{0.008\linewidth}%
    \begin{subfigure}{\imageheight}%
        \includegraphics[height=\imageheight]{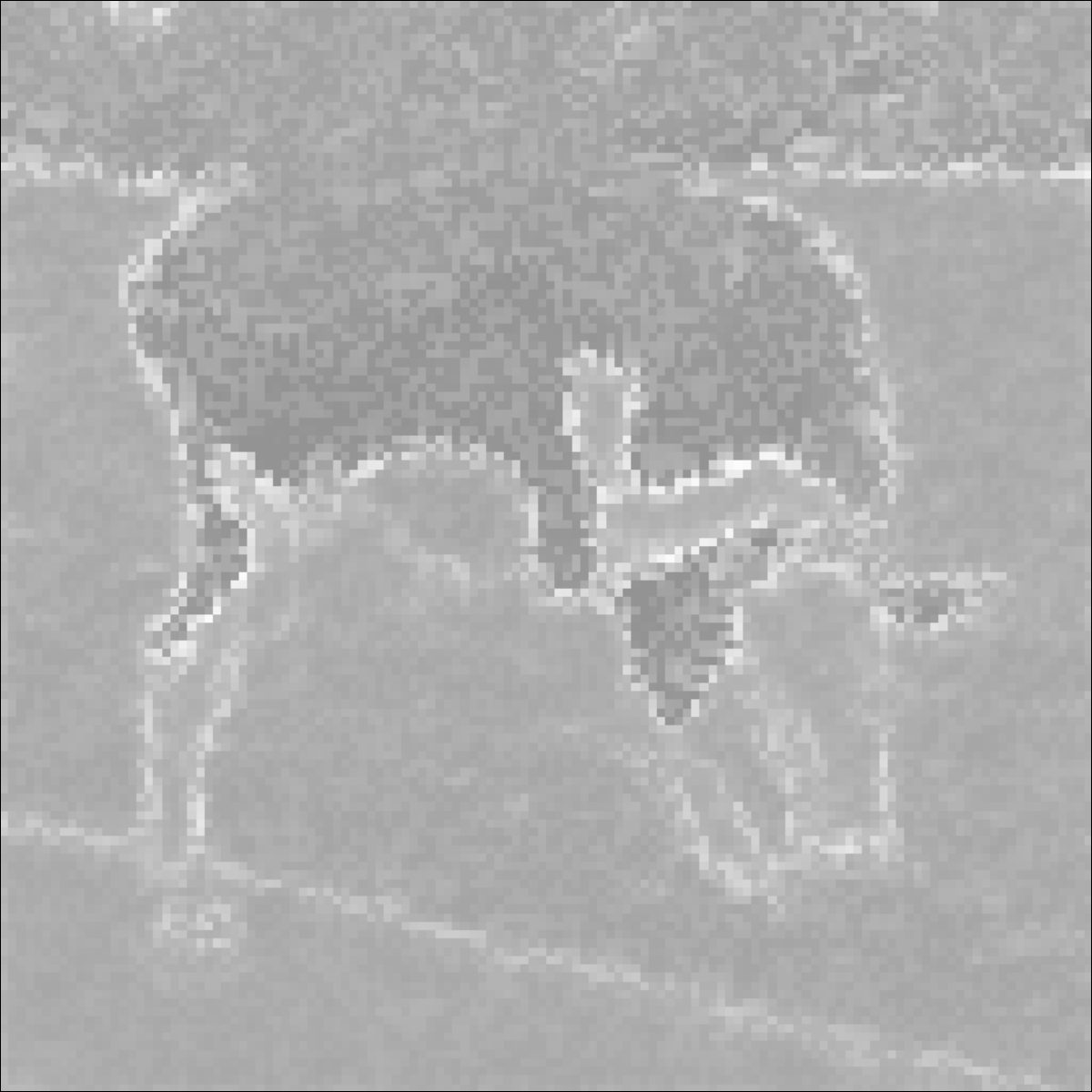}
    \end{subfigure}%
    \hspace{0.01\linewidth}%
    \begin{subfigure}{\imageheight}%
        \includegraphics[height=\imageheight]{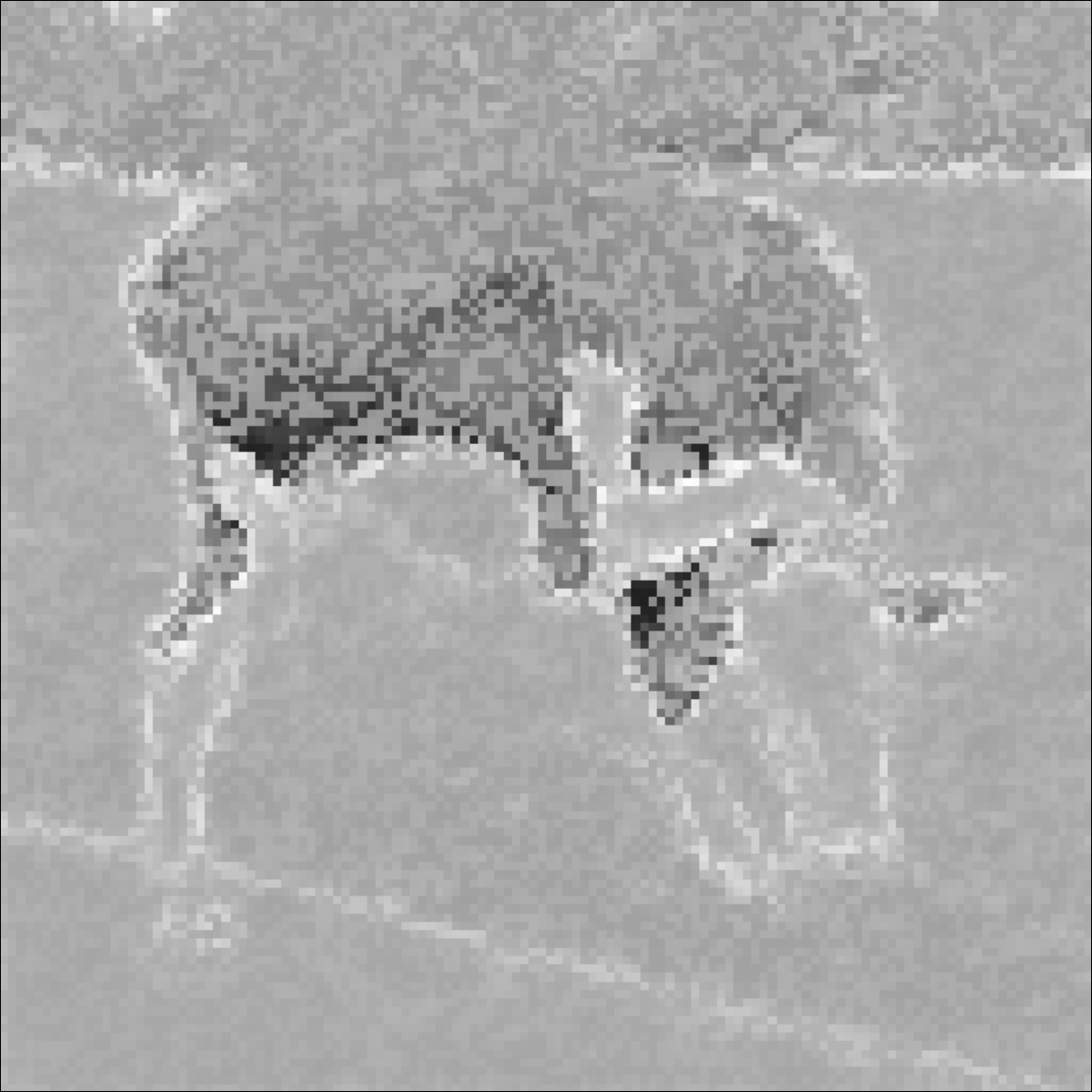}
    \end{subfigure}%
    \hspace{0.01\linewidth}%
    \begin{subfigure}{\imageheight}%
        \includegraphics[height=\imageheight]{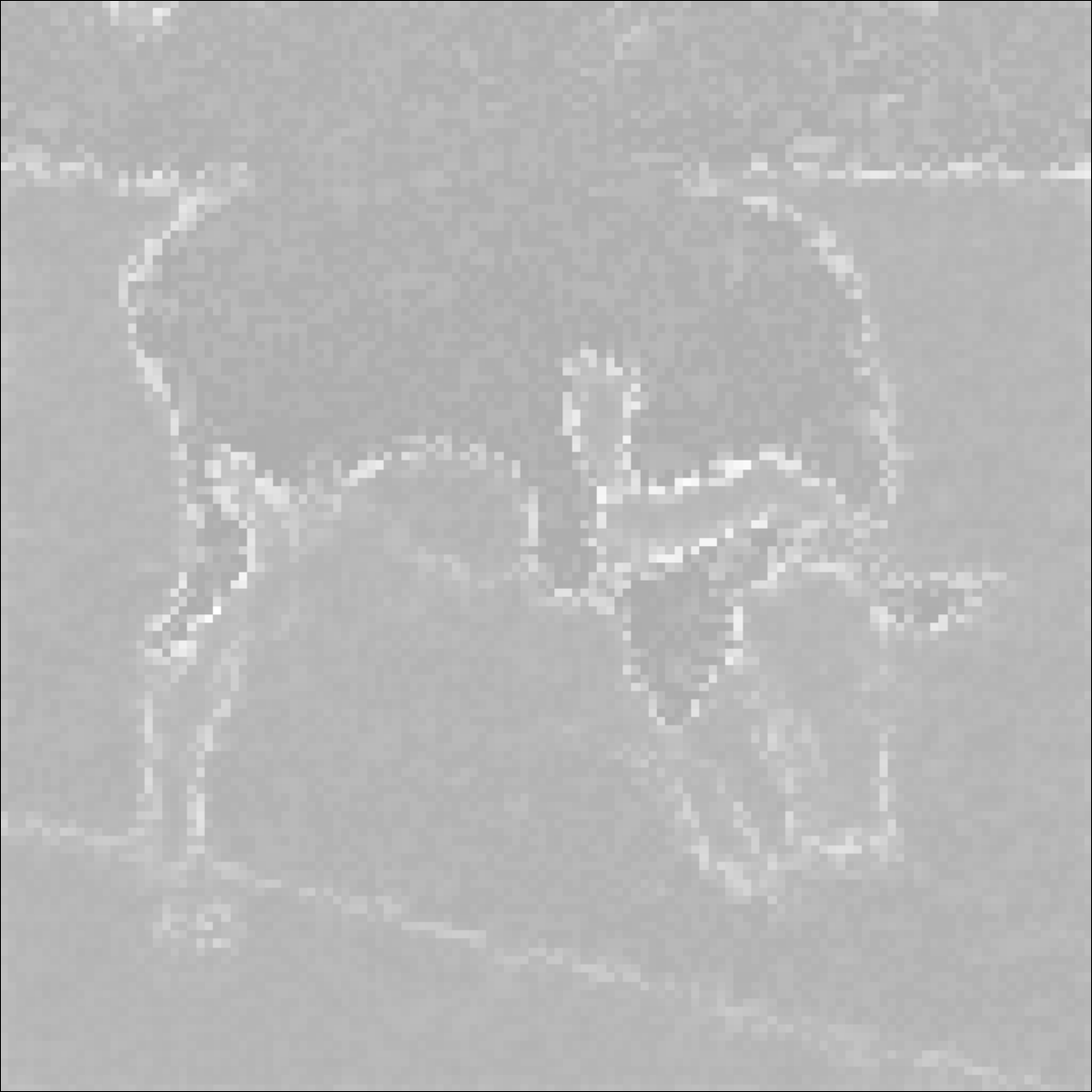}
    \end{subfigure}%
    \hspace{0.01\linewidth}%
    \begin{subfigure}{\imageheight}%
        \includegraphics[height=\imageheight]{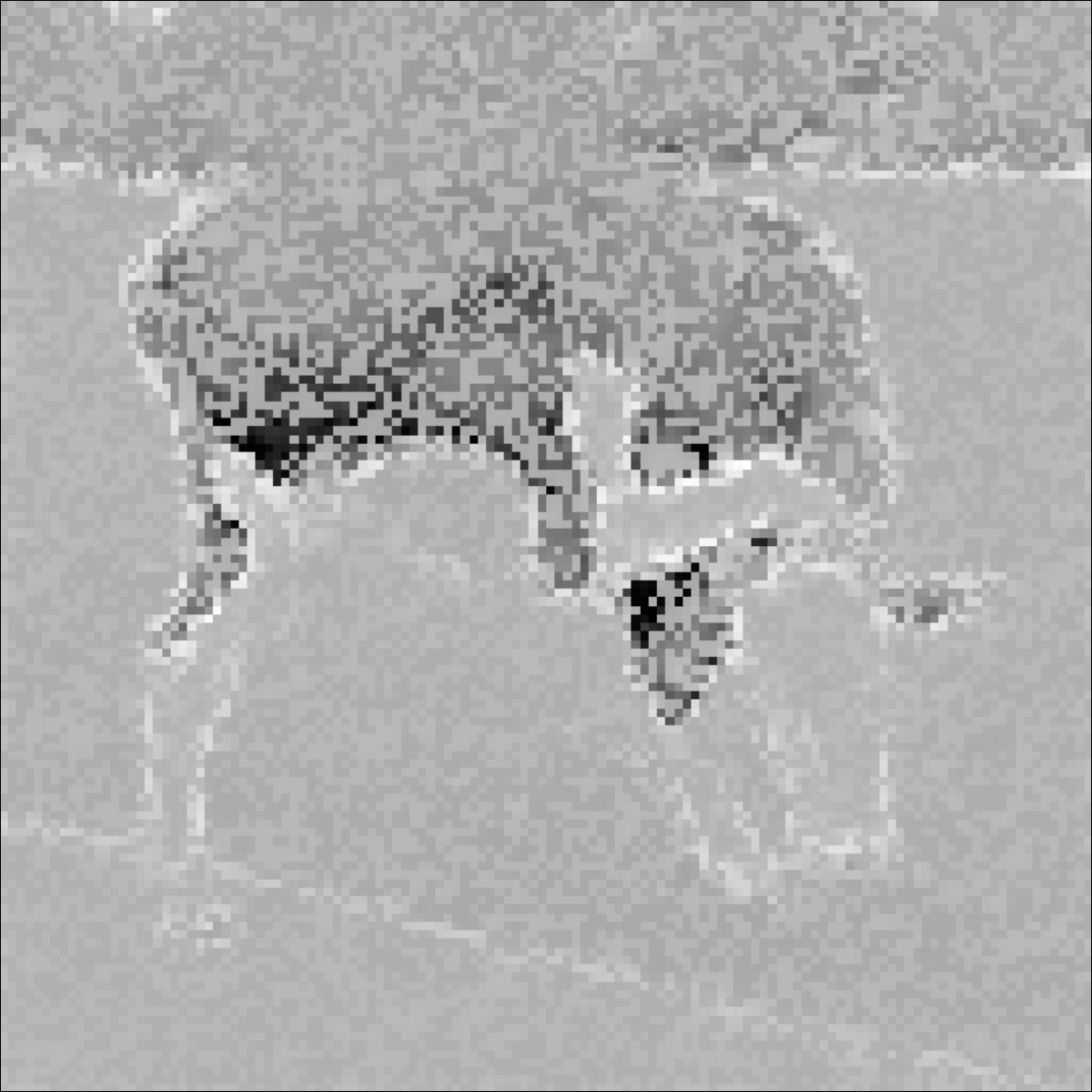}
    \end{subfigure}%
    \hspace{0.01\linewidth}%
    \begin{subfigure}{0.05\linewidth}%
        \includegraphics[height=\imageheight]{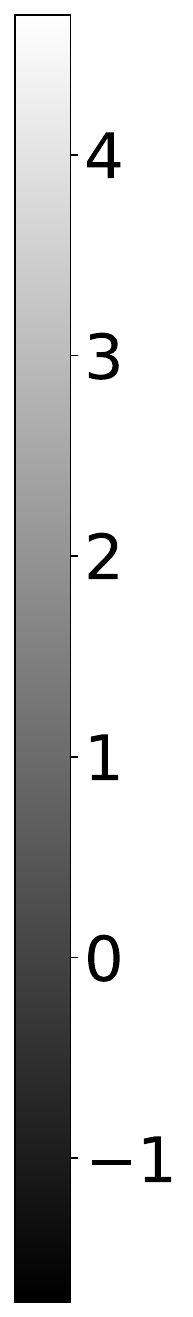}
    \end{subfigure}%
    \caption{Qualitative inpainting results after $10^4$ iterations. Top row: ground truth image and masked observation ($50\%$ of pixels removed, remaining corrupted by Poisson noise). Middle row: representative samples from the four algorithm variants \eqref{eq:primal-step-primal-dual-noise-outside}--\eqref{eq:primal-step-proxsub-noise-inside}; the percentage next to each label is the fraction of pixels with negative values in the sample. Bottom row: pixel-wise log standard deviation maps for variants \eqref{eq:primal-step-primal-dual-noise-outside}--\eqref{eq:primal-step-proxsub-noise-inside}. (A) standard primal-dual scheme \eqref{eq:pd_sampling_update}; (B) primal-dual with diffusion inside the proximal operator \eqref{eq:pd_sampling_update_xi_in_prox}; (C) Prox-Sub \cite{Habring2024}; (D) Prox-Sub with diffusion inside the proximal operator.}
    \label{fig:inpainting-qualitative}
\end{figure}

%% file: sections/6conclusion.tex
\section{Conclusion and Open Problems}\label{sec:conclusion}
We have analyzed a recently proposed algorithm for sampling from log-concave, non-smooth potentials which makes use of alternating primal-dual steps as is well-known in convex non-smooth optimization. We showed that the sampling algorithm corresponds to a stochastic differential equation and formulated the mean field limit Fokker--Planck equation. Using a simple coupling argument, contraction of the continuous-time limit to a stationary solution was proved. Unlike standard overdamped Langevin diffusion, the primal-dual Fokker--Planck equation is not a gradient flow in the corresponding Wasserstein space and does not have the target distribution as its stationary solution. A proposed correction of the Fokker--Planck equation forcing the target to be its stationary solution is possible, but eventually not useful for numerical implementation. However, the target distribution can be related to the stationary solution of the uncorrected scheme in the limit of the ratio of primal and dual step sizes growing to infinity. For the unmodified scheme, we made use of coupling techniques to prove a bound on Wasserstein distance between the stationary distribution and the target for a finite step size ratio. 

With similar coupling arguments in discrete time, estimates in the typical spirit of convergence analysis of the optimization algorithm can be transferred to the sampling case. This allowed us to show stability of the sampling method in Wasserstein distance, resulting in the existence of a stationary measure in discrete time under strong convexity assumptions. Under the additional assumption of Lipschitz-smoothness of one of the potentials, we can bound the discretization error using a coupling between discrete and continuous time. This final bound implies a convergence result that relates the generated samples with the target. An important open question is whether the regularity assumptions on the potentials $f,g$ in the stability result and the discretization error bounds can be relaxed to not require differentiability.

Numerically, the results suggest that the additional bias due to the finite step size ratio makes the primal-dual samples consistently over-dispersed compared to primal discretizations of overdamped Langevin diffusion like the method of \cite{Habring2024}. We currently see the primal-dual method's use cases in two settings. If the subgradient of the dualized term is not available, but the proximal mapping of its convex conjugate is, the method is the first applicable algorithm able to draw samples from an approximation of the posterior. Secondly, if the additional bias due to a larger finite step size ratio does not matter (because a correction step is added), the method can be applied with larger step sizes than a purely primal method, allowing for faster mixing time. To find out whether this makes the primal-dual method more efficient in applications, further numerical studies are needed. The combination of a proposal step using the primal-dual algorithm and a Metropolis--Hastings correction step seems like a promising method and should be analyzed further in future work.

%% file: sections/appendix.tex
\section{Proofs for \texorpdfstring{\cref{sec:continuous-time}}{Section 3}}\label{sec:appendix-continuous-time}
\subsection{Proofs for \texorpdfstring{\cref{subsec:mean-field-limit}}{Section 3.1}}\label{subsec:appendix-continuous-time:mean-field-limit}

\begin{proof}[Proof of \cref{lem:stability_general_dynamics_cont_time}]
With a change of variable we can write $\nabla_z + \nabla_{\tilde z} = \nabla_{z+\tilde z}$. Using the Fokker--Planck equation of the coupled diffusion \eqref{eq:joint_FP_general_strong_monotone_op}, we carry out the same steps as for overdamped Langevin:
    \begin{align*}
        \frac{1}2 \frac{\rmd}{\rmd t}
        \bbE_{(z,\tilde z) \sim \Pi_t} \left[ \norm{z-\tilde z}^2 \right] &= \frac{1}{2} \frac{\rmd }{\rmd t} \int \norm{z - \tilde z}^2 \Pi_t(\mathrm d z,\mathrm d \tilde z) \\
        &= \frac{1}{2} \int \norm{z - \tilde z}^2 \left( \nabla_z \cdot \left( w \Pi_t \right) + \nabla_{\tilde z} \cdot \left( \tilde w \Pi_t \right) + (\nabla_{z+\tilde z}) \cdot (\Sigma (\nabla_{z+\tilde z}) \Pi_t) \right) \rmd z \rmd \tilde z\\
        &= - \int \left( \inpro{z-\tilde z}{w} + \inpro{\tilde z - z}{\tilde w} \right) \Pi_t(\mathrm d z,\mathrm d \tilde z)\\
        &\le - \omega \bbE_{(z,\tilde z) \sim \Pi_t} \left[ \norm{z-\tilde z}^2 \right].
    \end{align*}
    In the second line we used \cref{prop:general_FP_joint} and in the third line partial integration together with the fact that $(\nabla_{z}+\nabla_{\tilde z})  (\norm{z-\tilde z}^2) = 0$. We now apply Gronwall's lemma and obtain
    \begin{equation*}
        \bbE_{(z,\tilde z) \sim \Pi_t} \left[ \norm{z-\tilde z}^2 \right] \le e^{- 2\omega t} \bbE_{(z,\tilde z) \sim \Pi_0} \left[ \norm{z-\tilde z}^2 \right]
    \end{equation*}
    which implies
    \begin{equation*}
        \calT(\pi_t,\tilde \pi_t) \le e^{-\omega t} \left(\bbE_{(z,\tilde z) \sim \Pi_0} \left[ \norm{z-\tilde z}^2 \right]\right)^{1/2}.
    \end{equation*}
    Taking the infimum over all couplings of $\pi_0$ and $\tilde \pi _0$ on the right hand side gives the desired inequality. The existence of a unique stationary solution to the joint PDE follows by a fixed point argument. Assume $t>0$ is fixed and $P_t$ is the solution operator of the PDE. Since $(\calP_2(\calZ), \calT)$ is a complete metric space ($\calT$ being equivalent to $\calW_2$), the semigroup contraction \eqref{eq:semigroup-contraction} implies that $P_t$ has a unique fixed point $\pi$ \cite{Sehgal1971}. Setting $\pi_0 = \pi$, it holds $\pi_t = P_t(\pi) = \pi$, so $\pi$ is a $t$-periodic solution of the PDE. For any other $s>0$, it holds $\calT(\pi_{s+t}, \pi_s) \le e^{-\omega s} \calT(\pi_t, \pi_0) = 0$. Hence $P_t(\pi_s) = \pi_s$, so $\pi_s = \pi$ by uniqueness of the fixed point; thus $\pi$ is stationary, and it is unique since every stationary solution is a fixed point of $P_t$.\\
    As convergence in $\calW_2$ (equivalently, in $\calT$) implies weak convergence, inserting $\tilde \pi_0 = \pi$ into \eqref{eq:semigroup-contraction} shows that any solution $\pi_t$ weakly converges to $\pi$.
\end{proof}

\begin{proof}[Proof of \cref{thm:convergence-pd-ct}]
    We employ \cref{lem:stability_general_dynamics_cont_time}, so the only part we need to show is the strong monotonicity of the negative drift coefficient $A_{\tup{PD}}$ defined by
    $$ A_{\tup{PD}} = \begin{pmatrix}
        \partial g & K^T\\
        -\lambda K & \lambda \partial f^\ast
    \end{pmatrix}. $$
    We write $z = (x,y) \in \calX \times \calY = \calZ$ for the joint primal and dual variable and equip $\calZ$ with the weighted inner product $\inpro{\cdot}{\cdot}_{1,\lambda^{-1}}$. Let now $w \in A_{\tup{PD}}(z)$ and $\tilde w \in A_{\tup{PD}}(\tilde z)$. Then there are $u \in \partial g (x)$ and $v \in \partial f^\ast (y)$ such that $w = (u+K^T y, \lambda( v - Kx))$. Analogously, if $\tilde w \in A_{\tup{PD}}(\tilde z)$ then there exist $\tilde u \in \partial g(\tilde x)$ and $\tilde v \in \partial f ^\ast (\tilde y)$ with $\tilde w = (\tilde u + K^T \tilde y, \lambda (\tilde v - K\tilde x))$. It follows that
    \begin{align*}
        \inpro{w-\tilde w}{z-\tilde z}_{1,\lambda^{-1}} &= \inpro{u - \tilde u}{x-\tilde x} + \lambda \inpro{v - \tilde v}{y-\tilde y}_{\lambda^{-1}} \\
        &\qquad + \inpro{K^T(y-\tilde y)}{x-\tilde x} - \lambda \inpro{K(x-\tilde x)}{y-\tilde y}_{\lambda^{-1}}\\
        &= \inpro{u - \tilde u}{x-\tilde x} + \inpro{v - \tilde v}{y-\tilde y}\\
        &\ge \omega_g \norm{x-\tilde x}^2 + \omega_{f^\ast} \norm{y-\tilde y}^2 \\
        &\ge \omega \norm{z-\tilde z}_{1,\lambda^{-1}}^2.
    \end{align*}
\end{proof}

\subsection{Proofs for \texorpdfstring{\cref{subsec:non-concentration}}{Section 3.2}}\label{subsec:appendix-continuous-time:non-concentration}

\begin{proof}[Proof of \cref{thm:1d-gauss-example}]
    From the diffusion equation, it follows that the distribution $\pi_t(z) = \pi_t(x,y)$ of $Z_t$ satisfies the Fokker--Planck equation
    $$ \partial_t \pi_t = \nabla_z \cdot (\pi_t A_{\tup{PD}}z) +  \Delta_x \pi_t,$$
    with an initial value $\pi_0$. The Cauchy problem can be solved in closed form; note that a similar computation has been carried out in \cite{Hoermander1967}. By taking the Fourier transform in the variable $z$ (denoted $\calF$), we obtain the first order initial value problem for $p_t(\zeta) = \calF(\pi_t)(\zeta)$ with $\zeta = (\xi,\omega)^T$:
    $$  \left\{\begin{array}{l}
        \partial_t p_t = - \zeta^T A_{\tup{PD}} \nabla_\zeta p_t - \xi^2 p_t\\
        p_0(\xi, \omega) = \calF \pi_0.
    \end{array}\right. $$
    The corresponding characteristic equations are
    \begin{align*}
        \dot \zeta(t) &= A_{\tup{PD}}^T \zeta(t)\\
        \dot p(t) &= - \xi^2(t) p(t),
    \end{align*}
    and they are solved by
    \begin{align*}
        \zeta(t) &= \exp(tA_{\tup{PD}}^T) \zeta_0 \\
        p(t) &= p_0(\zeta_0) \exp \left(- \int_0^t \xi^2(s) \rmd s \right).
    \end{align*}
    By eliminating the initial value $\zeta_0 = \exp(-tA_{\tup{PD}}^T) \zeta(t)$, the solution for the initial value problem is
    $$ p_t(\zeta) = p_0(\exp(-tA_{\tup{PD}}^T)\zeta) \exp \left( - \frac{1}{2} \int_0^t (B_{\tup{PD}}^T \exp(-sA_{\tup{PD}}^T) \zeta)^2 \rmd s \right).
    $$
    The quadratic form inside the integral is positive semidefinite in the variable $\zeta$. For $t>0$, it is also positive definite: Otherwise, there would need to be a $\zeta \neq 0$ such that $B_{\tup{PD}}^T \exp(-sA_{\tup{PD}}^T) \zeta = 0$ for all $s \in [0,t]$, which implies $B_{\tup{PD}}^T (A_{\tup{PD}}^T)^k \zeta = 0$ for all $k\ge 0$. This is equivalent to $(0,1)^T$ being an eigenvector of $A_{\tup{PD}}^T$, which is impossible since the eigenvectors of $A_{\tup{PD}}^T$ are of the form
    $$\left(\frac{c_g^{-1}-\lambda c_f}{2k} \pm \frac{1}{2k}\sqrt{(c_g^{-1} - \lambda c_f)^2-4\lambda k^2},1\right)^T,$$
    neither of which is a multiple of $(0,1)^T$. From the positive definiteness of the quadratic form it follows that $\abs{p_t(\zeta)} \le \exp(-c_t\norm{\zeta}^2)$ for some $c_t>0$, so $p_t$ decays faster than any polynomial for every $t>0$. In particular, its inverse Fourier transform $\pi_t$ is $C^\infty$-smooth in the variable $z$ for $t>0$, which proves \ref{item:1d-gauss-example-smoothness-of-soln}.

    If the initialization $\pi_0$ is Gaussian with mean $z_0 = (x_0,y_0)$ and (possibly singular) covariance $\Sigma_0$, we can simply take the inverse Fourier transform to obtain the solution $\pi_t$. It is worth noting, however, that we do not even have to compute the full distribution $\pi_t$ in order to analyze the $x$-marginal $\mu_t$ of $\pi_t$. Since marginalizing out the dual variable $y$ in $\pi_t$ corresponds to evaluating $p_t$ at $\zeta = (\xi, 0)$ in Fourier space, we get, denoting by $\calF^{-1}_\xi$ the inverse Fourier transform in $\xi$ and $E(t) := (e_{ij}(t))_{i,j} := \exp(-tA_{\tup{PD}}^T)$,
    \begin{align*}
        \mu_t(x) &:= \int \pi_t(x,y) \rmd y\\
        &= \calF^{-1}_\xi (p_t(\xi,0)) (x)\\
        &\propto \int \exp \left(i \xi (x - z_0 \cdot e_1(t)) - \frac{\xi^2}{2} \left(\norm{e_1(t)}_{\Sigma_0}^2 + 2\int_0^t e^2_{11}(s) \rmd s\right) \right) \rmd \xi\\
        &= \exp\left(- \frac12(x - z_0\cdot e_1(t))^2 \left(\norm{e_1(t)}_{\Sigma_0}^2 + 2 \int_0^t e^2_{11}(s) \rmd s \right)^{-1} \right),
    \end{align*}
    where we abbreviated $e_1(t)$ for the first column of $E(t)$. Hence, the primal variable follows a Gaussian distribution at every time $t>0$, with mean $z_0^T e_1(t)$ and variance $\norm{e_1(t)}_{\Sigma_0}^2 + 2\int_0^t e^2_{11}(s) \rmd s$. Due to the contraction result \cref{thm:convergence-pd-ct}, we know that $\mu_t$ converges to the stationary distribution $\mu_\lambda$ as $t\to \infty$. Indeed, this holds here since $\lim_{t\to \infty} E(t) = 0$, so $\bbE_{X \sim \mu_\lambda}[X] = 0$. The variance of $\mu_\lambda$ is
    $$ \Var_{X \sim \mu_\lambda}[X] = 2\int_0^\infty e^2_{11}(s) \rmd s = \frac{c_g(c_f + k^2c_g) + \lambda c_f^2c_g^2}{(c_f + k^2c_g)(1 + \lambda c_f c_g)}, $$
    which does not coincide with the variance of the target distribution $\mu^\ast$.
    The target is however approximated as $\lambda$ becomes large, since $\lim_{\lambda \to \infty} \Var[\mu_\lambda] = c_gc_f/(c_f + k^2 c_g) $.

    Analogously, the full solution $\pi_t$ of the Fokker--Planck equation can be computed by taking the inverse Fourier transform in $\zeta$ above. We obtain that $\pi_t$ is Gaussian at every time $t>0$ with mean $E(t)^Tz_0$ and positive definite covariance matrix
    $$ \Var_{Z\sim \pi_t}[Z] = E(t)^T \Sigma_0 E(t) + 2\begin{pmatrix}
        \int_0^t e_{11}^2 \rmd s & \int_0^t e_{11} e_{12} \rmd s\\
        \int_0^t e_{11} e_{12} \rmd s & \int_0^t e_{12}^2 \rmd s
    \end{pmatrix}.  $$
    In the limit $t\to \infty$, we obtain the stationary distribution which has mean zero and covariance matrix
    $$ \Var_{Z\sim \pi_\lambda}[Z] = \frac{1}{(c_f + k^2c_g)(1+\lambda c_f c_g)} \begin{pmatrix}
        c_g (c_f + k^2c_g) + \lambda c_f^2 c_g^2 & k \lambda c_f c_g^2 \\ k \lambda c_f c_g^2 & k^2 \lambda c_g^2
    \end{pmatrix}. $$
    In the limit $\lambda \to \infty$, the stationary distribution approaches a degenerate Gaussian (in the sense that the covariance matrix is singular), with mean zero and covariance matrix
    $$ \frac{1}{(c_f + k^2c_g)} \begin{pmatrix}
        c_f c_g & kc_g \\ kc_g & k^2 c_f^{-1} c_g
    \end{pmatrix}. $$
    This is identical to $(\id, f'(k\,\cdot))_{\#} \mu^\ast$, proving \ref{item:1d-gauss-example-soln-for-gaussinit}.
\end{proof}

\begin{proof}[Proof of \cref{thm:1d-gauss-no-homogeneous-correction}]
    We start by noting that the contraction result \cref{thm:convergence-pd-ct} also applies to the more general $\Sigma_{\tup{gPD}} = \frac 12 B_{\tup{gPD}} B_{\tup{gPD}}^T$ with any constant coefficient $B_{\tup{gPD}}$, since the relevant \cref{lem:stability_general_dynamics_cont_time} holds for any constant diffusion coefficient. In view of this contraction property, we can assume the initialization $Z_0 \sim \pi_0$ to be Gaussian with mean $z_0$, covariance $\Sigma_0$ and obtain $\pi_\lambda$ by computing $\lim_{t \to \infty} \pi_t$. The reason for initializing at a Gaussian is that we can again obtain the closed-form solution $\pi_t$ of the Cauchy problem for the Fokker--Planck equation. We carry out the same computations as in the proof of \cref{thm:1d-gauss-example}, showing that $\pi_t$ is Gaussian with mean $E(t)^Tz_0$ and covariance matrix
    $$ E(t)^T \Sigma_0 E(t) + \int_0^t E(s)^T B_{\tup{gPD}} B_{\tup{gPD}}^T E(s) \rmd s. $$
    The entries of $E$ can be computed explicitly (we omit the terms here) and by integrating and taking the limit $t\to \infty$, we find that the $x$-marginal of the stationary distribution $\pi_\lambda$ has variance
    \begin{equation*}
        \Var[\mu_\lambda] = \frac{(b_{11}^2+b_{12}^2)(\lambda c_f^2 c_g^2 + c_g(c_f + k^2c_g)) - 2(b_{11}b_{21}+b_{12}b_{22})kc_f c_g^2 + (b_{21}^2+b_{22}^2)\lambda^{-1}k^2c_g^2}{2(1+\lambda c_f c_g)(c_f+k^2 c_g)}.
    \end{equation*}
    We want to show that there is no choice of $b_{ij}$ (independent of $c_f,c_g$ and $k$) such that the target $\mu^\ast$ coincides with the $x$-marginal $\mu_\lambda$ of the stationary solution. The variance of both distributions is equal if and only if
    $$ 2c_g c_f (1+\lambda c_g c_f) = (b_{11}^2+b_{12}^2)(\lambda c_f^2 c_g^2 + c_g(c_f + k^2c_g)) - 2(b_{11}b_{21}+b_{12}b_{22})kc_f c_g^2 + (b_{21}^2+b_{22}^2)\lambda^{-1}k^2c_g^2. $$
    Isolating the terms containing $k^2$ shows that this implies
    $$ 0 = (b_{11}^2+b_{12}^2) c_g^2 + (b_{21}^2+b_{22}^2) \lambda^{-1} c_g^2, $$
    which is only satisfied if all $b_{ij}$ are zero since $\lambda > 0$. Hence, a problem-independent choice for $B_{\tup{gPD}}$ that results in the target being invariant with respect to the diffusion does not exist.
\end{proof}

\begin{proof}[Proof of \cref{thm:pd-hypoellipticity}]
    We apply \cref{lem:hoermander-condition-hypoellipticity}. $\calL_{\tup{PD}}$ has the correct form with $k = d$ vector fields $A_j = \partial_{x_j}$, $j = 1,\dots, d$ and
    $$A_0 = - A_{\tup{PD}}^T(z)\nabla_z = \sum_{j = 1}^d \left(- \frac{\partial g}{\partial x_j}(x) - k_j^Ty\right)\partial_{x_j} + \lambda \sum_{i = 1}^m \left((k^i)^Tx - \frac{\partial f^\ast}{\partial y_i}(y)\right)\partial_{y_i},$$
    where we denoted the columns of $K$ by $k_1,\dots,k_d$ and the rows by $k^1,\dots,k^m$.
    Clearly, the vector fields $A_1,\dots,A_d$ span the first $d$ dimensions $\partial_{x_j}$, $j = 1,\dots,d$, i.e. $\calA_0(z) = \calX \times \{0\}$ for all $z \in \calZ$. The Lie brackets in $\calA_1$ are given by
    $$ [A_j,A_0] = - (H_g(x))_j^T \nabla_x + \lambda k_j^T\nabla_y, $$
    where $H_g$ is the Hessian of $g$. Due to $\lambda > 0$ and the full rank assumption, the columns $k_j$ of $K$ span $\calY$ and we have $\calA_1(z) = \calZ$ for all $z \in \calZ$.
\end{proof}

\subsection{Proofs for \texorpdfstring{\cref{subsec:convergence-to-ula}}{Section 3.3}}\label{subsec:appendix-continuous-time:convergence-to-ula}

\begin{proof}[Proof of \cref{thm:ct-convergence-to-overdamped-langevin}]
Once we have proved the first part of the statement, the second part follows by taking the limit $t \to \infty$ and using weak convergence of $\pi_t$ to $\pi_\lambda$.

Assume $Z_t = (X_t,Y_t)$ with distribution $\pi_t$ solves \eqref{eq:ct-pd-diffusion-jointvar}. Now consider a coupling with a process $\tilde X_t$ which is initialized as $\tilde X_0 \sim \mu^\ast$ and solves \eqref{eq:ld-artificial-pd} with the same Brownian motion $W_t$ as in \eqref{eq:ct-pd-diffusion-jointvar}. Since $\mu^\ast$ is the invariant distribution of \eqref{eq:ld-artificial-pd}, it follows $\tilde X_t \sim \mu^\ast$ for all $t>0$. The joint measure $\Pi_t$ of $(X_t,Y_t,\tilde X_t)$ solves the Fokker--Planck equation
\begin{equation}\label{eq:fp_coupledlambda}
\begin{aligned}
     \frac{\partial}{\partial t} \Pi_t &= \nabla_x \cdot ( (p + K^T y) \Pi_t) + \lambda \nabla_y \cdot ( (q - Kx) \Pi_t) + \nabla_{\tilde x}\cdot ( (\tilde p + K^T \nabla f(K\tilde x)) \Pi_t) \\
     &\qquad +  (\nabla_x + \nabla_{\tilde x}) \cdot (\nabla_x + \nabla_{\tilde x})\Pi_t,
\end{aligned}
\end{equation}
where as before $p \in \partial g(x), q \in \partial f^\ast (y), \tilde p \in \partial g(\tilde x)$. With the artificial, concentrated dual variable $\tilde Y_t = \nabla f(K \tilde X_t)$, we denote the singular joint law of $(\tilde X_t,\tilde Y_t)$ by $\pi^\ast = (\id,\nabla f \circ K)_{\#} \mu^\ast$.

In order to prove dissipation along the solution $\Pi_t$ of \eqref{eq:fp_coupledlambda}, the energy we consider is
\begin{align*}
    \phi(t) &= \bbE_{(x,y,\tilde x)\sim \Pi_t} \left[ d_{1,\lambda^{-1}}(x,y,\tilde x)^2 \right]\\
    &= \bbE_{(x,y,\tilde x)\sim \Pi_t} \left[ \norm{x-\tilde x}^2 + \norm{y-\nabla f(K\tilde x)}_{\lambda^{-1}}^2 \right]
\end{align*}
where we introduced an augmented distance function $d_{1,\lambda^{-1}}(x,y,\tilde x)$. Note that this distance is related to the transport distance $\calT_{1,\lambda^{-1}}$ on $\calP_2(\calZ)$: As the mass of $\pi^*$ is concentrated on the graph $\{y =\nabla f(Kx)\}$, for all $\pi \in \calP_2(\calZ)$ we can write
$$\calT_{1,\lambda^{-1}}(\pi,\pi^\ast) = \inf_{\Pi \in \Gamma(\pi,\mu^\ast)} \left( \bbE_{(x,y,\tilde x)\sim \Pi} \left[ d_{1,\lambda^{-1}}(x,y,\tilde x)^2 \right] \right)^{1/2}.$$
We now denote $\tilde y(\tilde x) \coloneqq \nabla f(K\tilde x)$ as well as $M \coloneqq (\nabla \tilde y)^T = K^T H_f(K\,\cdot)$ and $m \coloneqq \nabla \cdot \nabla \tilde y$. By using the joint Fokker--Planck equation \eqref{eq:fp_coupledlambda}, we obtain
\begin{align}
    &\frac{1}2 \frac{\rmd}{\rmd t} \phi(t) \notag\\
    &= \frac12\int d_{1,\lambda^{-1}}(x,y,\tilde x)^2 \left(\nabla_x \cdot ( (p + K^T y) \Pi_t) + \lambda \nabla_y \cdot ( (q - Kx) \Pi_t) + \nabla_{\tilde x} \cdot ( (\tilde p + K^T \tilde y(\tilde x) ) \Pi_t)\right) \mathrm d (x,y,\tilde x) \notag\\
    & \qquad+ \frac12\int d_{1,\lambda^{-1}}(x,y,\tilde x)^2 \left((\nabla_x + \nabla_{\tilde x}) \cdot (\nabla_x + \nabla_{\tilde x})\Pi_t\right) \rmd(x,y,\tilde x) \notag\\
    &= - \int \left( (x-\tilde x) \cdot (p+K^T y) + (y- \tilde y(\tilde x)) \cdot (q - Kx) + (\tilde x - x) \cdot (\tilde p + K^T \tilde y(\tilde x))\right) \rmd \Pi_t \notag\\
    &\qquad - \frac 1 \lambda \int \left(M(\tilde x) (\tilde y(\tilde x) - y)\right) \cdot (\tilde p + K^T \tilde y(\tilde x)) \rmd \Pi_t  - \frac 1 \lambda \int \left(M(\tilde x) (\tilde y(\tilde x) - y)\right) \cdot
    (\nabla_x + \nabla_{\tilde x}) \Pi_t \rmd(x,y,\tilde x) \notag\\
    &= - \int (x-\tilde x) \cdot (p - \tilde p) \rmd \Pi_t -  \int (y-\tilde y(\tilde x)) \cdot (q - K\tilde x) \rmd \Pi_t\notag\\
    &\qquad + \frac 1 \lambda \int \left[ (\tilde y(\tilde x) - y) \cdot \left(m(\tilde x) -M^T(\tilde x)(\tilde p + K^T \tilde y(\tilde x))\right) + \norm{M(\tilde x)}_F^2 \right] \rmd \Pi_t.\label{eq:proof-dissipation-primal-convergence}
\end{align}
We now bound all terms in the last expression one by one: Strong convexity of $g$ and $f^\ast$ gives
\begin{gather}
    \int (x-\tilde x) \cdot (p - \tilde p) \rmd \Pi_t \ge \omega_g \bbE_{(x,y,\tilde x) \sim \Pi_t} \left[ \norm{x-\tilde x}^2 \right],\label{eq:proof-inequality-1a}\\
    \int (y-\tilde y(\tilde x)) \cdot (q - K \tilde x) \rmd \Pi_t \ge \lambda \omega_{f^\ast} \bbE_{(x,y,\tilde x) \sim \Pi_t} \left[ \norm{y-\tilde y(\tilde x)}_{\lambda^{-1}}^2 \right].\label{eq:proof-inequality-1b}
\end{gather}
Since the $\tilde x$-marginal of $\Pi_t$ is the target $\mu^\ast$, by assumption we have
\begin{equation}\label{eq:proof-inequality-2}
    \lambda^{-1} \int \norm{M(\tilde x)}_F^2 \rmd \Pi_t \le \lambda^{-1} D_1.
\end{equation}
In a similar way, using the Cauchy--Schwarz inequality and Young's inequality, for any $\delta > 0$ we can bound the final term by
\begin{align}
    \lambda^{-1} \int (\tilde y(\tilde x) - y) &\cdot \left(m(\tilde x) -M^T(\tilde x)(\tilde p + K^T \tilde y(\tilde x))\right) \rmd \Pi_t \notag\\
    &\le \lambda^{-1} \int \norm{\tilde y(\tilde x) - y} \cdot \norm{m(\tilde x) -M^T(\tilde x)(\tilde p + K^T \tilde y(\tilde x))} \rmd \Pi_t\notag\\
    &\le \frac{\delta}{2} \bbE_{(x,y,\tilde x) \sim \Pi_t} \left[ \norm{y-\tilde y(\tilde x)}_{\lambda^{-1}}^2 \right] + \frac{1}{2\delta\lambda} \bbE_{\tilde x \sim \mu^\ast} \left[ \norm{m(\tilde x) -M^T(\tilde x)(\tilde p + K^T \tilde y(\tilde x))}^2 \right]\notag\\
    &\le \frac{\delta}{2} \bbE_{(x,y,\tilde x) \sim \Pi_t} \left[ \norm{y-\tilde y(\tilde x)}_{\lambda^{-1}}^2 \right] + \frac{ D_2 + D_3}{\delta\lambda}\label{eq:proof-inequality-3}.
\end{align}
If we set $\delta = 2(\lambda\omega_{f^*} - \omega_g)$, inserting \eqref{eq:proof-inequality-1a}, \eqref{eq:proof-inequality-1b}, \eqref{eq:proof-inequality-2} and \eqref{eq:proof-inequality-3} into \eqref{eq:proof-dissipation-primal-convergence} gives
$$ \phi'(t) \le -2\omega_g \phi(t) + \frac{2D_1}{\lambda} + \frac{D_2 + D_3}{\lambda(\lambda \omega_{f^*} - \omega_g)}.$$
In order to abbreviate the bias term on the right, we define
$$ C_1(\lambda) \coloneqq \left( \frac{D_1}{\lambda \omega_g} + \frac{D_2 + D_3 }{2 \omega_g \lambda (\lambda \omega_{f^\ast} - \omega_g)} \right)^{1/2} = (\lambda \omega_g/D_1)^{-1/2} + \calO(\lambda^{-3/2}), $$
where we give the decay of $C_1$ in the limiting case $\lambda \to \infty$. Gronwall's lemma now implies
\begin{align*}
    \phi(t) &\le e^{-2\omega_g t} \phi(0) + 2\omega_g C_1^2(\lambda) \int_0^t e^{-2\omega_g(t-s)}\rmd s\\
    &= e^{-2\omega_g t} \phi(0) + C_1^2(\lambda)(1-e^{-2\omega_g t})\\
    &\le e^{-2\omega_g t} \phi(0) + C_1^2(\lambda).
\end{align*}
In the last line we dropped the negative bias term, which rapidly converges to zero anyway. By definition, we can lower bound the left side by $\calT_{1,\lambda^{-1}}^2(\pi_t, \pi^*) \le \phi(t)$. Taking the infimum over all possible initializations $\Pi_0$ with $(x,y)$-marginal $\pi_0$ and $(\tilde x, \tilde y)$-marginal $\pi^\ast = (\id,\nabla f \circ K)_{\#}\mu^\ast$ on the right side, we get
\begin{equation}\label{eq:proof-conv-to-ULA-intermediate-result}
    \calT_{1,\lambda^{-1}}^2(\pi_t,\pi^\ast) \le \phi(t) \le e^{-2\omega_g t} \calT_{1,\lambda^{-1}}^2(\pi_0,\pi^\ast) + C_1^2(\lambda).
\end{equation}
The left side of \eqref{eq:proof-conv-to-ULA-intermediate-result} is lower bounded by $\calW_2^2(\mu_t,\mu^\ast) \le \calT^2_{1,\lambda^{-1}}(\pi_t,\pi^\ast)$. Taking the square root of \eqref{eq:proof-conv-to-ULA-intermediate-result} and using $\sqrt{a+b}\le\sqrt{a}+\sqrt{b}$ therefore gives
\begin{equation*}
    \calW_2(\mu_t,\mu^\ast) \le e^{-\omega_g t} \calT_{1,\lambda^{-1}}(\pi_0,\pi^\ast) + C_1(\lambda).
\end{equation*}
Finally, assuming a concentrated initialization with $\mu_0 \in \calP_2(\calX)$ and $Y_0 = \nabla f(KX_0)$, we can use that
$$\calT_{1,\lambda^{-1}}(\pi_0,\pi^\ast) \le (1+\lambda^{-1} \omega_{f^\ast}^{-2} \norm{K}^2)^{1/2} \calW_2(\mu_0,\mu^\ast)$$
and obtain the inequality from the statement.
\end{proof}

\subsection{Proofs for \texorpdfstring{\cref{subsec:bias-corrected-diffusion}}{Section 3.4}}\label{subsec:appendix-continuous-time:bias-corrected-diffusion}
For the proof of \cref{thm:stationary-soln-of-modified-fp}, we need the following vector calculus identity.
\begin{lemma}\label{lem:weird-chain-rule}
    Let $\psi:  \bbR^{d} \times \bbR^{m} \rightarrow  \bbR^{n}$, $\eta: \bbR^{d} \rightarrow \bbR^{m}$ and $S: \mathbb{R}^{d} \rightarrow \mathbb{R}^{d \times n}$, each with $C^1$ entries. Define $r:\bbR^d \to \bbR^n$ by $r_j(x) \coloneqq \sum_{i} \frac{\partial S_{ij}}{\partial x_i}(x) $ and denote the partial Jacobians of $\psi$ by $D_x\psi : \bbR^d\times \bbR^m \to \bbR^{n\times d}$ with $(D_x\psi)_{i,j}(x,y) \coloneqq \frac{\partial \psi_i}{\partial x_j}(x,y)$ and $D_y\psi : \bbR^d\times \bbR^m \to \bbR^{n\times m}$ where $(D_y\psi)_{i,k}(x,y) \coloneqq \frac{\partial \psi_i}{\partial y_k}(x,y)$. Then for all $x \in \bbR^d$ it holds
    \begin{equation}\label{eq:weird_chain_rule}
        \nabla_x \cdot ( S(x) \psi(x, \eta(x))) = r(x) \cdot \psi(x, \eta(x)) + \tr(S(x) D_x \psi(x,\eta(x))) + \tr(S(x) D_y\psi(x,\eta(x)) D_x \eta(x)).
    \end{equation}
\end{lemma}

\begin{proof}
    The proof is a straightforward application of product and chain rule.
    \begin{align*}
    \nabla_x \cdot (& S(x) \psi(x, \eta(x))) \\
    &= \sum_{i = 1}^{d} \frac{\partial}{\partial x_i} (\sum_{j=1}^{n} S_{ij}(x) \psi_j(x, \eta(x)))\\
    &=\sum_{ij} \frac{\partial S_{ij}}{\partial x_i}(x) \psi_j(x ,\eta(x))
    + S_{ij}(x) \frac{\partial \psi_j}{\partial x_i} (x, \eta(x))
    + S_{ij}(x) \sum_{k=1}^m \frac{\partial \psi_j}{\partial y_k}(x,\eta(x)) \frac{\partial \eta_k}{\partial x_i}(x)\\
    &= r(x) \cdot \psi(x, \eta(x)) + \tr(S(x) D_x \psi(x,\eta(x))) + \tr(S(x) D_y\psi(x,\eta(x)) D_x \eta(x)).
    \end{align*}
\end{proof}

\begin{proof}[Proof of \cref{thm:stationary-soln-of-modified-fp}]
    It suffices to check that the right side of \eqref{eq:modified-weak-fp} evaluates to zero for any $\phi \in C_c^\infty(\calZ)$ when we insert $\pi^\ast$:
    \begin{subequations}
    \begin{align}
        \int%
        \begin{pmatrix}I & M(x)\\ M(x)^T & M(x)^TM(x) \end{pmatrix}%
        \,:\, 
        \underbrace{\begin{pmatrix}
        \nabla_x^2 \phi& D_y \nabla_x\phi\\
        D_x \nabla_y \phi&\nabla_y^2 \phi
        \end{pmatrix}}_{=\nabla_z^2 \phi} \rmd \pi^\ast%
        -%
        \int \nabla_x \phi \cdot (\nabla g(x)+K^Ty) \rmd \pi^\ast&%
        \label{eq:long-expression-proof-modified-fp-line1}\\%
        -%
        \int \nabla_y \phi \cdot (\lambda \nabla f^*(y) - \lambda K x) \rmd \pi^\ast%
        -%
        \int \nabla_y \phi \cdot M(x)^T (\nabla g(x) + K^T \nabla f(Kx)) \rmd \pi^\ast&%
        \label{eq:long-expression-proof-modified-fp-line2}\\%
        +\int \nabla_y \phi \cdot r \rmd \pi^\ast.
        \label{eq:long-expression-proof-modified-fp-line3}
    \end{align}
    \label{eq:long-expression-proof-modified-fp}
    \end{subequations}
    \noindent Here $D$ denotes the total derivative of a function $u:\bbR^m\rightarrow\bbR^d$, i.e. the matrix $(Du)_{i,j}=\partial_{j} u_i$. Using $(\nabla f)^{-1} = \nabla f^\ast$ and the fact that $\pi^\ast$ is partially concentrated on $y = \nabla f (Kx)$, it follows that
    $$ \int \nabla_y \phi \cdot (\lambda \nabla f^\ast(y) - \lambda Kx) \rmd \pi^\ast = 0. $$
    Using partial integration and \cref{lem:weird-chain-rule} with $S\equiv I$, $\psi = \nabla_x\phi$ and $\eta = \nabla f \circ K$, the second integral in \eqref{eq:long-expression-proof-modified-fp-line1} can be rewritten as
    \begin{align*}
        - \int \nabla_x &\phi \cdot (\nabla g(x)+K^Ty) \rmd \pi^\ast\\
        & = \frac{1}{Z} \int \nabla_x \phi(x,\nabla f(Kx)) \cdot (- \nabla g(x) - K^T\nabla f(Kx)) \exp(-f(Kx)-g(x)) \rmd x\\
        & = \frac{1}{Z} \int \nabla_x \phi(x,\nabla f(Kx)) \cdot \nabla_x (\exp(-f(Kx)-g(x))) \rmd x\\
        & = - \int \nabla_x \cdot (\nabla_x \phi(x,\nabla f(Kx))) \rmd \mu^\ast(x)\\
        &= - \int \tr(D_x \nabla_x \phi(x,y)) \rmd \pi^\ast(x,y) - \int \tr(D_y \nabla_x \phi(x,y) M(x)^{T}) \rmd \pi^\ast(x,y)\\
        &=-\int \Delta_x\varphi(x,y) \rmd\pi^\ast(x,y)-\int M(x)\,:\, D_y\nabla_x \varphi(x,y) \rmd\pi^\ast(x,y)
    \end{align*}
    where we used that $M(x)^T = D_x \eta(x)$ for $\eta = \nabla f \circ K$ and, under abuse of notation, the equality $\nabla \cdot  =\tr\circ D $. These two terms cancel with the terms from the first column of the diffusion matrix in \eqref{eq:long-expression-proof-modified-fp-line1}. The second integral in \eqref{eq:long-expression-proof-modified-fp-line2} can also be reformulated using partial integration and then \cref{lem:weird-chain-rule}, this time with $S = M$, $\psi = \nabla_y \phi$ and $\eta = \nabla f \circ K$. In the notation of \cref{lem:weird-chain-rule}, we get $r:\bbR^d \to \bbR^m$ with $r_j(x) \coloneqq \sum_{i} \frac{\partial M_{ij}}{\partial x_i}(x) $ and obtain
    \begin{align*}
        - \int &\nabla_y \phi(x,y) \cdot M(x)^T (\nabla g(x) + K^T \nabla f(Kx)) \rmd \pi^\ast\\
        & = \frac{1}{Z} \int \nabla_y \phi(x,\nabla f(Kx)) \cdot M(x)^T (- \nabla g(x) - K^T \nabla f(Kx)) \exp(-f(Kx)-g(x)) \rmd x\\
        & = \frac{1}{Z} \int \nabla_y \phi(x,\nabla f(Kx)) \cdot M(x)^T \nabla_x (\exp(-f(Kx)-g(x))) \rmd x\\
        & = - \int \nabla_x \cdot (M(x) \nabla_y \phi(x,\nabla f(Kx))) \rmd \mu^\ast(x)\\
        & = - \int r(x) \cdot \nabla_y \phi(x,y) \rmd \pi^\ast - \int \tr(M(x) D_x \nabla_y \phi(x,y)) \rmd \pi^\ast \\
        &\qquad\qquad - \int \tr(M(x) D_y \nabla_y \phi(x,y) M(x)^T) \rmd \pi^\ast\\
        &= - \int r(x) \cdot \nabla_y \phi(x,y) \rmd \pi^\ast - \int M^T(x)\,:\, D_x \nabla_y \phi(x,y) \rmd \pi^\ast\\
        &\qquad\qquad-\int M^TM\,:\, D_y \nabla_y \phi(x,y) \rmd \pi^\ast.
    \end{align*}
    For the last equality we use $\tr(M D_y\nabla_y \phi M^T) = \tr(M^T M D_y\nabla_y \phi) =M^TM\,:\, D_y \nabla_y \phi$. These terms cancel with the second column of the diffusion matrix in \eqref{eq:long-expression-proof-modified-fp-line1} and the term in \eqref{eq:long-expression-proof-modified-fp-line3}. It follows that \eqref{eq:long-expression-proof-modified-fp} is equal to zero.
\end{proof}

\section{Proofs for \texorpdfstring{\cref{sec:discrete-time}}{Section 4}}\label{sec:appendix-discrete-time}
\subsection{Proofs for \texorpdfstring{\cref{subsec:contraction-primal-dual-langevin}}{Section 4.1}}\label{subsec:appendix-discrete-time:contraction}

In order to prove \cref{thm:dt-stability-convex,thm:dt-contraction-strongly-convex}, we need the following auxiliary inequality, which has been used several times in similar forms for primal-dual optimization algorithms, see e.g. \cite{Chambolle2011}. It still holds in the present sampling setting, since considering the difference of two coupled processes makes the stochastic terms $\xi^n$ cancel out:

\begin{lemma}\label{lem:primal-dual-ineq}
    Let $(X^n,Y^n)$ and $(\tilde X^n, \tilde Y^n)$, $n \ge 0$ be given by the coupling \eqref{eq:coupling-ulpda} with overrelaxation parameter $\theta$ chosen such that $\theta \sigma \tau L^2 \le 1$. By definition of the proximal mapping, we have for $n \ge 0$:
     \begin{equation}\label{eq:subgradients-ulpda}
         \begin{array}{l}
            p^{n+1} \coloneqq (X^n-X^{n+1}+\sqrt{2\tau}\xi^{n+1})/\tau - K^T Y^{n+1} \in \partial g(X^{n+1}-\sqrt{2\tau} \xi^{n+1}) \vspace{0.5em}\\
            q^{n+1} \coloneqq (Y^n-Y^{n+1})/\sigma + K X_\theta^n \in \partial f^\ast(Y^{n+1}).
        \end{array}
     \end{equation}
     Analogously, we define $\tilde p^{n+1}, \tilde q^{n+1}$ for the second chain $(\tilde X^n,\tilde Y^n)$. Then it holds
    \begin{align*}
        &2 \left(D_{g,\text{sym}}^{p^{n+1},\tilde p^{n+1}}(X^{n+1}- \sqrt{2\tau}\xi^{n+1}, \tilde X^{n+1}- \sqrt{2\tau}\xi^{n+1}) + D_{f^\ast,\text{sym}}^{q^{n+1},\tilde q^{n+1}}(Y^{n+1}, \tilde Y^{n+1})\right)\\
        &\quad + \norm{X^{n+1}-\tilde X^{n+1}}_{\tau^{-1}}^2 +\norm{X^{n+1}-\tilde X^{n+1}-X^n+\tilde X^n}_{\tau^{-1}}^2 + \norm{Y^{n+1}-\tilde Y^{n+1}}_{\sigma^{-1}}^2 \\
        &\quad + 2\inpro{K(X^{n+1}-\tilde X^{n+1}-X^n+\tilde X^n)}{Y^{n+1}-\tilde Y^{n+1}}\\
        &\le \norm{X^n-\tilde X^n}_{\tau^{-1}}^2 + \theta \norm{X^n-\tilde X^n-X^{n-1}+\tilde X^{n-1}}_{\tau^{-1}}^2 + \norm{Y^n - \tilde Y^n}_{\sigma^{-1}}^2\\
        &\quad + 2\theta \inpro{K(X^n-\tilde X^n - X^{n-1}+\tilde X^{n-1})}{Y^n-\tilde Y^n},
    \end{align*}
    where we use the same notation $\norm{\cdot}^2_a = a\norm{\cdot}^2$ with the Euclidean norm $\norm{\cdot}$ as in \cref{subsubsec:contraction-general-diffusion}.
\end{lemma}

\begin{proof}
    In order to shorten the expressions, throughout the proof we abbreviate $U^n := X^n-\tilde X^n$, $U_\theta^n = X_\theta^n-\tilde X_\theta^n$ and $V^n := Y^n - \tilde Y^n$.\\
    By plugging in the representation of the subgradient $p^{n+1}$ of $g$ at $X^{n+1} - \sqrt{2\tau}\xi^{n+1}$ (and likewise $\tilde p^{n+1}$) given in the statement of the lemma, we can rewrite a symmetric Bregman divergence as
    \begin{align*}
        &2 D_{g,\text{sym}}^{p^{n+1},\tilde p^{n+1}}(X^{n+1}- \sqrt{2\tau}\xi^{n+1}, \tilde X^{n+1}- \sqrt{2\tau}\xi^{n+1}) \\
        &\quad = 2\inpro{U^{n+1}}{p^{n+1} - \tilde p^{n+1}} \\
        &\quad = 2\inpro{U^{n+1}}{(X^n-X^{n+1})/\tau-K^T Y^{n+1} - (\tilde X^n-\tilde X^{n+1})/\tau+K^T\tilde Y^{n+1}} \\
        &\quad = \frac{2}{\tau} \inpro{U^{n+1}}{U^n - U^{n+1}} - 2\inpro{U^{n+1}}{K^T V^{n+1}} \\
        &\quad = \norm{U^n}_{\tau^{-1}}^2 - \norm{U^{n+1}}_{\tau^{-1}}^2 - \norm{U^{n+1} - U^n}_{\tau^{-1}}^2 - 2\inpro{KU^{n+1}}{V^{n+1}},
    \end{align*}
    where we used the identity $-\inpro{A-B}{A} = \frac{1}{2}\norm{B}^2 - \frac{1}{2}\norm{A}^2 - \frac{1}{2}\norm{A-B}^2$ in the last equality. Repeating the same calculation for the subgradients $q^{n+1}$ of $f^\ast$ at $Y^{n+1}$, we further get
    \begin{align*}
        2D_{f^\ast,\text{sym}}^{q^{n+1},\tilde q^{n+1}}(Y^{n+1}, \tilde Y^{n+1}) &= 2\inpro{V^{n+1}}{(Y^n-Y^{n+1})/\sigma+KX_\theta^{n} - (\tilde Y^n-\tilde Y^{n+1})/\sigma - K \tilde X_\theta^{n}}\\
        &= \frac{2}{\sigma} \inpro{V^{n+1}}{V^n-V^{n+1}} + 2\inpro{V^{n+1}}{K U_\theta^{n}} \\
        &= \norm{V^n}_{\sigma^{-1}}^2 - \norm{V^{n+1}}_{\sigma^{-1}}^2 - \norm{V^{n+1}-V^n}_{\sigma^{-1}}^2 + 2\inpro{K U_\theta^{n}}{V^{n+1}}.
    \end{align*}
    Adding up the two equalities above gives
    \begin{align}\label{eq:discretetime-lemma-proof-identity}
        2\inpro{U^{n+1}}{p^{n+1}-\tilde p^{n+1}} + 2\inpro{V^{n+1}}{q^{n+1}-\tilde q^{n+1}} &= \norm{U^n}_{\tau^{-1}}^2 - \norm{U^{n+1}}_{\tau^{-1}}^2 - \norm{U^{n+1} - U^n}_{\tau^{-1}}^2 \notag\\
        &\quad + \norm{V^n}_{\sigma^{-1}}^2 - \norm{V^{n+1}}_{\sigma^{-1}}^2 - \norm{V^{n+1} - V^n}_{\sigma^{-1}}^2 \notag\\
        &\quad - 2\inpro{KU^{n+1}}{V^{n+1}} + 2\inpro{KU_\theta^{n}}{V^{n+1}}.
    \end{align}
    By definition it holds $U_\theta^n = U^n + \theta(U^n -U^{n-1})$, so the two inner products in the last line can be rewritten in the following way
    \begin{align}
        &- 2\inpro{KU^{n+1}}{V^{n+1}} + 2\inpro{KU_\theta^{n}}{V^{n+1}} \notag \\
        &= -2\inpro{K(U^{n+1} - U^n)}{V^{n+1}} + 2\theta \inpro{ K(U^n-U^{n-1})}{V^{n+1}} \notag \\
        &= -2\inpro{K(U^{n+1} - U^n)}{V^{n+1}} + 2\theta \inpro{K(U^n-U^{n-1})}{V^n} + 2\theta \inpro{K(U^n-U^{n-1})}{V^{n+1}-V^n}\notag \\
        &\le -2\inpro{K(U^{n+1} - U^n)}{V^{n+1}} + 2\theta \inpro{K(U^n-U^{n-1})}{V^n} \notag \\
        &\quad + \theta \norm{U^n-U^{n-1}}_{\tau^{-1}}^2 + \tau \sigma \theta L^2 \norm{V^{n+1}-V^{n}}_{\sigma^{-1}}^2.\label{eq:inequality-bilinear-terms}
    \end{align}
    Inserting this into \eqref{eq:discretetime-lemma-proof-identity} and using $(1-\sigma \tau \theta L^2)\norm{V^{n+1}-V^n}_{\sigma^{-1}}^2 \ge 0$ gives the statement of the Lemma.
\end{proof}

\begin{proof}[Proof of \cref{thm:dt-stability-convex}]
    In the convex case, we start by noting that symmetric Bregman distances for convex functions are always non-negative, i.e.
    \begin{gather*}
        D_{g,\text{sym}}^{p^{n+1},\tilde p^{n+1}}(X^{n+1}- \sqrt{2\tau}\xi^{n+1}, \tilde X^{n+1}- \sqrt{2\tau}\xi^{n+1}) \ge 0,\\
        D_{f^\ast,\text{sym}}^{q^{n+1},\tilde q^{n+1}}(Y^{n+1}, \tilde Y^{n+1}) \ge 0.
    \end{gather*}
    Setting $\theta = 1$, the inequality in \cref{lem:primal-dual-ineq} hence gives
    \begin{align*}
        & \norm{X^{n+1}-\tilde X^{n+1}}_{\tau^{-1}}^2 +\norm{X^{n+1}-\tilde X^{n+1}-X^n+\tilde X^n}_{\tau^{-1}}^2 + \norm{Y^{n+1}-\tilde Y^{n+1}}_{\sigma^{-1}}^2 \\
        &\quad + 2\inpro{K(X^{n+1}-\tilde X^{n+1}-X^n+\tilde X^n)}{Y^{n+1}-\tilde Y^{n+1}} \\
        &\le \norm{X^n-\tilde X^n}_{\tau^{-1}}^2 +\norm{X^n-\tilde X^n-X^{n-1}+\tilde X^{n-1}}_{\tau^{-1}}^2 + \norm{Y^n - \tilde Y^n}_{\sigma^{-1}}^2\\
        &\quad + 2\inpro{K(X^n-\tilde X^n - X^{n-1}+\tilde X^{n-1})}{Y^n-\tilde Y^n}.
    \end{align*}
    Iterating this inequality over $n = 0,\dots,N-1$ steps and using $X^{-1}=X^0$, $\tilde X^{-1} = \tilde X^0$ gives
    \begin{align*}
        &\norm{X^N-\tilde X^N}_{\tau^{-1}}^2 + \norm{Y^N-\tilde Y^N}_{\sigma^{-1}}^2 + \norm{X^N-\tilde X^N-X^{N-1}+\tilde X^{N-1}}_{\tau^{-1}}^2 \\
        &\quad + 2\inpro{K(X^N-\tilde X^N-X^{N-1}+\tilde X^{N-1})}{Y^N-\tilde Y^N}\\
        &\le \norm{X^{0}-\tilde X^{0}}_{\tau^{-1}}^2 + \norm{Y^{0}-\tilde Y^{0}}_{\sigma^{-1}}^2.
    \end{align*}
    Using the bound $2\inpro{a}{b} \ge - c^{-1}\norm{a}^2 - c \norm{b}^2$ with $c=\tau L^2$ on the remaining inner product implies
    \begin{align*}
        &\norm{X^N-\tilde X^N}_{\tau^{-1}}^2 + (1-\tau\sigma L^2)\norm{Y^N-\tilde Y^N}_{\sigma^{-1}}^2 \le \norm{X^{0}-\tilde X^{0}}_{\tau^{-1}}^2 + \norm{Y^{0}-\tilde Y^{0}}_{\sigma^{-1}}^2.
    \end{align*}
    Taking expectations on both sides, by definition of the Wasserstein distance, the left side is lower bounded by $\tau^{-1} \calW_2^2(\mu^N, \tilde \mu^N)$. Since the resulting inequality holds for all couplings of initializations $(X^0,Y^0)$ and $(\tilde X^0, \tilde Y^0)$, taking the infimum over all couplings on the right side gives $\tau^{-1} \calT^2_{1, \lambda^{-1}}(\pi^0, \tilde \pi^0)$, proving the first inequality. The second inequality in the statement of \cref{thm:dt-stability-convex} results from lower bounding the left side not by the Wasserstein distance in the primal variable, but instead by $(1-\tau\sigma L^2) \tau^{-1} \calT^2_{1,\lambda^{-1}}(\pi^N,\tilde \pi^N)$.
\end{proof}

\begin{proof}[Proof of \cref{thm:dt-contraction-strongly-convex}]
    We now assume strong convexity of $g,f^\ast$. In that case, the symmetric Bregman distances obey the following lower bounds
    \begin{gather*}
        D_{g,\text{sym}}^{p^{n+1},\tilde p^{n+1}}(X^{n+1}- \sqrt{2\tau}\xi^{n+1}, \tilde X^{n+1}- \sqrt{2\tau}\xi^{n+1}) \ge \omega_g \norm{X^{n+1} - \tilde X^{n+1}}^2,\\
        D_{f^\ast,\text{sym}}^{q^{n+1},\tilde q^{n+1}}(Y^{n+1}, \tilde Y^{n+1}) \ge \omega_{f^\ast} \norm{Y^{n+1} - \tilde Y^{n+1}}^2.
    \end{gather*}
    The inequality in \cref{lem:primal-dual-ineq} now implies
    \begin{align*}
        & \norm{X^{n+1}-\tilde X^{n+1}}_{\tau^{-1}+2\omega_g}^2 +\norm{X^{n+1}-\tilde X^{n+1}-X^n+\tilde X^n}_{\tau^{-1}}^2 + \norm{Y^{n+1}-\tilde Y^{n+1}}_{\sigma^{-1}+2\omega_{f^\ast}}^2 \\
        &\quad + 2\inpro{K(X^{n+1}-\tilde X^{n+1}-X^n+\tilde X^n)}{Y^{n+1}-\tilde Y^{n+1}} \\
        &\le \norm{X^n-\tilde X^n}_{\tau^{-1}}^2 + \theta \norm{X^n-\tilde X^n-X^{n-1}+\tilde X^{n-1}}_{\tau^{-1}}^2 + \norm{Y^n - \tilde Y^n}_{\sigma^{-1}}^2\\
        &\quad + 2\theta \inpro{K(X^n-\tilde X^n - X^{n-1}+\tilde X^{n-1})}{Y^n-\tilde Y^n}.
    \end{align*}
    Define now
    \begin{align*}
        \Delta^n &\coloneqq \norm{X^n-\tilde X^n}_{\tau^{-1}+2\omega_g}^2 +\norm{X^n-\tilde X^n-X^{n-1}+\tilde X^{n-1}}_{\tau^{-1}}^2 + \norm{Y^n - \tilde Y^n}_{\sigma^{-1}+2\omega_{f^\ast}}^2\\
        &\quad + 2 \inpro{K(X^n-\tilde X^n - X^{n-1}+\tilde X^{n-1})}{Y^n-\tilde Y^n}.
    \end{align*}
    Then under the assumptions on $\theta$ given in \cref{thm:dt-contraction-strongly-convex}, it holds
    \begin{align*}
        \Delta^{n+1} &\le \norm{X^n-\tilde X^n}_{\tau^{-1}}^2 + \theta \norm{X^n-\tilde X^n-X^{n-1}+\tilde X^{n-1}}_{\tau^{-1}}^2 + \norm{Y^n - \tilde Y^n}_{\sigma^{-1}}^2\\
        &\quad + 2 \theta \inpro{K(X^n-\tilde X^n - X^{n-1}+\tilde X^{n-1})}{Y^n-\tilde Y^n}\\
        &\le \theta \Delta^n.
    \end{align*}
    Iterating this inequality over $n=0,\dots,N-1$ we obtain
    \begin{align*}
        \Delta^N &\le \theta^{N} \Delta^0.
    \end{align*}
    We use again Cauchy--Schwarz and Young's inequality on the inner product in $\Delta^N$ as well as the initial condition $X^{-1}=X^0$, $\tilde X^{-1}=X^0$ and obtain
    \begin{align*}
        &\norm{X^{N}-\tilde X^{N}}_{\tau^{-1}+2\omega_g}^2 + (1-\theta \tau \sigma L^2) \norm{Y^{N}-\tilde Y^{N}}_{\sigma^{-1}+2\omega_{f^\ast}}^2 \\
        &\quad \le \theta^{N}\left(\norm{X^{0}-\tilde X^{0}}_{\tau^{-1}+2\omega_g}^2 + \norm{Y^{0}-\tilde Y^{0}}_{\sigma^{-1}+2\omega_{f^\ast}}^2 \right).
    \end{align*}
    From here, we proceed as in the proof of \cref{thm:dt-stability-convex}: The left side can be bounded below by $(\tau^{-1}+2\omega_g) \calW_2^2(\mu^N,\tilde \mu^N)$, or in the joint variable by $(1-\theta\tau\sigma L^2) (\tau^{-1}+2\omega_g) \calT_{1,\tilde \lambda^{-1}}^2(\pi^N,\tilde \pi^N)$, after taking expectations. The transport distance on the right is obtained by taking the infimum over all initialization couplings, completing the proof.
\end{proof}

\begin{proof}[Proof of \cref{cor:stability-modified-algorithm}]
   The central inequality in \cref{lem:primal-dual-ineq} still holds for this modified scheme. Under the changed iteration, we have
    \begin{equation*}
         \begin{array}{l}
            p^{n+1} \coloneqq (X^n-X^{n+1}+\sqrt{2\tau}\xi^{n+1})/\tau - K^T Y^{n+1} \in \partial g(X^{n+1}).
        \end{array}
    \end{equation*}
    Note the difference to \eqref{eq:subgradients-ulpda}, where $p^{n+1}$ was a subgradient at the point $X^{n+1}-\sqrt{2\tau}\xi^{n+1}$ instead. The only change we need to make in the proof of \cref{lem:primal-dual-ineq} is that instead of the symmetric Bregman distance at the point $(X^{n+1}- \sqrt{2\tau}\xi^{n+1}, \tilde X^{n+1}- \sqrt{2\tau}\xi^{n+1})$, we now start with
    $$ D_{g,\text{sym}}^{p^{n+1},\tilde p^{n+1}}(X^{n+1}, \tilde X^{n+1}) = \inpro{U^{n+1}}{p^{n+1} - \tilde p^{n+1}}. $$
    All the following computations remain unchanged.
\end{proof}

\subsection{Proofs for \texorpdfstring{\cref{subsec:discrete-time:convergence}}{Section 4.2}}\label{subsec:appendix-discrete-time:convergence}

For the coupling \eqref{eq:coupling-ulpda-SDE}, we can prove an inequality similar to \cref{lem:primal-dual-ineq}, but with additional discretization errors.

\begin{lemma}\label{lem:primal-dual-ineq-discr-error}
    Let $(X^n,Y^n)$, $n \ge 0$, and $(\tilde X_t, \tilde Y_t)$, $t\ge 0$ be given by \eqref{eq:coupling-ulpda-SDE} with $\theta \sigma \tau L^2 \le 1$. The subgradients $p^n \in \partial g(X^n - \sqrt{2\tau}\xi^{n})$ and $ q^n \in \partial f^\ast (Y^n)$ at the discrete time iterates are given in \eqref{eq:subgradients-ulpda}. Fix now some $n \ge 0$, and let $t = n\tau$. Then we have the following inequality.
    \begin{align*}
        &2 \left(D_{g,\text{sym}}^{p^{n+1},p'}(X^{n+1}-\sqrt{2\tau}\xi^{n+1}, \tilde X_{t+\tau}-\sqrt{2\tau}\xi^{n+1}) + D_{f^\ast,\text{sym}}^{q^{n+1},\tilde q_{t+\tau}}(Y^{n+1}, \tilde Y_{t+\tau}) \right) \\
        &\quad + \norm{X^{n+1}-\tilde X_{t+\tau}}_{\tau^{-1}}^2 +\norm{X^{n+1}-\tilde X_{t+\tau}-X^n+\tilde X_t}_{\tau^{-1}}^2 + \norm{Y^{n+1}-\tilde Y_{t+\tau}}_{\sigma^{-1}}^2 \\
        &\quad + 2\inpro{K(X^{n+1}-\tilde X_{t+\tau}-X^n+\tilde X_t)}{Y^{n+1}-\tilde Y_{t+\tau}} \\
        &\le \norm{X^n-\tilde X_t}_{\tau^{-1}}^2 + \theta \norm{X^n-\tilde X_t-X^{n-1}+\tilde X_{t-\tau}}_{\tau^{-1}}^2 + \norm{Y^n - \tilde Y_t}_{\sigma^{-1}}^2\\
        &\quad + 2\theta \inpro{K(X^n-\tilde X_t - X^{n-1}+\tilde X_{t-\tau})}{Y^n-\tilde Y_t}\\
        &\quad +\frac{2}{\tau} \inpro{X^{n+1} - \tilde X_{t+\tau}}{\int_t^{t+\tau} (\tilde p_s - p' + K^T(\tilde Y_s - \tilde Y_{t+\tau})) \rmd s} \\
        &\quad +\frac{2}{\tau} \inpro{Y^{n+1} - \tilde Y_{t+\tau}}{\int_t^{t+\tau} (\tilde q_s - \tilde q_{t+\tau} - K(\tilde X_s - \tilde X_{t,\theta})) \rmd s},
    \end{align*}
    where $p' \in \partial g(\tilde X_{t+\tau} - \sqrt{2\tau}\xi^{n+1})$ arbitrary, and $\tilde p_s \coloneqq \nabla g(\tilde X_s)$, $\tilde q_s \coloneqq \nabla f^\ast(\tilde Y_s)$ for $s \ge 0$.
\end{lemma}

\begin{proof}
    As before, we introduce the abbreviations $U^n := X^n-\tilde X_t$, $U_\theta^n = X_\theta^n-\tilde X_{t,\theta}$ and $V^n := Y^n - \tilde Y_t$.\\
   Integrating the SDE for the continuous time primal variable $\tilde X$ over $[t,t+\tau]$, we obtain the following identity
    \begin{gather}
        \tilde X_{t+\tau} = \tilde X_t - \int_{t}^{t+\tau} (\tilde p_s + K^T \tilde Y_s) \rmd s + \sqrt{2\tau} \xi^{n+1} \notag \\
        \Leftrightarrow \frac{1}{\tau} \int_{t}^{t+\tau} \tilde p_s \rmd s = (\tilde X_t - \tilde X_{t+\tau} + \sqrt{2\tau} \xi^{n+1})/\tau - \frac 1 \tau \int_{t}^{t+\tau} K^T \tilde Y_s \rmd s.\label{eq:integrated-primal-SDE-reformulated}
    \end{gather}
    As in the proof of \cref{lem:primal-dual-ineq}, we now rewrite the symmetric Bregman distance of $g$ at the point $(X^{n+1} - \sqrt{2\tau}\xi^{n+1},\tilde X_{t+\tau} - \sqrt{2\tau}\xi^{n+1})$. As before, we insert the representation of the subgradient $p^{n+1}$ of $g$ at $X^{n+1} - \sqrt{2\tau}\xi^{n+1}$. For the continuous time term, we add zero by adding and subtracting the two sides of \eqref{eq:integrated-primal-SDE-reformulated} in order to separate the discretization error in the drift:
    \begin{align*}
        &2D_{g,\text{sym}}^{p^{n+1}, p'}(X^{n+1}- \sqrt{2\tau}\xi^{n+1}, \tilde X_{t+\tau} - \sqrt{2\tau}\xi^{n+1}) \\
        &= 2\inpro{U^{n+1}}{p^{n+1} - p'} \\
        &= 2\inpro{U^{n+1}}{(X^n-X^{n+1})/\tau-K^T Y^{n+1} - (\tilde X_t - \tilde X_{t+\tau})/\tau + \frac{1}{\tau}\int_t^{t+\tau} K^T \tilde Y_s \rmd s + \frac 1 \tau \int_{t}^{t+\tau} (\tilde p_s - p')\rmd s } \\
        &= \frac{2}{\tau} \inpro{U^{n+1}}{U^n-U^{n+1}} - 2\inpro{U^{n+1}}{K^T V^{n+1}} + \frac{2}{\tau} \inpro{U^{n+1}}{\int_{t}^{t+\tau} (\tilde p_s - p' + K^T(\tilde Y_s - \tilde Y_{t+\tau}))\rmd s }\\
        &= \norm{U^n}_{\tau^{-1}}^2 - \norm{U^{n+1}}_{\tau^{-1}}^2 - \norm{U^{n+1} - U^n}_{\tau^{-1}}^2 - 2\inpro{KU^{n+1}}{V^{n+1}} \\
        &\quad + \frac{2}{\tau} \inpro{U^{n+1}}{\int_{t}^{t+\tau} (\tilde p_s - p' + K^T(\tilde Y_s - \tilde Y_{t+\tau}))\rmd s }.
    \end{align*}
    Similarly, for the dual variable in continuous time, we obtain the integrated SDE
    \begin{gather}
        \tilde Y_{t+\tau} = \tilde Y_t - \lambda \int_{t}^{t+\tau} (\tilde q_s - K \tilde X_s) \rmd s \notag \\
        \Leftrightarrow\quad \frac{1}{\tau} \int_{t}^{t+\tau} \tilde q_s \rmd s = (\tilde Y_t - \tilde Y_{t+\tau})/\sigma + \frac 1 \tau \int_{t}^{t+\tau} K \tilde X_s \rmd s. \label{eq:integrated-dual-SDE-reformulated}
    \end{gather}
    By plugging in $q^{n+1}$ from \eqref{eq:subgradients-ulpda} and adding and subtracting the sides of \eqref{eq:integrated-dual-SDE-reformulated}, we can rewrite the symmetric Bregman distance of $f^\ast$ at $(Y^{n+1},\tilde Y_{t+\tau})$ as
    \begin{align*}
        &2\inpro{Y^{n+1} - \tilde Y_{t+\tau}}{q^{n+1} - \tilde q_{t+\tau}} \\
        &= 2\inpro{V^{n+1}}{(Y^n-Y^{n+1})/\sigma + K X^n_\theta - (\tilde Y_t - \tilde Y_{t+\tau})/\sigma - \frac{1}{\tau}\int_t^{t+\tau} K \tilde X_s \rmd s + \frac{1}{\tau}\int_{t}^{t+\tau} (\tilde q_s - \tilde q_{t+\tau})\rmd s }\\
        &= \frac{2}{\sigma} \inpro{V^{n+1}}{V^n-V^{n+1}} + 2\inpro{V^{n+1}}{KU^n_\theta} + \frac{2}{\tau} \inpro{V^{n+1}}{\int_{t}^{t+\tau} (\tilde q_s - \tilde q_{t+\tau} - K(\tilde X_s - \tilde X_{t,\theta}))\rmd s }\\
        &= \norm{V^n}_{\sigma^{-1}}^2 - \norm{V^{n+1}}_{\sigma^{-1}}^2 - \norm{V^{n+1}- V^n}_{\sigma^{-1}}^2  + 2\inpro{V^{n+1}}{KU^n_\theta} \\
        &\quad + \frac{2}{\tau} \inpro{V^{n+1}}{\int_{t}^{t+\tau} (\tilde q_s - \tilde q_{t+\tau} - K(\tilde X_s - \tilde X_{t,\theta}))\rmd s }.
    \end{align*}
    Adding the primal and dual Bregman distances, we obtain
    \begin{align*}
        2&\inpro{U^{n+1}}{p^{n+1}-p'} + 2\inpro{V^{n+1}}{q^{n+1}-\tilde q_{t+\tau}} \\
        &= \norm{U^n}_{\tau^{-1}}^2 - \norm{U^{n+1}}_{\tau^{-1}}^2 - \norm{U^{n+1} - U^n}_{\tau^{-1}}^2 - 2\inpro{KU^{n+1}}{V^{n+1}} \\
        &\quad + \norm{V^n}_{\sigma^{-1}}^2 - \norm{V^{n+1}}_{\sigma^{-1}}^2 - \norm{V^{n+1} - V^n}_{\sigma^{-1}}^2 + 2\inpro{KU_\theta^{n}}{V^{n+1}} \\
        &\quad + \frac{2}{\tau} \inpro{U^{n+1}}{\int_{t}^{t+\tau} (\tilde p_s - p' + K^T(\tilde Y_s - \tilde Y_{t+\tau}))\rmd s }\\
        &\quad + \frac{2}{\tau} \inpro{V^{n+1}}{\int_{t}^{t+\tau} (\tilde q_s - \tilde q_{t+\tau} - K(\tilde X_s - \tilde X_{t,\theta}))\rmd s }.
    \end{align*}
    As in the proof of \cref{lem:primal-dual-ineq}, the argument is complete by plugging in inequality \eqref{eq:inequality-bilinear-terms}.
\end{proof}

\begin{lemma}\label{lem:discretization-error-bound}
    Assume that $g \in C^1(\calX)$ and $\nabla g$ is Lipschitz continuous with constant $L_{\nabla g} > 0$, and that $f$ is $\omega_f$-strongly convex with $\omega_f > 0$.
    Assume further that $g,f^\ast$ are $\omega_g$- and $\omega_{f^\ast}$-strongly convex, respectively, and let $D_4 = \bbE_{Z \sim \pi_\lambda} \left[ \norm{A_{\tup{PD}}(Z)}^2 \right] < \infty$ be the constant from \cref{thm:convergence-to-pi-lambda}.
    Let $Z_0 \sim \pi_\lambda$ and $Z_t = (X_t,Y_t), t>0$ satisfy the SDE \eqref{eq:ct-pd-diffusion-jointvar}, with $p_t = \nabla g(X_t)$ and $q_t = \nabla f^\ast(Y_t)$.
    For some $\tau > 0$ and $n\in \bbN_0$, fix $t=n\tau$. Let further $p' \in \partial g(X_{t+\tau} - \sqrt{2\tau}\xi)$ arbitrary, where $\xi = \tau^{-1/2} \int_t^{t+\tau} \mathrm d W_s$, and denote $X_{t,\theta}$ as before.
    Then we have the following bounds
    \begin{gather*}
        \bbE \left[\bigg\lVert \int_t^{t+\tau} (p_s - p' + K^T( Y_s - Y_{t+\tau})) \rmd s \bigg\rVert^2 \right] \le 4L_{\nabla g}^2 d \tau^3 + \calO(\tau^4), \\
        \bbE \left[\bigg\lVert \int_t^{t+\tau} (q_s - q_{t+\tau} - K(X_s - X_{t,\theta})) \rmd s \bigg\rVert^2 \right] \le 6 L^2 (1+2\theta^2) d \tau^3 + \calO(\tau^4).
    \end{gather*}
\end{lemma}

\begin{proof}
    We start with the first inequality. Applying first the Cauchy--Schwarz inequality and then the definitions of matrix norm and Lipschitz continuity, we can estimate
    \begin{align}
        &\bbE \left[\bigg\lVert \int_t^{t+\tau} (p_s - p' + K^T( Y_s -  Y_{t+\tau})) \rmd s \bigg\rVert^2 \right] \notag \\
        &\le 2\tau \int_t^{t+\tau} \bbE \left[\lVert p_s - p'\rVert^2\right] \rmd s +  2\tau \int_t^{t+\tau} \bbE \left[\lVert K^T( Y_s -  Y_{t+\tau}) \rVert^2\right] \rmd s \notag \\
        &\le 2\tau L_{\nabla g}^2 \int_t^{t+\tau} \bbE \left[\lVert X_s - X_{t+\tau} + \sqrt{2\tau}\xi \rVert^2\right] \rmd s +  2\tau L^2 \int_t^{t+\tau} \bbE \left[\lVert Y_s -  Y_{t+\tau} \rVert^2\right] \rmd s \label{eq:lemma-discretization-error-estimate1}
    \end{align}
    where $L_{\nabla g}$ is the Lipschitz constant of $\nabla g$ and $L = \norm{K}$. We now plug in the SDE, integrated over $[s,t+\tau]$, for both $X_{t+\tau}$ and $Y_{t+\tau}$ and obtain
    \begin{align*}
        X_s - X_{t+\tau} + \sqrt{2\tau}\xi &= \int_{s}^{t+\tau} (p_r + K^T Y_r) \rmd r + \sqrt{2} \int_t^{s} \mathrm d W_r\\
        Y_s - Y_{t+\tau} &= \lambda \int_{s}^{t+\tau} (q_r - K X_r) \rmd r.
    \end{align*}
    The continuous time process is stationary in distribution at $(X_r,Y_r) \sim \pi_\lambda$ for all $r>0$. Hence, by definition of $D_4$ we have $\bbE \left[ \norm{p_r + K^TY_r}^2 \right] \le D_4$ and $\lambda^2\bbE \left[ \norm{q_r - KX_r}^2 \right] \le D_4$. We can therefore estimate for all $t\le s \le t+\tau$:
    \begin{align}
        \bbE \left[\lVert X_s - X_{t+\tau} + \sqrt{2\tau}\xi \rVert^2\right] &\le 2 \bbE \left[ \bigg\lVert \int_{s}^{t+\tau} (p_r + K^T Y_r) \rmd r \bigg\rVert^2 \right] + 4 \bbE \left[ \bigg\lVert \int_t^s \mathrm d W_r \bigg\rVert^2 \right] \notag \\
        &\le 2 (t+\tau-s) \int_{s}^{t+\tau} \bbE \left[\norm{p_r + K^T Y_r}^2 \right] \rmd r + 4 d(s-t) \notag \\
        &\le 2 D_4 (t+\tau-s)^2 + 4 d(s-t) \notag \\
        \bbE \left[\lVert Y_s -  Y_{t+\tau} \rVert^2\right] &=  \bbE \left[ \bigg\lVert \lambda \int_{s}^{t+\tau} (q_r - KX_r) \rmd r \bigg\rVert^2 \right] \notag \\
        &\le D_4 (t+\tau-s)^2. \label{eq:lemma-discretization-error-estimate3}
    \end{align}
    Plugging these two bounds into \eqref{eq:lemma-discretization-error-estimate1} and integrating gives the desired bound
    \begin{equation*}
        \bbE \left[\bigg\lVert \int_t^{t+\tau} (p_s - p' + K^T( Y_s -  Y_{t+\tau})) \rmd s \bigg\rVert^2 \right] \le 4 L_{\nabla g}^2 d \tau^3 + \frac{2D_4 (2 L_{\nabla g}^2 + L^2)}{3} \tau^4.
    \end{equation*}
    The bound for the discretization error in dual space can be derived using similar estimates
    \begin{align}
        &\bbE \left[\bigg\lVert \int_t^{t+\tau} (q_s - q_{t+\tau} - K(X_s - X_{t,\theta})) \rmd s \bigg\rVert^2 \right] \notag \\
        &\le 2\tau \int_t^{t+\tau} \bbE \left[\lVert q_s - q_{t+\tau} \rVert^2\right] \rmd s + 2\tau \int_t^{t+\tau} \bbE \left[\lVert K(X_s - X_{t,\theta}) \rVert^2\right] \rmd s \notag \\
        &\le \frac{2\tau }{\omega_f^2} \int_t^{t+\tau} \bbE \left[\lVert Y_s - Y_{t+\tau} \rVert^2\right] \rmd s + 2\tau L^2 \int_t^{t+\tau} \bbE \left[\lVert X_s - X_{t,\theta} \rVert^2\right] \rmd s. \label{eq:lemma-discretization-error-estimate2}
    \end{align}
    If $n\ge 1$, the term in the second integral can be bounded using
    \begin{align*}
        X_s - X_{t,\theta} &= X_s - X_t - \theta (X_t - X_{t-\tau}) \\
        &= - \int_{t}^{s} (p_r + K^T Y_r) \rmd r + \sqrt{2} \int_{t}^s \mathrm d W_r + \theta \int_{t-\tau}^{t} (p_r + K^T Y_r) \rmd r - \sqrt{2} \theta \int_{t-\tau}^t \mathrm d W_r.
    \end{align*}
    In the boundary case $n=0$ implying $t = 0$, we have $X_{t-\tau} = X_t = X_0$ and the last two integrals can be dropped. In either case, we have the following bound
    \begin{align*}
        \bbE \left[\lVert X_s - X_{t,\theta} \rVert^2\right] &\le 3 \bbE \left[ \bigg\lVert \int_{t}^{s} (p_r + K^T Y_r) \rmd r \bigg\rVert^2 \right] + 3 \theta^2 \bbE \left[ \bigg\lVert \int_{t-\tau}^{t} (p_r + K^T Y_r) \rmd r \bigg\rVert^2 \right] \\
        &\qquad + 6 \bbE \left[ \bigg\lVert \int_t^s \mathrm d W_r - \theta \int_{t-\tau}^t \mathrm d W_r  \bigg\rVert^2 \right]\\
        &\le 3 D_4 ((s-t)^2 + \theta^2 \tau^2) + 6 d ((s-t)+\theta^2 \tau).
    \end{align*}
    Inserting the last inequality and \eqref{eq:lemma-discretization-error-estimate3} into \eqref{eq:lemma-discretization-error-estimate2} gives the error bound
    \begin{equation*}
        \bbE \left[\bigg\lVert \int_t^{t+\tau} (q_s - q_{t+\tau} - K(X_s - X_{t,\theta})) \rmd s \bigg\rVert^2 \right] \le 6 L^2 d (1+2\theta^2) \tau^3 + \left(2 D_4 L^2 (1 + 3\theta^2) + 2 D_4 \omega_f^{-2}\right) \tau^4.
    \end{equation*}
\end{proof}

\begin{proof}[Proof of \cref{thm:convergence-to-pi-lambda}]
Let $(X^n,Y^n)$ be coupled with a process $\tilde Z_t = (\tilde X_t, \tilde Y_t)$, $t\ge 0$ by means of \eqref{eq:coupling-ulpda-SDE}. Let $n \in \bbN_0$ be fixed and $t = n\tau$. Then
\begin{align*}
    \frac{2}{\tau} &\inpro{X^{n+1} - \tilde X_{t+\tau}}{\int_t^{t+\tau} (\tilde p_s - p' + K^T(\tilde Y_s - \tilde Y_{t+\tau})) \rmd s}\\
    &\le \omega_g \norm{X^{n+1} - \tilde X_{t+\tau}}^2 + \frac{1}{\omega_g \tau^2} \bigg\lVert \int_t^{t+\tau} (\tilde p_s - p' + K^T(\tilde Y_s - \tilde Y_{t+\tau})) \rmd s \bigg\rVert^2,
\end{align*}
as well as
\begin{align*}
    \frac{2}{\tau} &\inpro{Y^{n+1} - \tilde Y_{t+\tau}}{\int_t^{t+\tau} (\tilde q_s - \tilde q_{t+\tau} - K(\tilde X_s - \tilde X_{t,\theta})) \rmd s}\\
    &\le \omega_{f^\ast} \norm{Y^{n+1} - \tilde Y_{t+\tau}}^2 + \frac{1}{\omega_{f^\ast}\tau^2} \bigg\lVert \int_t^{t+\tau} (\tilde q_s - \tilde q_{t+\tau} - K(\tilde X_s - \tilde X_{t,\theta})) \rmd s \bigg\rVert^2.
\end{align*}
We now employ the auxiliary inequality \cref{lem:primal-dual-ineq-discr-error}. Inserting the quadratic lower bounds for the symmetric Bregman distances and the two bounds above, we obtain the inequality
\begin{align*}
    &\norm{X^{n+1}-\tilde X_{t+\tau}}_{\tau^{-1}+\omega_g}^2 +\norm{X^{n+1}-\tilde X_{t+\tau}-X^n+\tilde X_t}_{\tau^{-1}}^2 + \norm{Y^{n+1}-\tilde Y_{t+\tau}}_{\sigma^{-1}+\omega_{f^\ast}}^2 \\
    &\quad + 2\inpro{K(X^{n+1}-\tilde X_{t+\tau}-X^n+\tilde X_t)}{Y^{n+1}-\tilde Y_{t+\tau}} \\
    &\le \norm{X^n-\tilde X_t}_{\tau^{-1}}^2 + \theta \norm{X^n-\tilde X_t-X^{n-1}+\tilde X_{t-\tau}}_{\tau^{-1}}^2 + \norm{Y^n - \tilde Y_t}_{\sigma^{-1}}^2\\
    &\quad + 2\theta \inpro{K(X^n-\tilde X_t - X^{n-1}+\tilde X_{t-\tau})}{Y^n-\tilde Y_t}\\
    &\quad + \underbrace{\frac{1}{\omega_g \tau^2} \bigg\lVert \int_t^{t+\tau} (\tilde p_s - p' + K^T(\tilde Y_s - \tilde Y_{t+\tau})) \rmd s \bigg\rVert^2}_{\eqqcolon R_n^{(1)}} + \underbrace{\frac{1}{\omega_{f^\ast} \tau^2} \bigg\lVert \int_t^{t+\tau} (\tilde q_s - \tilde q_{t+\tau} - K(\tilde X_s - \tilde X_{t,\theta})) \rmd s \bigg\rVert^2}_{\eqqcolon R_n^{(2)}}.
\end{align*}
We now define
\begin{align*}
    \Delta^n &\coloneqq \norm{X^n-\tilde X_t}_{\tau^{-1}+\omega_g}^2 +\norm{X^n-\tilde X_t-X^{n-1}+\tilde X_{t-\tau}}_{\tau^{-1}}^2 + \norm{Y^n - \tilde Y_t}_{\sigma^{-1}+\omega_{f^\ast}}^2\\
    &\quad + 2 \inpro{K(X^n-\tilde X_t - X^{n-1}+\tilde X_{t-\tau})}{Y^n-\tilde Y_t}.
\end{align*}
Given the assumptions on $\theta$ and the above inequality, it holds
\begin{align*}
    \Delta^{n+1} &\le \norm{X^n-\tilde X_t}_{\tau^{-1}}^2 + \theta \norm{X^n-\tilde X_t-X^{n-1}+\tilde X_{t-\tau}}_{\tau^{-1}}^2 + \norm{Y^n - \tilde Y_t}_{\sigma^{-1}}^2 \\
    &\quad + 2 \theta \inpro{K(X^n-\tilde X_t - X^{n-1}+\tilde X_{t-\tau})}{Y^n-\tilde Y_t} + R_n^{(1)} + R_n^{(2)} \\
    &\le \theta \Delta^n + R_n^{(1)} + R_n^{(2)}.
\end{align*}
Iterating this inequality over $n=0,\dots,N-1$ we obtain
\begin{align*}
    \Delta^N \le \theta^N \Delta^0 + \sum_{n=0}^{N-1} \theta^n (R_{N-n-1}^{(1)} + R_{N-n-1}^{(2)}).
\end{align*}
Denote $T = N\tau$. Bounding the remaining inner product in $\Delta^N$ from below as in the proof of \cref{thm:dt-contraction-strongly-convex}, now using $\tau L^2 = \frac{\tau\sigma L^2}{1+\omega_{f^\ast}\sigma}(\sigma^{-1}+\omega_{f^\ast}) \le \theta\tau\sigma L^2(\sigma^{-1}+\omega_{f^\ast})$, which holds since $\theta \ge \theta_{\min} \ge (1+\omega_{f^\ast}\sigma)^{-1}$, this implies
\begin{align*}
    &\norm{X^{N}-\tilde X_T}_{\tau^{-1}+\omega_g}^2 + (1-\theta \tau \sigma L^2) \norm{Y^{N}-\tilde Y_T}_{\sigma^{-1}+\omega_{f^\ast}}^2 \\
    &\quad \le \theta^{N}\left(\norm{X^{0}-\tilde X_0}_{\tau^{-1}+\omega_g}^2 + \norm{Y^{0}-\tilde Y_0}_{\sigma^{-1}+\omega_{f^\ast}}^2 \right) + \sum_{n=0}^{N-1} \theta^n (R_{N-n-1}^{(1)} + R_{N-n-1}^{(2)}).
\end{align*}
We now take expectations on both sides. The left side can be lower bounded by either $(\tau^{-1}+ \omega_g) \calW_2^2(\mu^N, \mu_\lambda)$ or, in the joint variable, by $(\tau^{-1}+\omega_g)(1-\theta\tau\sigma L^2) \calT^2_{1,\bar \lambda^{-1}}(\pi^N, \pi_\lambda)$. Since the initialization of the coupling was so far arbitrary (we only used that the marginal measures are $\pi^0$ and $\pi_\lambda$), we can choose the one that realizes the transport distance $\calT_{1,\bar \lambda^{-1}}$ and obtain
$$\bbE\left[\norm{X^{0}-\tilde X_0}_{\tau^{-1}+\omega_g}^2 + \norm{Y^{0}-\tilde Y_0}_{\sigma^{-1}+\omega_{f^\ast}}^2 \right] = (\tau^{-1}+\omega_g) \calT^2_{1,\bar \lambda^{-1}}(\pi^0, \pi_\lambda).$$
Furthermore, by \cref{lem:discretization-error-bound} we know that for all $n = 0,\dots,N-1$ it holds
\begin{equation*}
    \bbE \left[ R_{N-n-1}^{(1)} + R_{N-n-1}^{(2)} \right] \le \underbrace{(4 L_{\nabla g}^2 \omega_g^{-1} + 6 L^2 (1+2\theta^2)\omega_{f^\ast}^{-1})}_{\eqqcolon C_R} d \tau + \calO(\tau^2).
\end{equation*}
This means we can use $\sum_{n=0}^{N-1}\theta^n \le (1-\theta)^{-1}$ and after dividing by $ \tau^{-1} + \omega_g $ we obtain the two estimates
\begin{equation*}
    \calW_2^2(\mu^N, \mu_\lambda) \le \theta^N \calT^2_{1,\bar \lambda^{-1}}(\pi^0, \pi_\lambda) + \frac{1}{1-\theta}\left(C_R d\, \tau^2 + \calO(\tau^3)\right)
\end{equation*}
and
\begin{equation*}
    \calT_{1,\bar \lambda^{-1}}^2(\pi^N, \pi_\lambda) \le \frac{\theta^N}{1-\theta \sigma \tau L^2} \calT^2_{1,\bar \lambda^{-1}}(\pi^0, \pi_\lambda) + \frac{C_R d\, \tau^2 + \calO(\tau^3)}{(1-\theta \sigma \tau L^2)(1-\theta)},
\end{equation*}
where we used that $(\tau + \calO(\tau^2))/(\tau^{-1}+\omega_g) = \tau^2 + \calO(\tau^3)$ as $\tau \to 0$. The statement of the theorem hence follows by taking the square root and using $\sqrt{a+b} \le \sqrt{a} + \sqrt{b}$.
\end{proof}

\begin{proof}[Proof of \cref{cor:distances-modified-algorithm}]
    As in the proof of \cref{cor:stability-modified-algorithm}, we have to modify the proof of the auxiliary inequality \cref{lem:primal-dual-ineq-discr-error}. It holds $p^{n+1} \in \partial g(X^{n+1})$, and accordingly we now choose $p' \in \partial g(\tilde X_{t+\tau})$. Then, we start at the symmetric Bregman distance
    $$D_{g,\text{sym}}^{p^{n+1},p'}(X^{n+1}, \tilde X_{t+\tau}) = \inpro{U^{n+1}}{p^{n+1}-p'}.$$
    The subsequent computations in the proof of \cref{lem:primal-dual-ineq-discr-error} remain unchanged.

    In order to bound the discretization error, the statement of \cref{lem:discretization-error-bound} has to be changed to consider $p' \in \partial g(X_{t+\tau})$ (instead of $p' \in \partial g(X_{t+\tau}-\sqrt{2\tau}\xi)$). The error bound remains unchanged; the only difference in the proof appears in \eqref{eq:lemma-discretization-error-estimate1}, where the term under the integral is now not $\bbE \left[\lVert X_s - X_{t+\tau} + \sqrt{2\tau}\xi \rVert^2\right]$, but $\bbE \left[\lVert X_s - X_{t+\tau}\rVert^2\right]$. Carrying out the subsequent computations in the same way, we arrive again at the upper bound $4L_{\nabla g}^2 d \tau^3 + \calO(\tau^4)$ for the primal discretization error.
\end{proof}